\documentclass[12pt]{article}
\usepackage{graphicx}%
\usepackage{amsmath,amssymb,amsfonts}%
\usepackage{amsthm}%
\usepackage[title]{appendix}%
\usepackage{tikz}
\usepackage{hyperref}

\newtheorem{prop}{Proposition}[section]{\bfseries}{\itshape}
\newtheorem{theo}[prop]{Theorem}{\bfseries}{\itshape}
\newtheorem{coro}[prop]{Corollary}{\bfseries}{\itshape}
\newtheorem{lemm}[prop]{Lemma}{\bfseries}{\itshape}

\newtheorem{remk}[prop]{Remark}{\bfseries}{\itshape}

\newcommand{\Po}{{\cal P}}

\newcommand{\tT}{\tilde{T}}
\newcommand{\cQ}{{\cal Q}}
\newcommand{\DG}{{\Delta}}

\newcommand{\Q}{{\end{document}}}

\newcommand{\xxi}{{\zeta}}

\newcommand{\N}{\mathbb{N}}
\newcommand{\Z}{\mathbb{Z}}

\newcommand{\bH}{\mathbb{H}}

\newcommand{\R}{\mathbb{R}}
\newcommand{\E}{\mathbb{E}}
\newcommand{\X}{\mathcal{X}}
\renewcommand{\Pr}{\mathbb{P}}
\renewcommand{\emptyset}{\varnothing}

\newcommand{\bx}{{\bf x}}

\renewcommand{\E}{\mathbb E \,}

\newcommand{\tar}{P}

\newcommand{\cH}{{\cal H}}

\newcommand{\Y}{{\cal Y}}

\newcommand{\cU}{{\cal U}}
\newcommand{\tcU}{\tilde{\cal U}}
\newcommand{\NN}{{\cal N}}

\newcommand{\cF}{{\cal F}}

\newcommand{\tA}{W}

\newcommand{\tQ}{C}

\newcommand{\eps}{\varepsilon}
\newcommand{\edm}{\end{displaymath}}
\newcommand{\be}{\begin{equation}}
\newcommand{\ee}{\end{equation}}
\newcommand{\bea}{\begin{eqnarray}}
\newcommand{\eea}{\end{eqnarray}}
\newcommand{\bean}{\begin{eqnarray*}}
\newcommand{\eean}{\end{eqnarray*}}

\renewcommand{\epsilon}{\varepsilon}
\newcommand{\bbS}{\mathbb{S}}
\newcommand{\Sn}{{\mathsf{S}}_n}
\newcommand{\Gu}{{\mathsf{Gu}}}
\newcommand{\dist}{\,{\rm dist}}
\newcommand{\lglg}{\log \log}
\newcommand{\blue}{black}

\begin{document}

\title{Random coverage from within with variable radii, and Johnson-Mehl cover times\footnotetext[0]{Research supported by Engineering and Physical Sciences Research Council grant EP/T028653/1.}}

\author{ Mathew D. Penrose\thanks{
		%
		%
			Department of
			Mathematical Sciences, University of Bath, Bath BA2 7AY, United
			Kingdom. 
			\texttt{m.d.penrose@bath.ac.uk}
		}
		\and
		Frankie Higgs\thanks{Department of
			Mathematical Sciences, University of Bath, Bath BA2 7AY, United
			Kingdom. \texttt{fh350@bath.ac.uk}
			}
	}



\maketitle

\begin{abstract}
Given a compact planar region $A$, let $\tau_A$ be  the (random) time it takes for the Johnson-Mehl tessellation of $A$ to be complete, i.e. the
	time for
	$A$ to be fully covered by a spatial birth-growth process in $A$ with seeds arriving as a unit-intensity Poisson point process in $A \times [0,\infty)$, where upon arrival each seed grows at unit rate in all directions. We show that if $\partial A$ is smooth or polygonal then $\Pr [ \pi \tau_{sA}^3 - 6 \log s - 4 \log \log s \leq x]$ tends to $\exp(- 
	(\frac{81}{4\pi})^{1/3} |A|e^{-x/3} 
	\linebreak -
	(\frac{9}{2\pi^2})^{1/3} |\partial A| e^{-x/6})$ in the large-$s$ limit; the second term in the exponent is due to boundary effects, the importance of which was not recognized in earlier work on this model.  We present similar results in higher dimensions (where boundary effects dominate). These results are derived using new results on the asymptotic probability of covering $A$ with a high-intensity spherical Poisson Boolean model \emph{restricted to $A$} with grains having iid small random radii, which generalize recent work of the first author that dealt only with grains of deterministic radius.
\end{abstract}


\section{Introduction}
\label{SecIntro}

The {\em Johnson-Mehl (J-M) tessellation} is a classic model of a
random tessellation in $\R^d$, where $d \in \N$.
Seeds are generated as a homogeneous Poisson point
process $\cH_\rho = \{(x_i,t_i)\}_{i \geq 1}$ 
of intensity $\rho$
in space-time
$\R^d \times [0,\infty)$.
If location $x_i$ is not already claimed by time $t_i$, the
seed $i$ becomes a cell at that instant, which immediately starts to grow 
from $x_i$
at unit rate in all directions, claiming previously unclaimed
territory as part of that cell. Whenever the growing cell hits another
cell, it stops growing in that direction.
Ultimately the whole of $\R^d$ is tessellated by cells.

This model dates back at least to work of Kolmogorov and
of Johnson and Mehl in the 1930s 
on modelling crystallization processes; other applications include 
growth of surface film on metals,  and more recent
applications in neurobiology are described in \cite{CSKM}.
For further discussion and references, see also \cite{Moller,OBSC}.

A natural variant is the {\em restricted}
Johnson-Mehl tessellation, which we define as follows. 
Given a 
specified compact region $A \subset \R^d$,
we allow only for seeds that arrive inside the
region $A$, that is, the seeds are generated by
a Poisson process in $A \times [0,\infty)$.
Cells grow by the same rules as described before, and the
restrictions of the cells to $A$
ultimately  tessellate $A$,
\textcolor{\blue}{as shown in Figure~\ref{f:tessellations}.}

We are interested here in the time at which $A$ becomes completely covered,
i.e. the first time at which every point of $A$ has been claimed by a 
cell, either for the J-M model
or for the restricted J-M model. In particular we are concerned with
the
limiting distribution
of the cover time, denoted $\tilde{T}_\rho$ for the J-M model and
$T_\rho$ for the restricted J-M model,
as $\rho \to \infty$.
Equivalently, one can keep $\rho$ fixed (say, $\rho=1$),
and consider the distribution of
the cover times for 
an expanding sequence of windows $(A_L)_{L>0}$ given by dilations
of a fixed  set $A$ by a factor of $L$, in
the large-$L$ limit. In
the later  case, the restricted Johnson-Mehl tessellation
uses only seeds that arrive inside $A_L $,
and we refer to the cover times in this limiting regime as
$\tau_L$ (for the restricted J-M model)
and $\tilde{\tau}_L$ (for the J-M model).

These cover times have previously been considered by S. N. Chiu
in
\cite{Chiu95}, for $A= [0,1]^d$. In fact, Chiu considers
a more general class of J-M models, where the rate at which
seeds arrive is allowed to be non-homogeneous in time,
and where moreover the seeds do not necessarily all grow
at the same rate.  We do not consider such generalizations here.

In \cite[page 893]{Chiu95}, 
Chiu rightly distinguishes between the J-M and restricted
J-M models.
%
However, he goes on to assert that
`all theorems in this paper are valid for both models' because
`almost surely
edge effects do not play a role in the limiting behaviour of
$T_L$' (Chiu's $T_L$ is  equivalent to our $\tau_L$).
In other words, Chiu seems to be asserting  that $\tau_L$
and $\tilde{\tau}_L$ have the same limiting distribution.

It is our contention that this assertion 
is incorrect. In dimensions $d \geq 2$, the limiting 
behaviour of $T_\rho$  is different from that of $\tT_\rho$,
and the limiting 
behaviour of $\tau_L$  is different from that of $\tilde{\tau}_L$;
in other words, edge effects \emph{do} play a role.
We justify this assertion with results identifying
the limiting
distribution of $T_\rho$ (suitably scaled and centred)
and showing  it is different from that of
$\tT_\rho$.
Likewise
we show that
the limiting
distribution of $\tau_L$ (suitably scaled and centred)
is different from that of
$\tilde{\tau}_L$; \textcolor{\blue}{our limiting result for $\tilde{\tau}_L$ is consistent with the result in \cite{Chiu95}}.
Our results for $T_\rho$ and $\tau_L$
apply when $A$ is polygonal
in $d=2$, or when it has a smooth (more precisely, $C^2$) boundary
for general $d \geq 2$.

When $A = [0,1]^d$,
we do not provide detailed
limiting distributions for $T_\rho$ and $\tau_L$ 
except in the case $d =2$.
To give a detailed limit distribution when $d \geq 3$ 
would require careful consideration
of all faces of all dimensions and is beyond the scope of
this paper.
However, the time to cover
all the $(d-1)$-dimensional faces will
be  a lower bound for the actual cover time 
so boundary effects can certainly not be neglected
in this case either.

We shall derive our results on the J-M model
from new results, of independent interest,  on the
{\em spherical Poisson Boolean model (SPBM)},  which is
defined to be a collection
of Euclidean balls (referred to as {\em grains}) of i.i.d. random radii,
centred on the points of a
homogeneous Poisson process in the whole of $\R^d$. 
For the the {\em restricted} SPBM, we take a 
Poisson process on $A$ rather  than on all of $\R^d$.

We  shall determine the probability that the restricted SPBM covers the 
whole of $A$, in the limit when the Poisson intensity becomes large and the
radii of balls become small in a linked manner.
In \cite{P23} we derived results of this nature for balls with
deterministic  radius; here we generalize them to allow for balls
of random radius, which is needed to derive our results on
the J-M cover time. Our results on coverage of $A$ by the restricted
SPBM complement the classic results of \cite{HallZW,Janson}
on the limiting probability of covering $A$ with an
{\em unrestricted} SPBM.

\textcolor{\blue}{
	The Johnson-Mehl cover time is the maximum
	of the random field $(\Xi_x, x \in A)$,
	where $\Xi_x$ denotes the  the time at which
at which $x$ is covered. For the restricted J-M model,
this maximum will be achieved at a vertex
$x$ of the restricted Johnson-Mehl tessellation
of $A$, as illustrated in Figure \ref{f:tessellations}.
	(In fact
there are other possibilities; e.g. if $A$ is polygonal with a sharp corner
then $\sup_{x \in A} \Xi_x$ could be achieved at a corner of $A$ but these
other possibilities have vanishing probability in our limit regimes,
at least when $d=2 $ or $d=3$.)
Thus it is the
maximum of a large finite random number of (perhaps weakly) dependent variables,
and one might perhaps expect one of the classical extreme value distributions
such as the Gumbel to arise in the limiting regimes that we consider. 
}

\textcolor{\blue}{
It turns out that the limit distribution for the restricted model
is indeed Gumbel in 3 or more dimensions, but
in two dimensions it is a {\em two-component extreme value distribution}
(see Remark \ref{rk:TCEV} below).
In short, this
is because for $d=2$, the maximum could be achieved either at a vertex
in the interior of $A$,
or at a vertex on the boundary of $A$,
and these different
possible contributions scale differently as $\rho \to \infty$;
in higher dimensions, the maximum is very likely to be at the boundary.}
%

\textcolor{\blue}{
For the restricted SPBM, again the probability of coverage 
can be framed 
as the extreme value of a random field.
	Label the Poisson points in $A$ as
$\{p_1,\ldots,p_N\}$ and let $Y_i$ be the random radius associated with
Poisson point $p_i$. 
Define the {\em coverage threshold} $R$ to be the
smallest $r$ such that the union of balls of radius $rY_i$ centred
on $p_i$ covers $A$.  Then for any scaling factor $r >0$,
the probability that that $A$ is fully covered by
balls of radius $rY_i$ equals the cdf $F_R(r):= \Pr[R \leq r]$.
The threshold is the maximum
of a random field, again denoted $(\Xi_x, x \in A)$, where now
$\Xi_x $ is the smallest $r$ such that $x \in \cup_i B(p_i,rY_i)$.
Moreover, once again for high intensities the maximum will typically
(at least in low dimensions) be achieved
at a vertex of a certain tessellation of $A$, namely the division
of $A$ into cells $C_i, 1 \leq i \leq N$ with 
$$
C_i := \big\{y \in A: \frac{\|y-p_i\|}{Y_i} \leq \frac{\|y- p_j \| }{Y_j} ~
\forall ~ j \in \{1,\ldots,N\} \big\},
$$
where $\|\cdot\|$ denotes the Euclidean norm on $\R^d$; here
		the cells might not be connected.
See Figure \ref{f:tessellations} for an illustration (which does indeed
include a disconnected cell) and Remark~\ref{r:videos}
for further discussion  (again the cells have a dynamical interpretation).
	In the special case where the $Y_i$ are a deterministic
	constant
	this is simply the
{\em Voronoi} tessellation of $A$ induced by the Poisson process in $A$.
In fact, in this case the coverage threshold is the largest circumscribed
radius of  the Poisson-Voronoi tessellation of $A$, as  discussed
in \cite{P23} for the restricted SPBM and in \cite{CC14} for the unrestricted
SPBM.
}

\begin{figure}[h]
 	\centering
 	\includegraphics[height=0.49\linewidth]{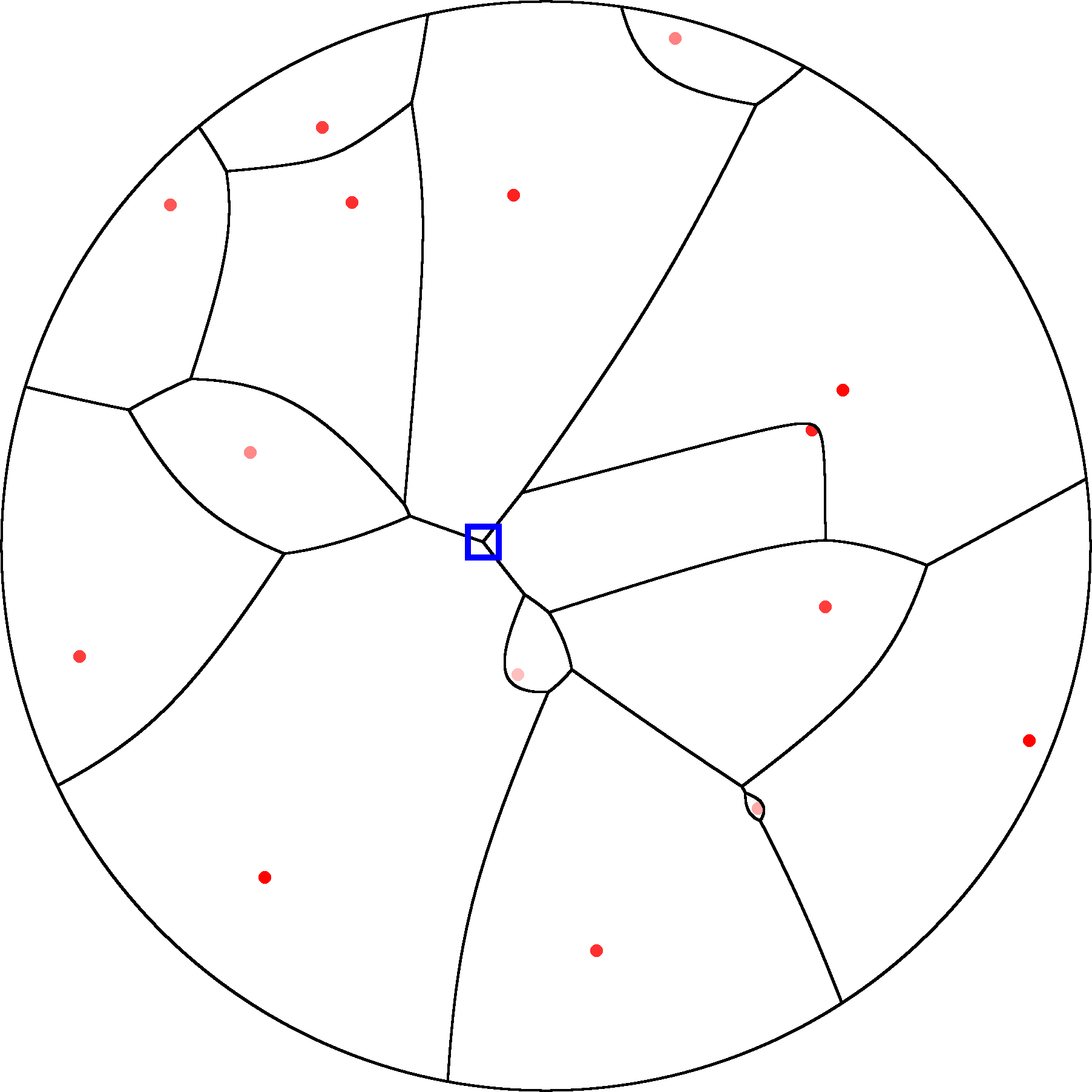}
 	\includegraphics[height=0.49\linewidth]{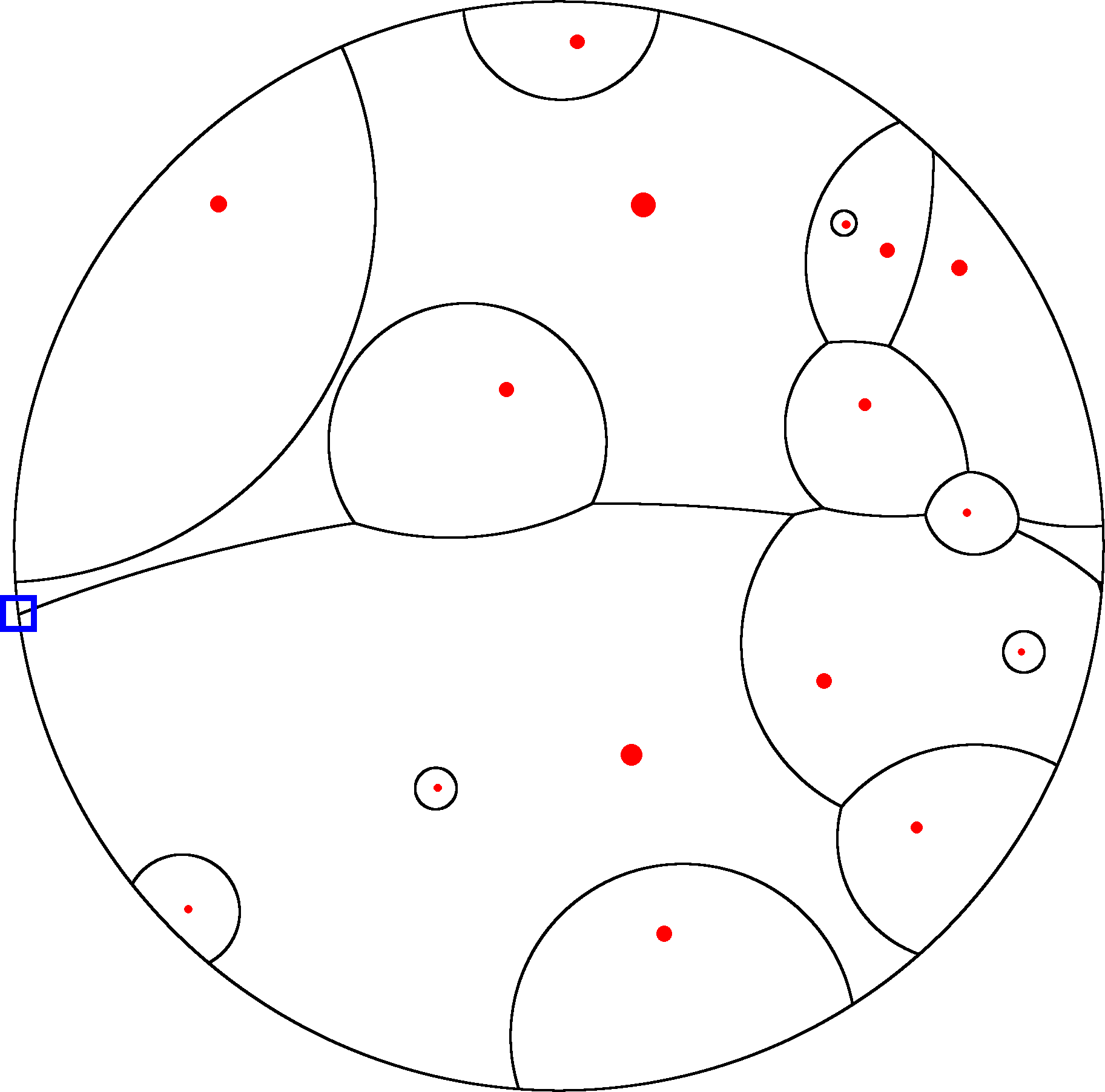}
 	\caption{\label{f:tessellations}
		Tessellations of a disc $A$ (diameter $0.9$)
		by the restricted Johnson--Mehl process
		with $\rho \approx 125$  
		(left)
 		and restricted spherical Poisson Boolean model
		with 
		$Y_i$ exponentially distributed (right).
 		In both cases the seeds are marked inside each cell with a red dot,
		and $\mathrm{argmax}_{x \in A} \Xi_x$
		(the ``last location covered'') is marked with a blue square.
		Later-arriving seeds 
		are marked with paler dots (left diagram);
		larger dots indicate larger $Y_i$ (right diagram).
	All vertices of both tessellations are of degree 3, including those
		on the boundary.
 		%
 	}
 \end{figure}

\textcolor{\blue}{
Thus in all cases we are interested in the distribution of
a random variable
(either the cover time or the coverage threshold) given by
the maximum of a certain geometrically-defined
random field  on $A$. 
Other somewhat related topics within this genre
of {\em geometric extreme value theory}  (as it might reasonably be termed)
include
multivariate scan statistics and clique number of random
geometric graphs (see e.g. \cite{Alm98}, \cite{P02}), and
largest nearest $k$-neighbour link ($k$-NNL) of a random sample of points
(see \cite{PY24} and references therein in a Euclidean space,
\cite{OT23} in hyperbolic space).
}


\section{Statement of results}
\label{secdefs}
\label{secweak}

Throughout this paper, we assume that
we are given $d \in \N$, and a compact,
Riemann measurable set
$A \subset \R^d$
(Riemann measurability of a bounded set
in $\R^d$ amounts to its boundary having zero Lebesgue measure).

For $x \in \R^d$ and $r>0$ set $B(x,r):= \{y \in \R^d:\|y-x\| \leq r\}$.
(We write $B_{(d)} (x,r)$ for this if we wish 
to emphasise the dimension.)
For $D \subset \R^d$, let $\overline{D}$ and $D^o$
denote the closure of $D$
and interior of $D$, respectively, and
set $\partial D:= \overline{D} \setminus D^o$, the
topological boundary of $D$. 
For $r>0$, let $D^{(r)}:=\{ x \in D: B(x,r) \subset D^o\}$, the
`$r$-interior' of $D$.
Let $|D|$ denote the Lebesgue  measure (volume) of $D$, and
$|\partial D|$ the perimeter of $D$, i.e. the
$(d-1)$-dimensional Hausdorff measure of
$\partial D$, when these are defined.
Write $\lglg t$ for $\log (\log t)$,
$t >1$. 

Define the set  $D^{[r]}$ to be
the interior of the union of all hypercubes of the
form \linebreak
$\prod_{i=1}^d[n_ir,(n_i+1)r] $, with $n_1,\ldots,n_d \in \Z$,
that are contained in $D^o$ (the set $D^{[r]}$ resembles
$D^{(r)}$ but is guaranteed
to be Riemann measurable).

We say that $D$ {\em has $C^2$ boundary} (for short: $\partial D \in
C^2$) if 
for each $x \in \partial D$
there exists a neighbourhood $U$ of $x$ and a real-valued function $f$ that
is  defined on an open set in $\R^{d-1}$
and twice continuously differentiable, such
that  $\partial D \cap U$, after a rotation, is the graph of the
function $f$. 
We say that $\partial D \in C^{1,1}$ (a weaker condition)
if for each $x$ the function
$f$ satisfies only that $f$ is continuously differentiable
with Lipschitz first order partial derivatives.

Given $d \in \N$, let $\omega_d := \pi^{d/2}/\Gamma(1 + d/2)$,
the volume of the unit ball in $\R^d$, and set
\bea
c_d := \frac{1}{d!} \left( \frac{\sqrt{\pi} \; \Gamma(1+ d/2) }{ \Gamma((d+1)/2) }
\right)^{d-1}.
\label{cdef}
\eea
Note that $c_1=c_2=1$, and  $c_3 =3 \pi^2/32$.
Moreover, using Stirling's formula one can show that
$c_d^{1/d}  \sim e (\pi/(2d))^{1/2}$ as $d \to \infty$.
The constant $  c_d$ is  denoted $ \psi_d$
in \cite[page 894]{Chiu95}, while $\omega_d$ is denoted
$\theta_d$ in \cite{P23}. Later it will be useful to have
$\omega_0$ defined: we set $\omega_0 :=1$.

\subsection{Results for the Johnson-Mehl  model}
\label{ss:JMresults}

Assume that on a common  probability space
$(\bbS,\cF,\Pr)$, we have a family of Poisson point
processes $\cH_\rho, \rho >0$. Here
$\cH_\rho$ is homogeneous of intensity $\rho$ on
$\R^d \times \R_+$, where
$\R_+ := [0,\infty)$.
We write $\cH_{\rho,A}$ for $\cH_\rho \cap (A \times \R_+)$.

Given 
$\rho \in (0,\infty)$,
define the coverage times $T_{\rho}$ and $\tilde{T}_{\rho}$ by
\bea
T_{\rho} :=
\inf \left\{ t >0: A \subset \cup_{(x,s) \in \cH_\rho \cap (A \times [0,t])}
B(x,t-s) \right\};
\label{e:T}
\\
\tT_{\rho} : = 
\inf \left\{ t >0: A \subset \cup_{(x,s) \in \cH_\rho \cap
	(\R^d \times [0,t])}
B(x,t-s)  \right\}.
\label{e:tT}
\eea
Also, for $L >0$ set $A_L := \{Lx:x \in A\}$ and let
\bea
\tau_{L} :=
\inf \left\{ t >0: A_L \subset \cup_{(x,s) \in \cH_1 \cap (A_L \times [0,t])}
B(x,t-s) \right\};
\label{e:tau}
\\
\tilde{\tau}_{L} : = 
\inf \left\{ t >0: A_L \subset \cup_{(x,s) \in \cH_1 \cap
	(\R^d \times [0,t])}
B(x,t-s)  \right\}.
\label{e:ttau}
\eea
%
Then
$T_\rho$, $\tT_\rho$ are the cover times of $A$ for the restricted
J-M model and for the J-M model, respectively, with 
intensity $\rho$.
Likewise
$\tau_L$, $\tilde{\tau}_L$, 
are the cover time of $A_L$ for the restricted
J-M model and for the J-M model, respectively, with 
intensity $1$.



We are concerned with the asymptotic  distribution
of the cover times $T_{\rho}$  
and $\tT_\rho$ as $\rho \to \infty$, and  
the asymptotic  distribution of $\tau_{L}$  
and $\tilde{\tau}_L$ as $L \to \infty$.
We start with a result in general $d$ for the unrestricted
cover time $\tT_{\rho}$ (and for $\tilde{\tau}_L)$.
This is simpler to deal with than $T_\rho$
because boundary effects are avoided. 
\begin{prop}
	\label{Hallthm}
	Suppose $A \subset \R^d $ is compact and Riemann measurable with $ |A| > 0$.
	Let $\beta \in \R$.  Then 
	\textcolor{\blue}{
	\begin{align}
		&	\lim_{\rho \to \infty} 
		\Pr \left[ \omega_d \rho \tT_{\rho}^{d+1}
		- d \log \rho   -
		d^2 \lglg \rho 
		\leq \beta \right] 
		\nonumber \\
		& = \exp(-c_d (d^d \omega_d)^{-1/(d+1)}  |A|e^{-\beta/(d+1)} )
		\label{1228a}
		\\
		& = 	\lim_{L \to \infty}
		\Pr \left[ \omega_d 
		\tilde{\tau}_L^{d+1}
		- d (d+1) \log L   -
		d^2 \lglg L
		- d^2\log (d+1) 
		\leq \beta \right]. 
		\label{e:ttaulim}
	\end{align}
	}
\end{prop}
In Appendix \ref{a:consist}, we shall verify
that this result is consistent with the
case $\gamma = v =1$ of \cite[Theorem 4]{Chiu95}.
The proof in
\cite{Chiu95} refers to an unpublished preprint; for 
completeness we shall present the
(fairly short) proof of Proposition \ref{Hallthm}
in Section \ref{s:pfJM}.

Our main new results in this subsection are concerned with
$T_\rho$ and $\tau_L$,  the cover times 
for the restricted  J-M model.
We first give a result for
the case where $A$ is a polygon in $d=2$.

\begin{theo}
	\label{thmwksq}
	Suppose that $d=2$ and $A$ is polygonal.
	Let $\beta \in \R$.
	Then  
		\textcolor{\blue}{
	\begin{align}
			& \lim_{\rho \to \infty} 
		\Pr[ \pi \rho 
		T_{\rho}^3 - 2 \log \rho  - 4   \lglg \rho 
		\leq \beta ]
		\nonumber \\
		& =
		\exp( - (4 \pi)^{-1/3} |A|  e^{- \beta/3}
		- (2 \pi^2)^{-1/3} |\partial A| e^{-\beta/6} )
		\label{0322b}
		\\
		&		= \lim_{L \to \infty}
		\Pr[ \pi \tau_L^3 - 6 \log L -4 \log \log L
		- \log 81
		\leq \beta].
		\label{e:taulim}
	\end{align}
	}
\end{theo}

Our result for general $d $ and $A$ with $C^2$ boundary
involves a constant $c'_{d}$ defined by
\begin{align}
	c'_d & := 
	\frac{c_{d-1}  \omega_{d-1}^{2d-3} 
	}{ \omega_{d-2}^{d-1} d^{d-1} 
	} 
	\Big( \frac{(d-1)^d}{2 \omega_d^d} \Big)^{(d-1)/(d+1)}
	.
	\label{e:cd'1}
\end{align}

\begin{theo}
	\label{thm3d}
	Suppose that $d \geq 2$ 
	and $\partial A \in C^2$ with $\overline{A^o} =A$.
	Let $\beta \in \R$.
	If $d=2$ then \eqref{0322b} 
	and \eqref{e:taulim} hold while
	if $d \geq 3$ then
	\textcolor{\blue}{
	\begin{align}
		& \lim_{\rho \to \infty} 
		\Pr[ \omega_d \rho 
		T_{\rho}^{d+1} - 2(d-1) \log \rho  - 2d(d-1)   \lglg \rho 
		\leq \beta ]
		\nonumber \\
		& =
		\exp( 
		- c'_d  |\partial A| e^{-\beta/(2d+2)} )
		\label{e:Tlim3d}
		\\
		& = \lim_{L \to \infty}
		\Pr[ \omega_d \tau_L^{d+1} - 2(d^2-1) \log L - 2d(d-1)
		\log \log L
		\nonumber \\
		&~~~~~~~~~~~~~~~~~~~~~~~~~~~~~~~~~~~~~ ~~~~~
		- 2 d(d-1) \log (d+1)
		\leq \beta].
		\label{e:taulim3d}
	\end{align}
	}
\end{theo}

\begin{figure}[h]
	\centering
	\includegraphics[width=1.0\linewidth, trim=11 14 11 12, clip]{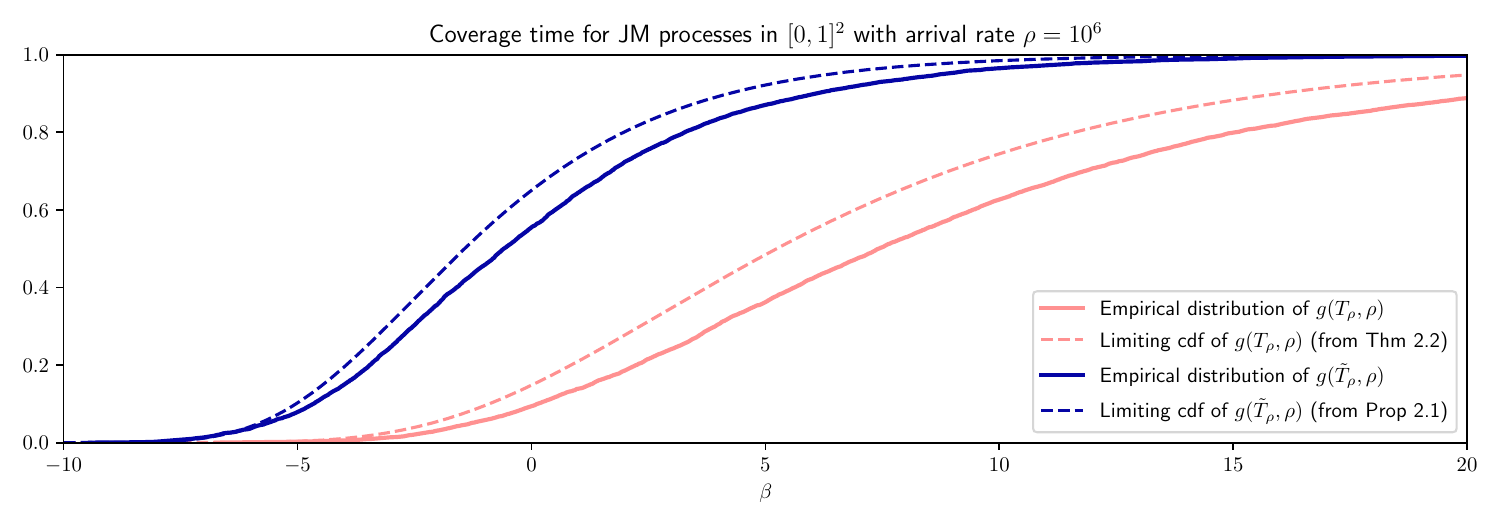} 
	\caption{\label{f:sim}
		Empirical distributions of the standardised coverage times $T_\rho$ and $\tilde T_{\rho}$
		as described in Proposition~\protect\ref{Hallthm}
		and Theorem~\protect\ref{thmwksq},
		obtained by sampling
		many independent realisations
		of the Johnson-Mehl process
		in $[0,1]^2$.
		The standardisation function
		$g(T,\rho) = \pi \rho T^3 - 2 \log \rho - 4 \log\log \rho$
		is the same for both $T_\rho$
		and $\tilde T_\rho$.}
\end{figure}


Let $\mathsf{Gu}$ denote a standard Gumbel random variable, i.e.
one with cdf  $\Pr[\Gu \leq \beta] = \exp(- e^{-\beta}),
\beta \in \R$. The limit in \eqref{1228a} is the cdf of the random
variable 
$$
(d+1) \Gu +
\log(c_d^{d+1}d^{-d}/ \omega_d) + (d+1) \log |A| 
,
$$
so Proposition \ref{Hallthm} says that $\omega_d \rho \tilde{T}_\rho^{d+1}$
suitably centred by a constant depending on $\rho$ and $d$ but not $A$,
and $\omega_d \tilde{\tau}_L^{d+1}$ suitably centred likewise,
both converge in distribution to the random variable $(d+1) (\Gu + \log |A|)$.

	Similarly, Theorem \ref{thm3d} 
says for $d \geq 3$ that $\omega_d \rho T_\rho^{d+1}$
suitably centred by a constant depending on $\rho$ and $d$ but not
 $A$,
converges in distribution to 
$(2d+2)( \Gu + \log |\partial A| )$, as does $\omega_d \tau_L^{d+1}$.

\textcolor{\blue}{
	Theorem \ref{thmwksq} 
	says for $d=2$ that
$\pi \rho T_\rho^3 - 2 \log \rho - 4 \log \log \rho$,
and $\pi \tau_L^3 - 6 \log L - 4 \log \log L - \log 81$,
both converge in distribution to the random variable
\begin{align}
\max ( 3 ( \Gu + \log |A|) - \log (4\pi)  , 
6 (\Gu' + \log |\partial A|) - \log
4 \pi^4),
	\label{e:TCEVlimit}
\end{align}
where $\Gu'$ is  another standard Gumbel variable, independent
of $\Gu$.}

\begin{remk}
\label{r:JMvideos}
{\rm
	A series of videos of very high intensity restricted Johnson-Mehl 
	processes inside $A = [0,1]^2$
	can be viewed \href{https://www.youtube.com/playlist?list=PLiaV5rk6Gk7rh6nvAts1w4WZrHE0TKeW2}{here}\footnote{\href{https://www.youtube.com/playlist?list=PLiaV5rk6Gk7rh6nvAts1w4WZrHE0TKeW2}{\texttt{https://www.youtube.com/playlist?list=PLiaV5rk6Gk7rh6nvAts1w4WZrHE0TKeW2}}}.
	The cells are coloured (with five different colours)
	so that no two adjacent cells in the final ``tessellation''
	have the same colour.
	To illustrate the above remark about which point is covered last when $d=2$,
	we have included one video in the playlist where the last point to be covered
	is at the boundary, and one where it is in the interior.

	An interactive tool for generating Johnson-Mehl tessellations is available at \href{https://frankiehiggs.pyscriptapps.com/johnson-mehl-plot/latest}{\texttt{https://frankiehiggs.pyscriptapps.com/johnson-mehl-plot/latest}}.\
	
	The results of more simulations, comparing the distribution of cover times for the
	restricted and unrestricted J-M models with finite $\rho$
	and their limiting distributions,
	are presented in Figure~\ref{f:sim}.
	%
	
	In \cite{Chiu95},  \textcolor{\blue}{it is}
	suggested that the same limit
	theorem should apply to $\tau_L$ as to $\tilde \tau_L$,
	as the boundary effects would not affect the limit.  This would
	imply that the same limit should apply to $T_\rho$ as to $\tilde{T}_\rho$. Our
	Theorem~\protect\ref{thmwksq} shows that in fact the
	boundary effects do affect the limiting distribution of $g(T_\rho,\rho)$,
	and the simulations shown in
	Figure~\ref{f:sim} back this up.
}
\end{remk}

\begin{remk} 
{\rm
	One could consider
	a slight variant of the restricted J-M model, in which a growing
	cell, whenever it hits the boundary of $A$,  stops 
	growing in that direction.
	Let  $T^*_\rho$ denotes the time at which the union of cells
	first covers $A$ for this variant of the model.
	
	When $A$ is convex, the value of $T^*_\rho$ is the same
	as that of $T_\rho$.
	If $A$ is not convex,
	we would expect our results to still be true for $T^*_\rho$
	as well as for $T_\rho$,
	but we do not prove this.
}
\end{remk} 
\begin{remk}  {\rm
	Suppose we take $d=2$ and $A= [0,1]^2$.
	Comparing \eqref{1228a} with \eqref{0322b}
	we see that
	the limiting distribution of $T_\rho^3$ 
	(appropriately scaled and centred) is different from that of
	$\tilde{T}_\rho^3$ (scaled and centred the same way).
	Similarly, comparing
	\eqref{e:ttaulim} with \eqref{e:taulim}
	we see that
	the limiting distribution of $\tau_L^3$ 
	(appropriately scaled and centred) is different from that of
	$\tilde{\tau}_L^3$ (scaled and centred the same way).
	
	When $d \geq 3$ and $\partial A \in C^2$,
	we see from \eqref{1228a} and \eqref{e:Tlim3d} that
	even the centring  constants required to get a nondegenerate
	limit distribution for $ \omega_d \rho T_\rho^{d+1}$
	are different (larger) compared to the corresponding constants 
	for  $ \omega_d \rho \tilde{T}_\rho^{d+1}$. Again
	similar remarks apply for $\tau_L$ and $\tilde{\tau}_L$.
}
\end{remk} 
\begin{remk} 
{\rm
	As mentioned earlier, in \cite{Chiu95} Chiu considers cover times 
	for more general J-M models where the growth rate of
	cells, or the arrival rates of seeds, are not necessarily constants.
	If growth rates are non-constant, then it is
	rather problematic
	to model the covered region at a given time as simply a SPBM; one has to
	modify our initial description of the  cell growth model to allow a
	faster-growing cell to pass through a slower-growing cell
	and come out the other side. If one is willing to do this
	(which seems to be the approach in \cite{Chiu95}),
	the methods of this paper should be applicable for dealing
	with boundary effects for these more
	general J-M models too.

	\textcolor{\blue}{
	These more general J-M type models with non-constant growth
	rates (in particular parabolic growth rates) and non-constant
	arrival rates, and with the balls allowed to interpenetrate,
	have been studied in the literature in relation
	to the understanding of asymptotics of convex hulls of
	random points \cite{Schreiber},
	maximal points of a multidimensional sample \cite{Schreiber}
	and constructing generalized Delaunay tessellations \cite{Gusakova}.
	It could be interesting to try to map our methods
	here onto questions of interest in these related models;
	this is a possible future direction of research.
}
	}
\end{remk} 
\begin{remk} 
	\label{rk:TCEV}
{\rm
	The distribution of the maximum of two independent Gumbel
	variables with different scale parameters
	\textcolor{\blue}{
		(e.g. the variable shown at \eqref{e:TCEVlimit})}
	is known
	as a {\em two-component extreme value} (TCEV) distribution
	in the hydrology literature \cite{Rossi}.
	
	Theorems \ref{thmwksq} and \ref{thm3d} tell us  that when
	$d=2$ and $A$ is polygonal or $\partial A \in C^2$,
	the random variables
	$\omega_d \rho T_\rho^{d+1}$
	and of $\omega_d \tau_L^{d+1}$, suitably centred, 
	are asymptotically
	TCEV distributed, while if $d \geq 3$ and $\partial A \in C^2$,
	they are asymptotically  Gumbel distributed. 

\textcolor{\blue}{The TCEV arises elsewhere in geometric extreme value theory. For example,
	suppose $L_{n,k,d}$ denotes
	the largest $k$-nearest neighbour link in
a sample of $n$ uniform random points in $[0,1]^d$, $n > k$.
That is, denoting these points $p_1,\ldots,p_n$, we have
	$L_{n,k,d} = \max_{1 \leq i \leq n}( k$-$\min_{j \leq d, j \neq n} \|p_i-p_j\|)$, where $k$-$\min(\cdot)$ means the $k$th smallest of a set of at least
	$k$ numbers.
It is shown in \cite{PY24}
	that $n \pi L_{n,2,2}^2 $, suitably centred, is asymptotically
 TCEV distributed.
 Again this is due to boundary effects,
and again, these have sometimes been missed in the literature;
	for example it is claimed in \cite[p214]{OT23}, incorrectly,
	that $n \pi L_{n,2,2}^2$ suitably
	centred is asymptotically Gumbel. The limiting distribution of
	$L_{n,k,d}$ (suitably transformed)
	for general fixed $(k,d)$ is given in \cite[Theorem 8.4]{P03},
	which the authors of \cite{OT23} were apparently unaware of.}
}
\end{remk}

\subsection{Results for the spherical Poisson Boolean model}

As discussed in the introduction,
we shall derive our results for the J-M model 
from
results on coverage by the restricted SPBM, which are of independent interest
and which we now state. 
Given a nonnegative random variable $Y$
and given $n > 0$, suppose we have a collection
of balls of independent random radii with the distribution of $Y$,
centred on the points of 
a homogeneous Poisson point process in $A$ of intensity $n$.

Given also $k \in \N$, let 
$Z_{A,k}(n,Y)$ be the set of points $x \in \R^d$ such that
$x$ is covered by at least $k$ of the balls,
so in particular
$Z_{A,1}(n,Y)$ is the union of the balls
(we shall provide a more detailed description of $Z_{A,k}(n,Y)$ 
in Section \ref{ss:half-space}.) 
Throughout this paper, $n $ is not necessarily an integer.

\begin{theo}[\textcolor{\blue}{Limiting probability of $k$-coverage of a 
	polygonal domain by a restricted SPBM}]
	\label{th:general2d}
	Suppose $Y$ is a nonnegative random variable with
	$0 < \E[Y^{2+\eps}] < \infty$ for some $\eps >0$.
	Suppose that $d=2$ and $A$ is polygonal.
	Let $k \in \N, \beta \in \R$, and suppose $(r_n)_{n >0}$
	are nonnegative numbers 
	satisfying
	$n \pi r_n^2 \E[Y^2] - \log n -  (2k-1) \log \log n \to \beta$
	as $n \to \infty $.
	Then  
	as $n \to \infty$, 
	\begin{align}
		\Pr[ A \subset Z_{A,k}(n,r_nY)]
		\to
		\exp \Big( - \Big( \frac{(\E[Y])^2}{\E[Y^2]} \Big)
		{\bf 1}_{\{k=1\}}
		|A|  e^{- \beta}
		- \Big( \frac{c_{2,k} \E[Y]
			|\partial A|
		}{(\E[Y^2])^{1/2}} \Big)
		e^{-\beta/2} \Big),
		\label{0128a}
	\end{align}
	where we set $c_{2,k} := \pi^{-1/2}(1/2)^{k-1}/(k-1)!$.
\end{theo}

For general $d, k \in \N$ with $ d \geq 2$ 
we define $c_{d,k}$ (as in \cite{P23}) by
\begin{align}
	c_{d,k} & :=  c_{d-1} 
	\omega_d^{2-d-1/d} \omega_{d-1}^{2d-3}
	\omega_{d-2}^{1-d}
	(1- 1/d)^{d+k -3 + 1/d} 2^{-1+1/d} /(k-1)!.
	\label{cd1def}
\end{align}
It can be checked that
\begin{align}
	c_{d,1} 
	=  ((d-1)!)^{-1}2^{1-d}
	(d-1)^{d-2+1/d}
	\pi^{(d/2)-1}
	\Gamma
	\left( \frac{d+1}{2} \right)^{1-d} \Gamma \left( \frac{d}{2}
	\right)^{d - 1 + 1/d },
	\label{e:cdkdef2}
\end{align}
and $c_{d,k} = c_{d,1} (1-1/d)^{k-1}/(k-1)!$.
Also $c_{2,1} = \pi^{-1/2}$
and $c_{3,1}= 2^{-4} \pi^{5/3}$.
In  the right hand side of
\eqref{e:cdkdef2} the first factor of $((d-1)!)^{-1}$
was given, incorrectly, as $(d!)^{-1}$ in \cite{P23}. 

\begin{theo}[\textcolor{\blue}{Limiting probability of $k$-coverage
	of a smoothly bounded region in $d\geq 2$ by a restricted SPBM}]
	\label{th:generaldhi}
	Suppose that $d \geq 2$ and 
	$\partial A \in C^2$ with $\overline{A^o} =A$.
	Suppose $Y$ is a nonnegative random variable with
	$0 <\E[ Y^{d+ \eps} ] < \infty $ for some $\eps >0$.
	Let $k \in \N, \beta \in \R$, and suppose $(r_n)_{n >0}$
	are nonnegative numbers satisfying 
	\begin{align}
		n \omega_d r_n^d \E[Y^d]  - (2-2/d) \log n - 
		2(d+k-3+1/d) \log \log n \to \beta
		\label{e:rn}
	\end{align}
	as $n \to \infty $.
	If $d=2$ then
	\eqref{0128a} holds, while if $d \geq 3$ 
	then  
	\begin{align}
		\lim_{n \to \infty} 
		\Pr[ A \subset Z_{A,k}(n,r_nY)]
		=
		\exp\Big( - c_{d,k} 
		\Big(\frac{(\E[Y^{d-1}])^{d-1}}{(\E[Y^{d}])^{d-2+1/d}} \Big)
		|\partial A| e^{-\beta/2} \Big).
		\label{0128b}
	\end{align}
\end{theo}

\begin{remk}
	{\rm Taking $Y \equiv 1$ in these results gives
		Theorem 3.2 and Theorem 3.1 of \cite{P23}.
			\textcolor{\blue}{That paper}
		also includes
		parallel results (in the case $Y \equiv 1$) 
		for a situation where the number of grains in
		the spherical Boolean model is a large deterministic
		constant rather than a Poisson variable. We
		expect that  a result along these lines
		could be given 
		for general $Y$ using the methods of this paper.
	}
\end{remk}
\begin{remk}
	{\rm The condition $\overline{A^o}=A$ should have
		been included in the statement of \cite[Theorem 3.1]{P23}.
		For example, if $d=2$ and
		$A$ is the union of a disk (filled in) and a circle
		(not filled in, and disjoint from the disk) then
		that theorem does not apply.}
\end{remk}
\begin{remk} {\rm In Theorem \ref{th:generaldhi}, we conjecture
		that
		the 
		condition $\partial A \in C^2$ can be relaxed 
		to $\partial A \in
		C^{1,1}$. 
	}
	
\end{remk}
\begin{remk}
	\label{r:heavy-tail}
	{\rm
		If $\E[Y^d]= \infty$ then
		$\Pr[A \subset Z_{\R^d,k}(n,rY)] =1$ for any $k \in \N$ and
		any $r>0$;
		see e.g.\ \cite[Theorem 16.4]{LP} for the case $k=1$.
		However, one could still look for a sequence
		$(r_n)$ such that $\Pr[A \subset Z_{A,1}(n,r_nY)]$ converges
		to a nontrivial limit. It seems plausible in this case
		that $A$ is most likely to be covered (if at all)
		by a small number of balls.
		This looks like an interesting problem that
		lies beyond the scope of this paper.
	}
\end{remk}
\begin{remk}
	\label{r:videos}
	{\rm
		A series of videos illustrating the SPBM
		and Remark~\ref{r:heavy-tail}
		can be viewed at
		\href{https://www.youtube.com/playlist?list=PLiaV5rk6Gk7qEXpLOU7FSvN4b8dy1_GJn}{this link}\footnote{\href{https://www.youtube.com/playlist?list=PLiaV5rk6Gk7qEXpLOU7FSvN4b8dy1_GJn}{\texttt{https://www.youtube.com/playlist?list=PLiaV5rk6Gk7qEXpLOU7FSvN4b8dy1\_GJn}}}.
		There are three videos, with $n = 1000$, $n=10000$ and $n=20000$ respectively.
		Each video has a Poisson point process of intensity $n$ inside $[0,1]^2$
		with points $\{ p_1, \dots, p_{N_n} \}$.
		The frame at time $t$ is an image of a tessellation corresponding
		to a restricted SPBM,
		$Z_{A,1}(n,r_t Y_t)$,
		where $A = [0,1]^2$,
		$Y_t$ is Pareto distributed with shape parameter $\alpha(t)$
		and $r_t$ is chosen so that $A \subset Z_{A,1}(n,r_t Y_t)$.
		There are $N_n$ cells in the frame:
		the pixel with centre $x \in [0,1]^2$ is assigned to a cell based on
		$\mathrm{argmin}_{i}\inf\{r \geq 0 : x \in B(p_i,r Y_{t,i})\}$,
		where $(Y_{t,i})_{1 \leq i \leq N_n}$ are independent copies of $Y_t$.
		In other words, for $t$ fixed the tessellation is generated
		by a growth process 
		for which all seeds are born at time 0, 
		growth rates are random and distributed like $Y_t$, and 
		the growing cells are interpenetrable.
		
		Note that cells are not necessarily connected.
		When displaying the video
		we colour the cells
		so that no two adjacent cells have the same colour,
		although two distinct cells which never touch are allowed to have the same colour.
		
		Over the video, $\alpha(t)$ changes from $3.00$ to $0.25$,
		i.e. $Y_t$ becomes more heavy-tailed,
		with the $Y_t$s coupled using the definition $Y_t := U^{-1/\alpha(t)}$
		where $U$ is uniformly distributed on $[0,1]$.
		
		At the beginning of each video,
		$\E[Y_t^{3-\varepsilon}] < \infty$ for all $\varepsilon > 0$
		so Theorem~\ref{th:general2d} applies.
		As the video progresses,
		we gradually reduces the number of finite moments,
		finishing when $\E[Y_t^{0.25}] = \infty$.
		
		The videos after the point where $\E[Y_t^2] = \infty$
		illustrate Remark~\ref{r:heavy-tail}.
		In the period when $\E[Y_t^2] = \infty$
		but $\E[Y_t^{7/4}] < \infty$,
		there is not yet a single cell containing more than half the area of $[0,1]^2$,
		indicating that there may be a non-trivial period where several balls are required
		to cover $[0,1]^2$.
	}
\end{remk}

\subsection{Strategy of proof}
\label{ss:strategy}
Next, we briefly describe the strategy for the proof, in Sections
\ref{s:pfJM} and \ref{secpfwk}, of
the weak convergence results that were stated above.
\textcolor{\blue}{In both sections,}
we shall use a known result (Lemma \ref{lemHall}) 
giving the limiting probability of covering
a  bounded region
of $\R^d$ by an unrestricted SPBM,
in the limit of 
high intensity $n$ and small balls of random radius,
i.e. distributed as $r_nY$ for a specified nonnegative random
variable $Y$ and for constants $r_n$ that become small as $n $ becomes large.

\textcolor{\blue}{We claim that the question of coverage
for  the (restricted) SPBM with uniformly distributed
radii maps onto the same question for the restricted Johnson-Mehl model.
Indeed,
given $t > 0$, consider a SPBM in $A$ with Poisson intensity $n$
and radii uniformly distributed over $[0,t]$. Label
the centres $x_1,\ldots,x_N$ with associated radii
$t_1,\ldots,t_N$. Then for each $i$, $t_i$ is uniformly
distributed over $[0,t]$ and therefore
so is $t-t_i$. Hence by the Marking theorem (see e.g. \cite{LP}),
the point process $\eta:= \{(x_i,t - t_i), 1 \leq i \leq N\}$
is a Poisson process in $A \times [0,t]$ with intensity given by
the product of Lebesgue measure on $A$ and the uniform probability distribution
on $[0,t]$, i.e. $n/t$ times Lebesgue measure on $A \times [0,t]$.
Hence the covered region for the original SPBM in $A$ is the
same as the covered region for the restricted J-M model (run up to time
$t$) obtained by using the Poisson process $\eta$. We can argue similarly
in the unrestricted case too.}

Using this claim and
applying Lemma \ref{lemHall} for the SPBM with $Y$
uniformly distributed on $[0,1]$ will
yield a proof of Proposition \ref{Hallthm}.
Similarly, applying Theorems \ref{th:general2d}
and \ref{th:generaldhi} with uniform $Y$ 
will yield  proofs of 
Theorems \ref{thmwksq} and
\ref{thm3d} respectively.


\textcolor{\blue}{Our strategy  for
proving Theorems 
\ref{th:general2d}
and \ref{th:generaldhi}
goes as follows.}
We shall consider, for large $n$
and suitable choice of $r_n$, 
the SPBM with radii distributed as $r_nY$, restricted to
a $d$-dimensional half-space $\bH$.  
In Lemma \ref{lemhalfd} we determine the limiting probability that
a given  bounded 
set within the hyperplane $\partial \bH$
is covered, 
by applying Lemma 
\ref{lemHall} in $d-1$ dimensions.
Moreover we will show
that the probability  that a region in the half-space  within distance 
$n^\xxi r_n$
of  that set  is covered with the same limiting probability, where
$\xxi$ is a small positive constant.

We shall then prove  
Theorem \ref{th:general2d} 
by applying Lemma \ref{lemhalfd}
to determine the limiting probability that  the
region near the edges of a polygonal set $A$ is covered,
and Lemma   \ref{lemHall} directly to determine the limiting
probability that the interior region is covered,
along with a separate argument 
to show the regions
near the corners  of  $A$ are  covered with high probability.

To prove Theorem
\ref{th:generaldhi} we approximate to $A$ by a polytopal set $A_n$
with faces of width  $O(n^{9 \xxi} r_n)$ and follow similar steps to
those just mentioned for polygonal $A$. 

In the case where $Y$ is unbounded, to obtain the required
independence,
for example between coverage events for different faces of our
polyhedron or polygon, we consider only those balls of radius
at most $n^\zeta r_n $,  with a separate argument (Lemma \ref{l:2PPs})
to show
this does not affect the limiting coverage probabilities.

\section{Proof of results for the J-M model}
\label{s:pfJM}

\textcolor{\blue} We now use Theorems 
\ref{th:general2d} and \ref{th:generaldhi} to
prove Theorems \ref{thmwksq}
and \ref{thm3d}.
The proof of Theorems
\ref{th:general2d} and \ref{th:generaldhi} is much longer
and we defer this to the next section.

We shall repeatedly use   
the following result from \cite{P23}, which is based
on results in \cite{Janson} or
\cite{HallZW}. Recall that $c_d$ and $Z_A(n,Y)$ were defined at
(\ref{cdef}) and just before Theorem \ref{th:general2d}
respectively.
\begin{lemm}[Limiting probability of coverage by an unrestricted SPBM]
\label{lemHall}
Let $d \in \N$.
Let $\alpha := \omega_d \E[Y^d] $.
Let $\beta \in \R$. Suppose 
$ \delta(\lambda) \in (0,\infty)$  
is defined for all $\lambda >0$, and satisfies
\bea
\lim_{\lambda \to \infty} \left( \alpha \delta(\lambda)^d \lambda -
\log \lambda  - (d+k-2) \lglg \lambda \right)  
= \beta.
\label{0315c}
\eea
Let $B \subset \R^d$  be compact and Riemann measurable, and for 
each $\lambda >0$ let
$B_\lambda \subset B$
be Riemann measurable with the properties that  
$B_\lambda \subset B_{\lambda'}$ 
whenever $\lambda \leq \lambda'$, and  that
$\cup_{\lambda >0} B_\lambda \supset B^o$.
Then 
\bea
\lim_{\lambda \to \infty} 
\Pr[B_\lambda \subset Z_{\R^d,k}(\lambda,\delta(\lambda)Y)
=
\exp \left(-  \left( 
\frac{c_d  (\E [Y^{d-1}] )^d}{ 
(k-1)!(\E[Y^d ])^{d-1} } 
\right)|B| e^{-\beta}  \right). 
\label{0315b}
\eea
\end{lemm}
\begin{proof}
See \cite[Lemma 7.2]{P23}.
There, it was assumed that $Y$ is bounded, but the proof
carries over if this condition is relaxed to the $(d+\eps)$-th moment
condition used here.
\end{proof}

\begin{lemm}[Scaling lemma]
\label{l:scaling}
Suppose 
$L >0$ and
$\rho = L^{d+1}$.
Then
$L \tT_\rho$ has the same distribution as $\tilde{\tau}_L$.
Moreover
$L T_\rho$ has the same distribution as $\tau_L$.
\end{lemm}
\begin{proof}
For any  $ t > 0$, setting $t' = Lt$, we have
\begin{align*}
A \subset 
\cup_{(x,s) \in \cH_\rho \cap
	(\R^d \times [0,t])} B(x,t-s) 
\Leftrightarrow 
A_L \subset 
\cup_{(y,u) \in L \cH_\rho \cap
	(\R^d \times [0,t'])} B(y,t'-u) 
\end{align*}
so that by \eqref{e:tT}, 
\begin{align*}
L \tT_\rho = \inf \{ t': 
A_L \subset 
\cup_{(y,u) \in L \cH_\rho \cap
	(\R^d \times [0,t'])} B(y,t'-u) 
\}.
\end{align*}
By our choice of $\rho$ and the Mapping theorem (see e.g.\ \cite{LP}),
$L \cH_\rho$ is a homogeneous Poisson process of
intensity $1$ on $\R^d \times \R_+$,
and therefore by the definition \eqref{e:ttau},
$L \tT_\rho$ has the same distribution as $\tilde{\tau}_L$.

The proof  that
$L T_\rho$ has the same distribution as $\tau_L$ is similar.
\end{proof}

\begin{proof}[Proof of Proposition \ref{Hallthm}]
Suppose for some $\beta \in \R$ that $(t_\rho)_{\rho >0}$ satisfies
\bea
\lim_{\rho \to \infty} 
\left( \omega_d  \rho  t_\rho^{d+1}
- d \log  \rho  -   d^2 \lglg \rho   
\right) 
= \beta, 
\label{0511a}
\eea
i.e.  $t_\rho = \left(\left(d\log \rho + d^2 \lglg \rho + \beta
+ o(1)\right)/(
\omega_d \rho)\right)^{1/(d+1)}$.

Write $\cH_{\rho} \cap (\R^d \times [0,t_\rho])
= \{(X_i,S_i)\}_{i \in \N}$. 
Then the point process $\Po_{\rho t_\rho} :=
\{X_i\}_{i \in \N}$ is a homogeneous
Poisson process of intensity $\rho t_\rho $ in $\R^d$.

The balls $B(X_i,t_\rho -S_i) $, $i \in \N,$ form a SPBM 
in $\R^d$,  and in the notation of 
Lemma \ref{lemHall}, here we have
$\delta (\lambda) = t_\rho$, 
and $Y $ uniformly distributed over $[0,1]$,
so that $\alpha =\omega_d/(d+1)$,
and $\lambda = \rho t_\rho $,
so that
$$
\log \lambda = \left( \frac{d}{d+1} \right)
\log \rho + \frac{\log \log \rho + \log (d/\omega_d)}{d+1} 
+ o(1),
$$
and
$\lglg \lambda = \lglg \rho + \log (d/(d+1)) + o(1)$.
Hence,
\begin{align*}
\alpha \delta(\lambda)^d \lambda - \log \lambda - (d-1) \lglg \lambda
& = 
\frac{\omega_d
	\rho t_\rho^{d+1}
}{d+1}
- \frac{d \log \rho}{d+1}  - \frac{
	d^2 \lglg \rho }{d+1}  
- \frac{\log(d/\omega_d)}{d+1}
\nonumber \\
& 
~~~~~~~ - (d-1) \log \left( \frac{d}{d+1} \right)+ o(1),
\end{align*}
and by 
(\ref{0511a}), as $\rho \to \infty$  this tends to
$$
\frac{\beta}{d+1} - \log \left( \left( \frac{d^{d^2}}{\omega_d
(d+1)^{d^2 -1}} \right)^{1/(d+1)} \right).
$$
Hence  we have the condition 
(\ref{0315c})
from Lemma \ref{lemHall} (with $\beta$ on the right hand side
replaced by
the last displayed expression).
Also $\E[Y^{d-1}] = 1/d$ 
and $\E[Y^d] = 1/(d+1)$.

We have $\tT_\rho \leq t_\rho$
if and only if $A \subset \cup_{i\in \N}B(X_i,t_\rho -S_i)$.
Therefore by
taking $k=1$ and $B_\lambda = A$ for all $\lambda$ in
Lemma \ref{lemHall},
we obtain that
$$
\Pr[\tT_{\rho} \leq t_\rho] \to  \exp \left(-
c_d 
\left( \frac{(d+1)^{d-1}}{d^d} \right)
\left( \frac{d^{d^2}}{\omega_d (d+1)^{d^2-1} } \right)^{1/(d+1)}
|A|e^{-\beta/(d+1)}\right).
$$
This yields (\ref{1228a}).

For \eqref{e:ttaulim}, 
take $\rho = L^{d+1}$, so that $\log \log L
= \log \log \rho - \log(d+1)$.
By Lemma \ref{l:scaling},
$L \tT_\rho$ has the same distribution as $\tilde{\tau}_L$.
Therefore
	\textcolor{\blue}{
\begin{align*}
	& \Pr \left[ \omega_d \tilde{\tau}_L^{d+1}
- d (d+1) \log L   -
d^2 \lglg L 
		- d^2 \log (d+1)
\leq \beta \right] 
\\
	& = \Pr\left[
\omega_d \rho \tilde{T}_\rho^{d+1}
- d \log \rho - d^2 \log \log \rho 
\leq \beta \right],
\end{align*}
		}
and by \eqref{1228a}, 
\textcolor{\blue}{
this gives us
 \eqref{e:ttaulim}.}
\end{proof}

We now derive Theorems \ref{thmwksq} and
\ref{thm3d} from Theorems  \ref{th:general2d} and
\ref{th:generaldhi}, respectively.

\begin{proof}[Proof of Theorem \ref{thmwksq}]
We take $Y$ to be uniformly distributed
on $[0,1]$. If the restricted J-M model with intensity $\rho$
runs for time $t$, the 
intensity of the resulting restricted SPBM is $n=t\rho$,
and the distribution of radii is that of $tY$.
Let $\beta \in \R$, and define 
\begin{align*}
t_\rho : = \Big( \frac{(2 \log \rho + 4 \log \log \rho + \beta)
	\vee 0
}{\pi \rho } \Big)^{1/3}, ~~~~~~~~~~~\rho > 1,
\end{align*}
so that
$
\pi \rho t_\rho^3 - 2 \log \rho - 4 \log \log \rho = \beta,
$
for all large enough $\rho$. 
Then as $\rho \to \infty$,
\begin{align*}
\log  (\rho t_\rho )
& =  (2/3) \log \rho + (1/3) \log \log \rho + 
\log ( (2/\pi)^{1/3}) + o(1),
\end{align*}
and
$
\log \log (\rho t_\rho) = \log \log \rho + \log (2/3) + o(1).
$
Therefore since $\E[Y^2]=1/3$, we have
\begin{align*}
\pi \rho t_\rho^3 \E[Y^2]- \log ( \rho t_\rho) -
\log \log (\rho t_\rho)
= & \beta/3 + (2/3) \log \rho + (4/3) \log \log \rho \\
& - (2/3) \log \rho - (1/3) \log \log \rho -  \log
((2/\pi)^{1/3})
\\
& - \log \log \rho - \log (2/3) + o(1)
\\
= & (\beta/3) - \log ((16/(27 \pi))^{1/3}) + o(1),
\end{align*}
so applying Theorem \ref{th:general2d} with $n=\rho t_\rho$,
$r_n = t_\rho$, $k=1$ and $Y$ uniform on $[0,1]$ yields
\begin{align*}
\Pr[ \pi \rho 
T_{\rho}^3 - 2 \log \rho  - 4   \lglg \rho 
\leq \beta ]
= \Pr[T_\rho \leq t_\rho] 
= \Pr[A \subset Z_{A,1}(n,r_nY)]
\\
\to
\exp \Big( - \Big(\frac{3}{4}\Big) 
|A|\Big(\frac{16}{27 \pi}\Big)^{1/3}   e^{- \beta/3}
- 
\Big(\frac{3}{4\pi}\Big)^{1/2} 
|\partial A|
\Big(\frac{16}{27 \pi} \Big)^{1/6}
e^{-\beta/6} \Big),
\end{align*}
and hence
\eqref{0322b}.

For \eqref{e:taulim}, let $\rho = L^{3}$, so that
$\log \log \rho = \log \log L + \log (3)$.
Then by Lemma \ref{l:scaling},
\textcolor{\blue}{
\begin{align*}
\Pr[ \pi \tau_L^3 - 6 \log L -4 \log \log L - \log 81
\leq \beta]
= \Pr [ \pi \rho T_\rho^3 - 2 \log \rho - 4 \log \log \rho
\leq \beta],
\end{align*}
and by \eqref{0322b}, this yields
\eqref{e:taulim}.
}
\end{proof}

\begin{proof}[Proof of Theorem \ref{thm3d}]
Let $\beta \in \R$ and for $\rho > 1$ define
\begin{align}
t_{\rho} :=
\Big(
\frac{(2(d-1) \log \rho  + 2d(d-1)   \lglg \rho 
	+ \beta) \vee 0 }{\omega_d \rho}  \Big)^{1/(d+1)}, 
\label{e:for3dpf}
\end{align}
so that
$
\omega_d \rho 
t_{\rho}^{d+1} - 2(d-1) \log \rho  - 2d(d-1)   \lglg \rho 
= \beta ,
$
for all large $\rho$.
Then as $\rho \to \infty$,
\begin{align*}
\log (\rho t_\rho)
= & \frac{d\log \rho +  \log \log \rho}{d+1}
+ \log \Big( \Big(\frac{2 (d-1)}{\omega_d} \Big)^{1/(d+1)}\Big)
+ o(1),
\end{align*}
and
$
\log \log (\rho t_\rho) = \log \log \rho + \log (d/(d+1)) 
+ o(1).
$
Therefore taking $Y$ to be uniformly distributed
on $[0,1]$,
so that $\E[Y^d] = 1/(d+1)$,
using \eqref{e:for3dpf} we obtain that
\begin{align*}
\rho t_\rho  \omega_d t_\rho^d \E[Y^d]
- (2-2/d) \log (\rho t_\rho) - 2(d-2+1/d) \log \log
(\rho t_\rho)
\\
= \frac{2(d-1) \log \rho + 2d(d-1) \log \log \rho}{d+1} 
+ \frac{\beta}{d+1} 
\\
- \frac{(2-2/d)(d\log \rho +  \log \log \rho)}{d+1}
- (2-2/d)\log\Big(\Big(\frac{2(d-1)}{\omega_d}
\Big)
^{1/(d+1)}
\Big)
\\
- 2(d-2+1/d) \log \log \rho - 2(d-2+1/d) \log(d/(d+1)) +o(1)
\\
= \frac{2d(d-1) - (2-2/d) -2 (d-2 +1/d) (d+1) }{d+1}
\log \log \rho +
\frac{\beta}{d+1} - c_d''' +o(1)
\\
= \frac{\beta}{d+1} - c_d''' +o(1),
\end{align*}
where we set
\begin{align*}
c'''_d := \log \Big( \Big( \frac{2(d-1)}{\omega_d} 
\Big)^{(2-2/d)/(d+1)} 
\Big( \frac{d}{d+1}
\Big)^{2(d-2+1/d)}
\Big).
\end{align*}
Therefore applying Theorem \ref{th:generaldhi}
with $k=1$, $n=\rho t_\rho$ and $r_n = t_\rho$,
we obtain that if $d \geq 3$ then
\begin{align*}
\Pr[ \omega_d \rho 
T_{\rho}^{d+1} - 2 (d-1) \log \rho  - 2d (d-1)   \lglg \rho 
\leq \beta ]
= \Pr[T_\rho \leq t_\rho] 
= \Pr[A \subset Z_{A,1}(n,r_nY)]
\\
\to \exp \Big(
-\Big(
\frac{
	c_{d,1} 
	(d+1)^{d-2+1/d}
	|\partial A|
}{d^{d-1}} 
\Big)
\Big( \frac{2(d-1)}{\omega_d} \Big)^{(1-1/d)/(d+1)}
\\
\times
\Big( \frac{d}{d+1} \Big)^{d-2+1/d} e^{-\beta/(2d+2)} 
\Big)
\\
= \exp \Big( - \frac{
	c_{d-1} \omega_{d-1}^{2d-3}
	(d-1)^{d-2+(1/d) + ((1-1/d)/(d+1))}
	|\partial A| e^{-\beta/(2d+2)}}{
	\omega_{d-2}^{d-1}\omega_d^{d-2+(1/d)
		+ ((d-1)/(d(d+1)))} 2^{1-(1/d) - (1-1/d)/(d+1)} d^{d-1}} \Big)
\\
= \exp \Big( - \frac{
	c_{d-1} \omega_{d-1}^{2d-3}
	(d-1)^{d(d-1)/(d+1)}
	|\partial A| e^{-\beta/(2d+2)}}{
	\omega_{d-2}^{d-1} \omega_d^{d(d-1)/(d+1)} 2^{(d-1)/(d+1)}
	d^{d-1}
} \Big),
\end{align*}
and hence \eqref{e:Tlim3d}. 

If $d=2$ then our $t_\rho$ defined at \eqref{e:for3dpf}
is the same as in the proof of Theorem \ref{thmwksq}.
Applying the case $d=2$ of Theorem \ref{th:generaldhi}, in the same way
as we applied Theorem \ref{th:general2d}
in the proof of Theorem \ref{thmwksq}, gives us the same outcome
as in Theorem \ref{thmwksq}, namely \eqref{0322b}.

When $d=2$ we obtain \eqref{e:taulim} from \eqref{0322b}
as in the proof of Theorem \ref{thmwksq}.
When $d \geq 3$, to get
\eqref{e:taulim3d} we set $\rho = L^{d+1}$. By Lemma \ref{l:scaling},
\textcolor{\blue}{
\begin{align*}
	& \Pr[\omega_d \tau_L^{d+1} -2 (d^2-1) \log L -2 d (d-1) [\log \log L
	+ \log (d+1)]
	\leq \beta] &
\\
	& = \Pr[ \omega_d \rho T_\rho^{d+1} -2 (d-1) \log \rho
-2 d(d-1) \log \log \rho 
	\leq \beta ],
\end{align*}
and then by  \eqref{e:Tlim3d},
we have \eqref{e:taulim3d}.
}
%
\end{proof}

\section{Proof of Theorems \ref{th:general2d} and \ref{th:generaldhi}}
\label{secpfwk}

Throughout this section, we assume
$d, k \in \N$ are fixed with $d \geq 2$, and
$A \subset \R^d$ is compact and nonempty 
with $A = \overline{A^o}$, and that we are given a nonnegative random variable
$Y$ satisfying $0 < \E[Y^\gamma] < \infty$ for some $\gamma >d$.

\subsection{Preliminaries}
\label{subsecprelims}
We now give  some further notation used throughout.
Let $o$ denote
the origin in $\R^d$.
Set $\bH := \R^{d-1}\times [0,\infty)$ and 
$\partial \bH := \R^{d-1}\times \{0\}$.


Given two  sets $\X,\Y \subset \R^d$, we
set  $ \X \triangle \Y := (\X \setminus \Y) 
\cup (\Y \setminus \X)$, the symmetric difference between $\X$ and $\Y$.
Also, we write $\X \oplus \Y$ for the set $\{x+y: x \in \X, y \in \Y\}$,
and $\#(\X)$ for the number of elements of $\X$ (possibly $+\infty$).
Given also $x \in \R^d$ we write $x+\Y$ for $\{x\} \oplus \Y$.

Given a Borel measure $\mu$ on $\R^d$, and Borel $D \subset \R^d$,
let $\mu|_D$ denote the restriction of the measure $\mu$ to $D$
(a Borel measure on $D$). 

Given $x,y \in \R^d$, we denote by $[x,y]$ the line segment from
$x$ to $y$, that  is, the convex hull of the set $\{x,y\}$.
We write $a \wedge b$ (respectively $a \vee b$) for the minimum
(resp. maximum) of any two numbers $a,b \in \R$.

Given  $m \in \N$ and functions
$f: \N \cap [m,\infty) \to \R$ and
$g: \N \cap [m,\infty) \to (0,\infty)$,
we write $f(n) = O(g(n))$ as $n \to \infty$ 
if  
$\limsup_{n \to \infty }|f(n)|/g(n) < \infty$.
We write $f(n)= o(g(n))$ as $n \to \infty$
if $\lim_{n \to \infty} f(n)/g(n) =0$, 
and $f(n) \sim g(n)$ as $n \to \infty$
if $\lim_{n \to \infty} f(n)/g(n) =1$.
We write $f(n)= \Theta(g(n))$ as $n \to \infty$
if $f$ takes only positive values, and
both $f(n)= O(g(n))$ and $g(n)= O(f(n))$.
Given $s >0$ and  functions $f: (0,s) \to \R$ and $g:(0,s) \to (0,\infty)$,
we write $f(r) = O(g(r))$ as $r \downarrow 0$
$\limsup_{r \downarrow 0} |f(r) |/g(r) < \infty$,
and write
$g(r) = \Omega(g(r))$ as $ r \downarrow 0$, if 
$\liminf_{r \downarrow 0} f(r) /g(r) > 0$.
We write $f(r)= o(g(r))$ as $r \downarrow 0$
if $\lim_{r \downarrow 0 } f(r)/g(r) =0$,
and $f(r) \sim g(r)$ as $r \downarrow 0$
if this limit is 1.

From time to time we shall use the following geometrical lemma.

\begin{lemm}[Geometrical lemma]
\label{lemgeom1a}
Suppose $\partial A \in C^2$, and
$A = \overline{A^o}$.  Given $\eps > 0$,
there  exists $r_0 = r_0(d,A, \eps)>0$  such that
\bea
|  B(x,r) \cap A| \geq ((\omega_d/2)-\eps)  r^d,
~~~~ \forall x \in A, r \in (0,r_0).
\label{e:0430a}
\eea
\end{lemm}
\begin{proof}
	See \textcolor{\blue}{\cite[Lemma 3.4]{CovXY} (or for a proof from first principles,
	Lemma 3.2(i) of v1 of that paper on Arxiv).}
\end{proof}

\subsection{Coverage in the boundary by a SPBM in a half-space}
\label{ss:half-space}

In this subsection, we 
assume $(r_n)_{n >0}$ satisfies  
\eqref{e:rn} 
for some $\beta \in \R$.
Then  as $n \to \infty$,
\begin{align}
r_n = \left( \frac{(2 - 2/d) \log n + 2(d +k-3 + 1/d) \log \log n
+ \beta +o(1)}{n \omega_d \E[Y^d]} \right)^{1/d},
\label{e:trho}
\end{align}
and
\bea
\label{rtd+}
\exp(-\omega_d  n r_n^d \E[Y^d]) \sim
n^{-(2 - 2/d)} (\log n)^{-2(d+k-3+1/d)} e^{-\beta}.
\eea

We shall use the following notation throughout the sequel.  
Given $n >0$ let $\cU_n$
be a  Poisson point process in $\R^d \times \R_+$ 
(viewed as a random set of points in $\R^{d+1}$)
with intensity measure 
$n {\rm Leb}_d \otimes \mu_Y$, where ${\rm Leb}_d$ denotes $d$-dimensional
Lebesgue measure and $\mu_Y$ denotes the distribution of $Y$.
Given Borel $D \subset \R^d$, 
and given any point process $\X$ in $\R^d \times \R_+$,
we define
\begin{align}
\cU_{n,D} & := \cU_{n} \cap (D \times \R_+);
\label{e:Zdef}
\\
Z_n(\X) & : =  
\{y \in \R^d: \# \{(x,s) \in \X: y \in B(x,r_ns)\} \geq k\};
\label{e:occ}
\\
Z_n^o(\X) &:=
(Z_n(\X))^o.
\label{e:occint}
\end{align}
Thus, $\cU_{n,D}$ is  a Poisson process in
$D \times \R_+$
with intensity measure $n {\rm Leb}_d \otimes \mu_Y$
(strictly speaking, 
with intensity measure $n {\rm Leb}_d|_D \otimes \mu_Y$).
Equivalently, it can be viewed
as a homogeneous, independently {\em marked} Poisson process
in $D$ of intensity $n$ with marks having the distribution of $Y$
(see e.g.\ \cite[Theorem 5.6]{LP}.)

The set $Z_n(\cU_{n,D})$  
is the region that is covered at least $k$ times
(for short: the $k$-covered region)
for the restricted SPBM in $D$
with intensity $n$ and with radii having the distribution of
$r_n Y$. It is the same as $Z_{D,k}(n,r_nY)$ in earlier notation;
we now consider $k$ to be fixed and suppress it from the notation.
The set $Z_n^o(\cU_{n,D})$ is the region covered at least $k$ times
by a union of {\em open} balls.

Recall that  $\bH := \R^{d-1} \times [0,\infty)$.
The main \textcolor{\blue}{results of this subsection
are Lemmas \ref{lemhalfd} and \ref{lemhalfdint}  below, concerning}
the limiting
probability of covering a $(d-1)$-dimensional
region of the form $\Omega \times \{0\} $ in
the hyperplane  $\partial \bH := \R^{d-1} \times \{0\}$ by
$Z_n(\cU_{n,\bH})$,
or of covering
the $a_nr_n$-neighbourhood of $\Omega \times \{0\}$ in $\bH$, for $a_n$
not growing too fast with $n$.
It is crucial for dealing with boundary regions 
in the proof of Theorem \ref{thm3d}.

Before getting to that, we present another idea which will be important
later.
Let $\xxi >0$ to be chosen later (think of $\xxi$ as a small constant). 
Given Borel $D \subset \R^d$, 
we consider two separate, independent Poisson processes
$\cU'_{n,D}$ \textcolor{\blue}{and $\cU''_{n,D}$}, defined by
\begin{align}
\cU'_{n,D} = \cU_n \cap (D \times [0,n^\xxi]);
~~~~~~~~~
\cU''_{n,D} = \cU_n \cap (D \times (n^\xxi,\infty)).
\label{e:U'}
\end{align}
For various sets $D$ and ultimately for $D=A$, we shall be interested
in coverage of certain regions by the set $Z_n(\cU_{n,D})$; the next
lemma says that the limiting probability of coverage is unaffected
if we take
$Z_{n}(\cU'_{n,D})$
instead of
$Z_{n}(\cU_{n,D})$; this is useful because the random sets 
$Z_{n}(\cU'_{n,D})$ have better spatial independence properties, considered
as a function of $D$.

\begin{lemm}[Asymptotic equivalence of coverage events using $\cU_{n,D}$
or $\cU'_{n,D}$]
\label{l:2PPs}
Let $n_0 \in (0,\infty)$.
Suppose that Borel  $D_n \subset \R^d$ and $ E_n \subset \R^d$ 
are defined for each $n \geq n_0$.  Then
as $n \to \infty$,
\begin{align}
\Pr[\{E_n \subset  Z_n(\cU_{n,D_n} )\} \setminus
\{E_n \subset Z_n(\cU'_{n,D_n})\}] \to 0. 
\label{e:0212d}
\end{align}
\end{lemm}
\begin{proof}
Recall that we are assuming $\E[Y^\gamma] < \infty$ for some 
$\gamma >d$.
Given $x \in \R^d$, let $N_n(x)$ denote the number of points
$(y,t) \in \cU''_{n,D_n}$ such that $\|y-x\| \leq r_n t$. Then
by the H\"older and Markov inequalities,
\begin{align*}
\Pr[N_n(x) \geq 1 ] \leq
\E[N_n(x)] & \leq n \int_{(n^\xxi,\infty)}  |B(x,r_nt)| \mu_Y(dt)
\\
& = n \omega_d r_n^d \E[Y^d {\bf 1}_{\{Y > n^\xxi\}} ] 
\\
& \leq n \omega_d r_n^d (\E[Y^{\gamma}])^{d/\gamma}
(\Pr[Y > n^\xxi ])^{(\gamma -d)/\gamma} 
\\
& \leq \omega_d
(\E[Y^{\gamma}])^{d/\gamma}
nr_n^d  (\E[Y^{\gamma}]/n^{\xxi \gamma}
)^{(\gamma - d)/\gamma},
\end{align*}
which tends to zero since
$nr_n^d = O(\log n)$ by \eqref{e:trho}.  Therefore
\begin{align}
\sup_{x \in \R^d} \Pr[x \in Z_n(\cU''_{n,D_n} ) ] \to 0 
~~
{\rm as} ~ n \to \infty.
\label{e:0212e}
\end{align}
Let $n \geq n_0$.
Since $E_n \subset \R^d$, $E_n$ is separable;
see e.g.\ \cite[page 20]{Pryce}. Let $\{x_{n,i}\}_{i \in \N}$
be an enumeration of a  countable dense set in $E_n$.  
Define the ($\N \cup +\infty$)-valued random variable
$J_n$ to be the first $i$ such that 
$x_{n,i} \notin Z_n(\cU'_{n,D_n})$,
or $J_n := +  \infty$  if there is no such $i$.
Since
$E_n \setminus Z_n(\cU'_{n,D_n})$ is open in $E_n$, if 
$E_n \setminus Z_n(\cU'_{n,D_n}) \neq \emptyset $ then $J_n < \infty$.
Therefore 
\begin{align*}
\Pr[ E_n \subset Z_n(\cU_{n,D_n}) |\{E_n \subset
Z_n(\cU'_{n,D_n}) \}^c] \leq
\Pr[x_{n,J_n} \in Z_n(\cU''_{n,D_n})|J_n < \infty] ,
\end{align*}
which tends to zero by \eqref{e:0212e} and the independence
of $\cU'_{n,D_n}$ and 
$\cU''_{n,D_n}$,
and \eqref{e:0212d}
follows.
\end{proof}

The following terminology and notation will be used
\textcolor{\blue}{repeatedly in the sequel.}
Given $x \in \R^d$, we let $\pi_1(x),\ldots,\pi_d(x)$ denote
the co-ordinates of $x$, and refer to $\pi_d(x)$ as the
{\em height} of $x$.
Given $\bx_1 = (x_1,s_1), \ldots, \bx_d = (x_d,s_d) \in \R^d
\times \R_+$, 
if $\cap_{i=1}^d \partial B(x_i,s_i)$
consists of exactly two points, we refer to these as
$p(\bx_1,\ldots,\bx_d)$ and
$q(\bx_1,\ldots,\bx_d)$ 
with $p(\bx_1,\ldots,\bx_d)$ at a
smaller height than 
$q(\bx_1,\ldots,\bx_d)$ 
(or if they are at the same height, take 
$p(\bx_1,\ldots,\bx_d) <  
q(\bx_1,\ldots,\bx_d)$ in the lexicographic ordering).
Define the indicator function
\begin{align}
h(\bx_1,\ldots,\bx_d) := & {\bf 1} \{
\pi_d(x_1) \leq \min(\pi_d(x_2), \ldots, \pi_d(x_d)) \}
\nonumber \\
& \times {\bf 1} \{ \#( \cap_{i=1}^d \partial B(x_i,r_i)  
) = 2 \} {\bf 1} \{ \pi_d(x_1) <
\pi_d(q(\bx_1,\ldots,\bx_d))\} .  
\label{hrdef2a}
\end{align}
Given also $n >0$ \textcolor{\blue}{we define} 
\begin{align}
h_n(\bx_1,\ldots,\bx_d):=h((x_1,r_ns_1),\ldots,(x_d,r_ns_d));
\label{e:hn}
\\
p_n(\bx_1,\ldots,\bx_d):=p((x_1,r_ns_1),\ldots,(x_d,r_ns_d));
\label{e:pn}
\\
q_n(\bx_1,\ldots,\bx_d):=q((x_1,r_ns_1),\ldots,(x_d,r_ns_d)).
\label{e:qn}
\end{align}
\textcolor{\blue}{These notations are illustrated in Figure
\ref{f:pnqn}.}
\begin{figure}[!h]
        \center
        \includegraphics[width=15cm]{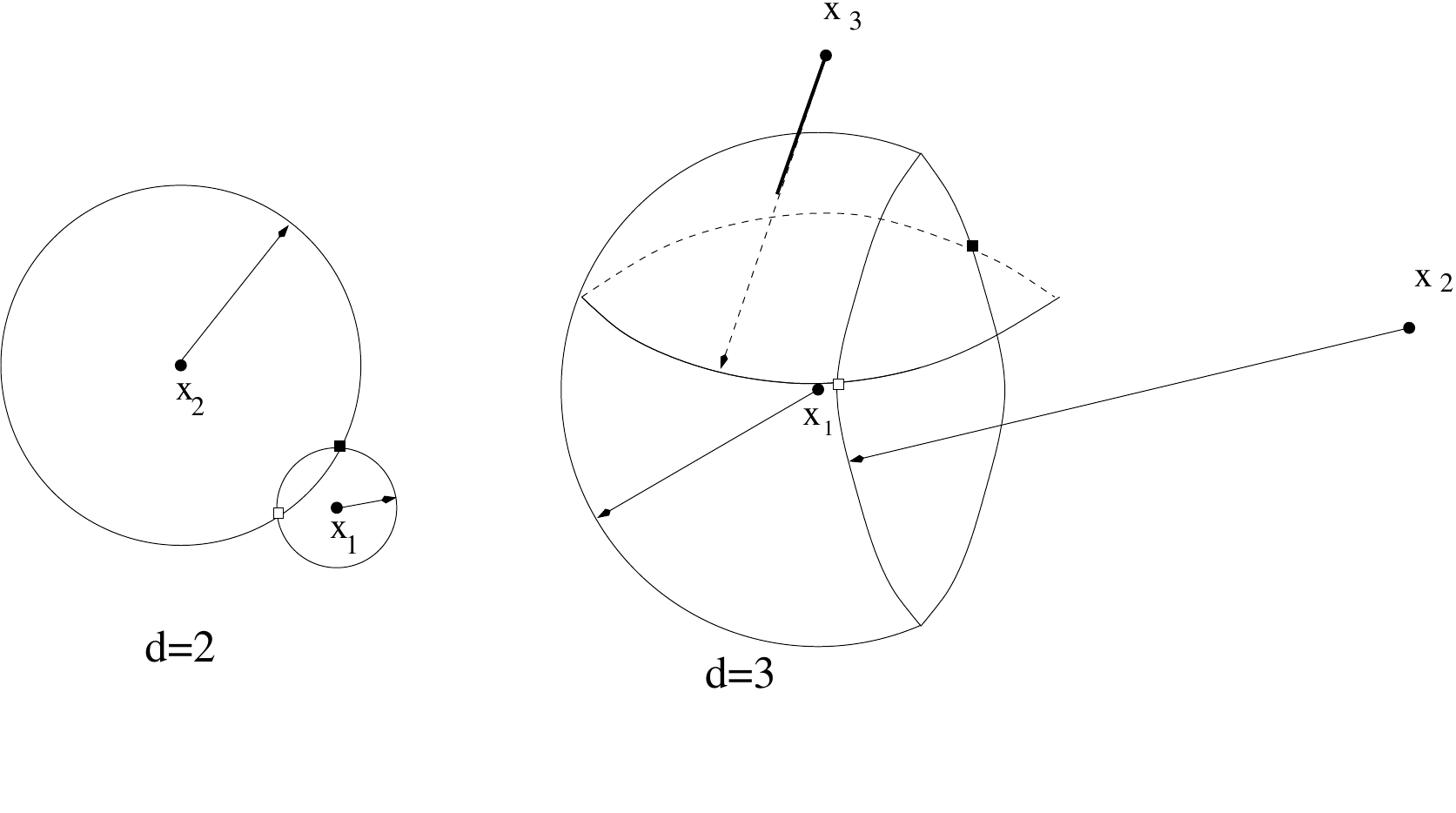}
	\caption{\label{f:pnqn} \textcolor{\blue}{
	Examples in $d=2$ and $d=3$ showing
	a $d$-tuple $\{\bx_1,\ldots,\bx_d\}$,
	satisfying  $h_n(\bx_1,\ldots,\bx_d)=1$ and showing
	$p_n(\{\bx_1,\ldots ,\bx_d\})$ (white square)
	and
	$q_n(\{\bx_1,\ldots,\bx_d\})$ (black square),
	where, for $i= 1,\ldots,d$, $\bx_i= (x_i,s_i)$ and 
	the arrow from $x_i$ is of length $r_n s_i$.}
        }
\end{figure}
\begin{lemm}[Integrability of $h$]
\label{l:inth}
There exists a constant $c$ depending only on $d$ such 
that for all $x_1 \in \R^d$ and all $s_1,\ldots,s_d \in \R_+$,
\begin{align}
\int_{\R^d} \cdots \int_{\R^d}
h((x_1,s_1),\ldots,(x_d,s_d)) dx_2 \cdots dx_d 
\leq c\prod_{i=1}^d (1 \vee s_i^{d-1}).
\label{e:ubinth}
\end{align}
\end{lemm}
\begin{proof}
Denote the left hand side of \eqref{e:ubinth} by 
$I(x_1,s_1,\ldots,s_d)$.
Divide $\R^d$ into rectilinear unit hypercubes
$Q_z, z \in \Z^d$, with $Q_z$
centred at $z$ for each $z$. 
Then
\begin{align*}
I(x_1,s_1,\ldots,s_d)
&	\leq \sum_{z \in \Z^d}
\int_{\R^d} \cdots \int_{\R^d}
{\bf 1}\{ \cap_{i=1}^d \partial
B(x_i,s_i) \cap Q_z \neq \emptyset\}
dx_2 \cdots dx_d
\\
& \leq \sum_{z \in \Z^d} 
\int_{\R^d} \cdots \int_{\R^d}
\prod_{i=1}^d
{\bf 1}\{  \partial
B(x_i,s_i) \cap Q_z \neq \emptyset\}
dx_2 \cdots dx_d
\\
& = \sum_{z \in \Z^d} 
{\bf 1}\{  \partial B(x_1,s_1) \cap Q_z \neq \emptyset\} 
\times
\prod_{i=2}^d
\int_{\R^d} 
{\bf 1}\{  \partial B(x_i,s_i) \cap Q_z \neq \emptyset\}
dx_i.
\end{align*}
Given $s \geq 0$,
the value of the integral 
$\int_{\R^d} {\bf 1}\{  \partial B(x,s) \cap Q_z \neq \emptyset\}
dx$ does not depend on $z$, and is bounded by $c' (1 \vee s^{d-1})$
for some constant $c'$, independent of $s$.
Moreover the sum $\sum_{z \in \Z^d} 
{\bf 1}\{  \partial B(x,s) \cap Q_z \neq \emptyset\} 
$ is bounded by $c'' (1 \vee s^{d-1})$
for some further constant $c''$, independent of $s$ and $x$.
Applying these two observations gives us
\eqref{e:ubinth}.
\end{proof}

We are now ready to present the \textcolor{\blue}{first main result
of this section.
In this result}, $|\Omega|$ 
denotes the $(d-1)$-dimensional
Lebesgue measure of $\Omega$.
Recall the definition of $c_{d,1}$ at \eqref{cd1def}.

\begin{lemm}[Coverage \textcolor{\blue}{of a portion of the boundary} of a half-space]
\label{lemhalfd}
Let $\Omega \subset \R^{d-1}$ 
	be closed, bounded and
	Riemann measurable.
Assume that \eqref{e:rn} holds for some $\beta \in \R$
or $\beta = + \infty$,
and also that $ \limsup_{n \to \infty}
( n r_n^d/(\log n )) < \infty$. 
Then
\begin{align}
\lim_{n \to \infty} (\Pr[
\Omega \times \{0\}
\subset Z_n(\cU_{n,\bH})] 
) 
= 
	\exp \left(-
	c_{d,k} \Big( \frac{(\E[Y^{d-1}])^{d-1}}{(\E[Y^{d}])^{d-2+ 1/d}} \Big)
	|\Omega| 
	e^{-  \beta/2}  \right). 
	\label{0517c2a}
\end{align}
\end{lemm}
\begin{remk}
{\rm
	\textcolor{\blue}{When $\beta < \infty$,}
	the extra condition
	$ \limsup_{n \to \infty}
	( n r_n^{d}/(\log n )) < \infty$ 
	is automatic.
	When $\beta =\infty$, in  (\ref{0517c2a}) we use the convention
	$e^{-\infty} := 0$.}
\end{remk}

\begin{proof}[Proof of Lemma \ref{lemhalfd}]
Assume for now that $\beta < \infty$.
Considering the slices of  balls induced by the points of
$\cU_{n,\bH}$ that
intersect the hyperplane $\R^{d-1} \times \{0\}$,
we have a $(d-1)$-dimensional SPBM, 
where
we claim the parameters 
(in the notation of Lemma \ref{lemHall}) are
\bea
\delta = r_n, ~~~ \lambda = n r_n \E[Y], ~~~~
\alpha = \omega_{d} \E[Y^d]/(2 \E[Y]).
\label{e:Boo1a}
\eea 
We justify these claims as follows.
The intensity of balls intersecting
the hyperplane $\partial \bH_d$
	per \textcolor{\blue}{unit `area'} (i.e. per unit
$(d-1)$-dimensional Lebesgue measure),
is equal to $n \int_0^\infty \Pr[r_n Y \geq t]dt = n r_n \E[Y]$. 
Also if $W$ denotes the radius of a slice, divided
by $r_n$, the distribution
of $W$ is that of $\tilde{Y}(1-U^2)^{1/2}$, where 
$\tilde{Y}$ has the size-biased distribution of $Y$
(i.e., $\Pr[\tilde{Y} \in dt] = t \Pr[Y \in dt]/\E[Y]$),
and $U$ is uniformly distributed on $[0,1]$, independent
of $\tilde{Y}$. Therefore
%
\begin{align}
	\E[W^{d-1}] & = \E[\tilde{Y}^{d-1}]
	\int_0^1
	(1 -u^2)^{(d-1)/2} du 
	\nonumber	\\
	& =  (\E[Y^d]/\E[Y]) \omega_d/(2\omega_{d-1}), 
	\label{e:EYd}
\end{align}
and the asserted value of $\alpha$ follows.

By \eqref{e:Boo1a},  followed by \eqref{e:rn},
\begin{align*}
	\alpha \delta^{d-1} \lambda & =
	\Big( \frac{ \omega_d \E[Y^d]}{2 \E[Y]}
	\Big) r_n^{d-1} n r_n \E[Y]  = 
	(\omega_d/2) n r_n^d \E[Y^d] \\
	& = 
	(1-1/d) \log n  + (d+k-3 + 1/d)
	\log \log n +
	(\beta/2)
	+ o(1). 
\end{align*}
Also by \eqref{e:Boo1a} and \eqref{e:trho},
\begin{align*}
	\log \lambda 
	& =
	(1-1/d)
	\log n + \Big(\frac{1}{d}\Big)
	\log \log n +
	\log \Big(\Big(\frac{2-2/d}{\omega_d \E[Y^d]}\Big)^{1/d}
	\E[Y] \Big) 
	+ o(1).
\end{align*}
Also
$
\log \log \lambda = \log \log n + \log (1-1/d)
+ o(1).
$
Hence
\begin{align*}
	\alpha \delta^{d-1} \lambda - \log \lambda
	- (d+k-3) \log \log \lambda =  (\beta/2) -
	\log(c''_{d,k,Y}) + o(1),
\end{align*}
where we take
\begin{align*}
	c''_{d,k,Y} 
	& : = 
	\left( \frac{2-2/d}{\omega_d \E[Y^d]
	}\right)^{1/d} 
	\E[Y]
	(1-1/d)^{d+k-3}
	=
	(1-1/d)^{d+k-3+1/d}
	\left( \frac{2
		(\E[Y])^d
	}{\omega_d \E[Y^d]}
	\right)^{1/d} 
	.
\end{align*}

By  \eqref{e:EYd}, and a similar calculation  for $\E[W^{d-2}]$,
\begin{align*}
	\frac{(\E[W^{d-2}])^{d-1}}{( \E[W^{d-1}])^{d-2}} =
	\left( \frac{\omega_{d-1} \E[Y^{d-1}]}{2 \omega_{d-2}
		\E[Y]} \right)^{d-1} 
	\left( \frac{\omega_d \E[Y^d]}{2 \omega_{d-1} \E[Y]}
	\right)^{2-d}
	= 
	\frac{ \omega_{d-1}^{2d-3} (\E[Y^{d-1}])^{d-1}}{2 \omega_{d-2}^{d-1} 
		\omega_d^{d-2} (\E[Y^d])^{d-2} \E[Y]}.
\end{align*}
Thus by Lemma \ref{lemHall} we obtain that
\begin{align*}
	\lim_{n \to \infty}
	(\Pr[ \Omega \times \{0\} \subset Z_n(\cU_{n,\bH})])
	= \exp \Big(
	-
	\Big(\frac{c_{d-1} 
		\omega_{d-1}^{2d-3}
		(\E[Y^{d-1}])^{d-1}
	}{ (k-1)!
		2
		\omega_{d-2}^{d-1} 
		\omega_d^{d-2}
		\E[Y] (\E[Y^d])^{d-2}
	}\Big) \\
	\times c''_{d,k,Y} |\Omega| e^{-\beta/2}
	\Big), 
\end{align*}
and hence \eqref{0517c2a}.

Having now verified (\ref{0517c2a}) in the case where $\beta < \infty$,
we can then easily deduce (\ref{0517c2a}) in the other case too.
\end{proof}

\textcolor{\blue}{Next we aim to show that for any bounded Borel set $\Omega$ in $\R^{d-1}$,
the probability of there being any uncovered region in $\bH$, lying close
to the boundary region $\Omega \times \{0\}$ but not intersecting
$\partial \bH$ itself, is vanishingly small. We shall do this
by using the fact that such a region must have
an ``exposed lower corner''  in $\bH$ near $\Omega \times \{0\}$,
and estimating  the number of such corners.}

\textcolor{\blue}{For this argument we need further notation.
Given $\Omega \subset \R^{d-1}$, $a >0$, and given $(r_n)_{n >0}$,
let $M_n(\Omega,a)$ denote the number of
 $d$-tuples of marked points $\bx_1 = (x_1,s_1),\ldots,
\bx_d = (x_d,s_d) \in \cU_{n,\bH}$,
such that
$h_n(\bx_1,\ldots,\bx_d)=1$,
and moreover
$q_n(\bx_1,\ldots,\bx_d) \in 
( \Omega \times (0,a r_n])
\setminus Z_n(\cU_{n,\bH} \setminus \{\bx_1,\ldots,\bx_d\}) 
$.
Thus if $k=1$ then $M_n(\Omega,a) $ is
the number of corners of an uncovered region that lie in
$\Omega \times (0,a r_n]$ for which at least one of
the balls having that corner on its boundary has its centre
below the corner, as illustrated in Figure \ref{f:Mn}.} 

\begin{figure}[!h]
        \center
	\includegraphics[width=8cm]{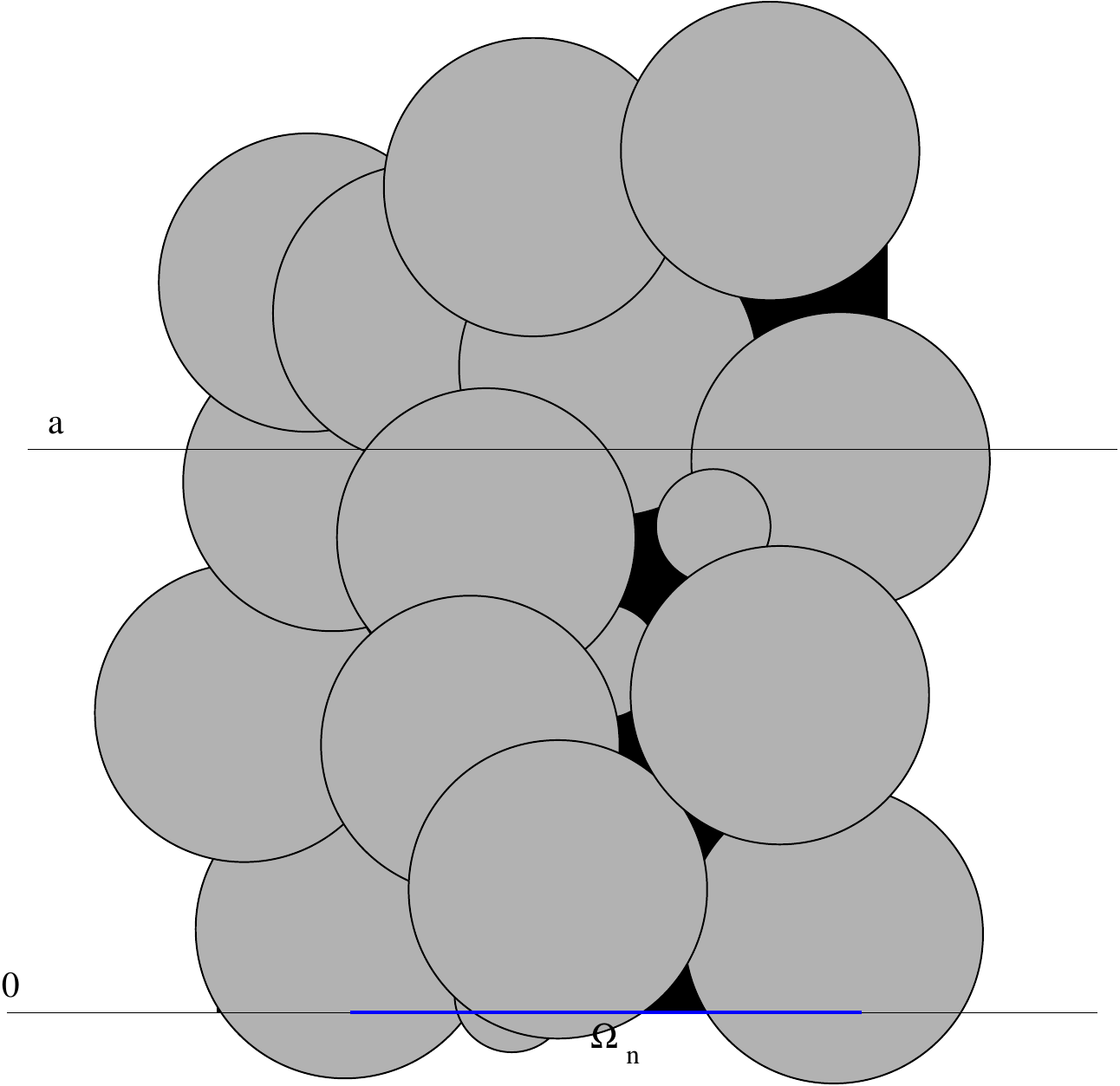} 
	\caption{\label{f:Mn}{Example in
	$d=2$ with $k=1$. Here $M_n(\Omega,a) = 11$ 
	since, taking the uncovered regions below height $a$ in descending
	order, we have a region with 5 corners that all contribute to
	$M_n(\Omega,a)$,
	one with 4 corners three of which contribute, 
	and one with 3 corners that all contribute, and a region
	with one interior corner that does not contribute.
	The thick blue line within the $x$-axis is the region $\Omega_n$.
	}
        }
\end{figure}

\begin{lemm}[\textcolor{\blue}{Estimating the mean number of exposed lower corners
	near $\partial \bH$}]
	\label{l:corners}
	\textcolor{\blue}{Let $\Omega $ be a bounded Borel set in $\R^{d-1}$,
	let $a \in [1,\infty)$ and let $(r_n)_{n >0}$ satisfy
	\eqref{e:rn}  for some $\beta \in \R$
or $\beta = + \infty$,
and that $ \limsup_{n \to \infty} ( n r_n^d/(\log n )) < \infty$. 
	Let $a \in [1,\infty).$ Then}
\begin{align}
	\lim_{n \to \infty} (\E[M_n(\Omega,a)]) =  0.
	\label{e:0212c}
\end{align}
\end{lemm}
\begin{proof}
Choose \textcolor{\blue}{$\delta \in (0,1)$}
such that $\Pr[ Y > \delta ] > \delta$.
Then \textcolor{\blue}{we can and do} choose $c' >0$ so that  
$$
\omega_{d-1} \delta^d 2^{1-d} ((\delta/2) \wedge u) \geq c'u,
~~~
\forall u \in [0,a].
$$
\textcolor{\blue}{
	To be definite, take $c':=
a^{-1} \omega_{d-1} \delta^{d} 2^{-d}$; then the displayed inequality
holds for $u=a$, and hence also for smaller $u$.
	}

For any $y \in \bH$, let $N_n(y)$ denote the number of balls
making up $Z_n(\cU_{n,\bH})$ which cover $y$, i.e. the number
of $(x,s) \in \cU_{n,\bH}$ such that $\|x-y\| \leq r_n s$.
Also, let $\bH_y := \{z \in \bH:\pi_d(z) > \pi_d(y) \}$.
Then for all $n$ and
all  $y \in \bH$ with $0 \leq \pi_d(y) \leq ar_n$,
for all $s >0$, the half-ball $B(y,r_ns) \cap \bH_y$, and also
the cylinder with radius $sr_n/2$,
upper face centred on $y$ and height $(r_ns/2) \wedge \pi_d(y)$, 
are disjoint and contained in $B(y,r_ns) \cap \bH$, 
\textcolor{\blue}{as illustrated in Figure \ref{f:mushroom}.}
\begin{figure}[!h]
        \center
	\includegraphics[width=8cm]{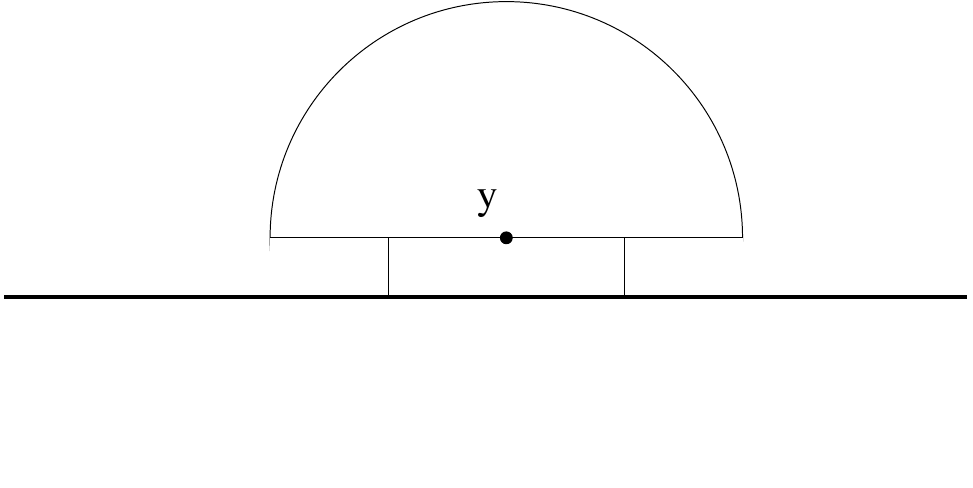}
	\vspace{-1cm}
	\caption{\label{f:mushroom} \textcolor{\blue}{Illustration
	of the geometric estimate
	used to derive equation \eqref{e:0212b}. The thick line at
	the bottom is $\partial \bH$.}}
\end{figure}
Therefore
\begin{align}
	\E[N_n(y)] & = n \int_{\R_+}  |B(y,r_ns) \cap \bH| \mu_Y(ds) 
	\nonumber \\
	& \geq n \int_{\R_+} (\omega_d r_n^d s^d/2)\mu_Y(ds)
	+ n \int_{(\delta,\infty)} \omega_{d-1} (r_ns/2)^{d-1} ((r_ns/2)
	\wedge \pi_d(y))
	\mu_Y(ds)
	\nonumber	\\
	& \geq (n \omega_d/2) r_n^d \E[Y^d] + 
	n\delta \omega_{d-1} (\delta r_n/2)^{d-1} r_n
	((\delta /2) \wedge (\pi_d(y)/r_n)) 
	\nonumber	\\
	& \geq (n \omega_d/2) r_n^d \E[Y^d] + c' n r_n^{d-1} \pi_d(y).
	\label{e:0212b}
\end{align}
%
%
%

By the Mecke formula (see e.g.\ \cite{LP}), 
taking $Y_1,\ldots,Y_d$ to be independent random variables
with the  distribution of $Y$, 
there is a constant $c$ such that
\begin{align}
	\E[M_n(\Omega,a)]
	\leq & c n^d  \int_{\bH }  
	\cdots
	\int_{\bH }  \E [h_n((x_1,Y_1),\ldots,(x_d,Y_d))
	\nonumber \\
	& \times {\bf 1}\{  q_n((x_1,Y_1),\ldots,(x_d,Y_d)) \in \Omega
	\times [0,ar_n] \} (k (n\omega_d r_n^d)^{k-1})
	\nonumber \\
	& \times
	e^{- n (\omega_d/2) \E[Y^d]   r_n^{d}   - c'  n r_n^{d-1}
		\pi_d( q_n((x_1,Y_1),\ldots,(x_d,Y_d))  )}]
	dx_d \cdots dx_1
	.  ~~~
	\label{0522a3}
\end{align}

Now we change variables to $y_i= r_n^{-1}(x_i-x_1)$ for
$2 \leq i \leq d$, noting that 
\begin{align*}
	\pi_{d}(q_n((x_1,Y_1),\ldots,(x_d,Y_d))) 
	= \pi_d(x_1) +
	\pi_d( q_n((o,Y_1),(x_2 -x_1,Y_2), \ldots,  (x_d-x_1,Y_d)) ). 
\end{align*}
With these changes of variable, we get a factor of
$r_n^{d(d-1)}$ on changing from $(x_2,\ldots,x_d)$ to 
$(y_2,\ldots,y_d)$.
Hence by (\ref{rtd+}), there is a constant $c''$ such that
the expression  on the right hand side of (\ref{0522a3}) is at most
\begin{align}
	& c'' n^{d+k-1} r_n^{d(d+k-2)}
	n^{- (1-1/d)} (\log n)^{-(d+k-3+1/d)}
	\int_0^{a r_n} e^{-c'  n r_n^{d-1} u} du 
	\nonumber \\
	& \times
	\int_{\bH } \cdots \int_{\bH }
	\E [h_n((o,Y_1),(r_ny_2,Y_2),\ldots,(r_ny_d,Y_d)) 
	\nonumber
	\\ &  ~~~~~~~~~~~~~~~~~  \times
	e^{ - c' n r_n^{d-1} \pi_d( q_n((o,Y_1),(r_ny_2,Y_2),
		\ldots,(r_ny_d,Y_d))) } ] 
	dy_d \ldots dy_2.
	\label{0512b2}
\end{align}
In the last expression the first line is bounded by a constant times
the expression
\bean
n^{d+k-3+1/d} (r_n^d)^{d+k-2 - (d-1)/d} (\log n )^{-(d+k-3 +1/d)}
= (n r_n^{d}/\log n)^{d+k-3+1/d},
\eean
which is bounded 
because 
$n r_n^{d}
= O(\log n)$ 
by \eqref{e:trho}.
Using Fubini's theorem and the definitions of $h_n(\cdot)$
and $q_n(\cdot)$ in \eqref{e:hn} and \eqref{e:qn},
the expression in the second and third
lines of  (\ref{0512b2}) can be rewritten as
\begin{align*}
	\E \Big[
	\int_{\bH } \cdots \int_{\bH }
	h((o,Y_1),(y_2,Y_2),\ldots,(y_d,Y_d)) 
	e^{ - c' n r_n^{d} \pi_d( q((o,Y_1),(y_2,Y_2),
		\ldots,(y_d,Y_d))) }  
	dy_d \ldots dy_2 \Big]
	\\
	=  
	\int_{\R_+^d}
	\int_{\bH^{d-1}} 
	h((o,u_1),(y_2,u_2),\ldots,(y_d,u_d))
	e^{ - c' n r_n^{d} \pi_d( q((o,u_1),(y_2,u_2),
		\ldots,(y_d,u_d)))}
	\\
	d(y_2,\ldots,y_d)
	\mu_Y^d(d(u_1,\ldots,u_d)).
\end{align*}
This
tends to zero by dominated convergence
because $n r_n^{d} \to \infty$ by (\ref{e:rn}) and
because the 
indicator
function $ ({\bf u}, {\bf y}) \mapsto h((o,u_1),(y_2,u_2),\ldots,
(y_d,u_d)) $ is integrable 
and  is zero when $\pi_d(q((o,u_1),(y_2,u_2),\ldots,(y_d,u_d))) \leq 0$;
the integrability follows from Lemma \ref{l:inth}
and the fact that $\E[Y^{d-1}] < \infty$
since we assume $\E[Y^{\gamma}]<\infty$ for some $\gamma > d $.

Therefore the expression in \eqref{0512b2} tends to zero as $n \to \infty$,
and \eqref{e:0212c} follows.
\end{proof}


\textcolor{\blue}{We are now ready to state the second main result of
this subsection.}

	\begin{lemm}[\textcolor{\blue}{Coverage of a region
		just inside the boundary of a half-space}]
	\label{lemhalfdint}
		\textcolor{\blue}{
		For each $n >0$ let $\Omega_n \subset \R^{d-1}$
		be Riemann measurable, with $\cup_{n>0} \Omega_n$
		bounded.
Assume that \eqref{e:rn} holds for some $\beta \in \R$
or $\beta = + \infty$,
and that $ \limsup_{n \to \infty} ( n r_n^d/(\log n )) < \infty$. 
Also let $a_n \in (0,\infty)$
for each $n >0$,
		and assume for some $\eps >0$ 
		that $a_n = O(n^{(1/d) - \eps})$ as $n \to \infty$.
		Then
}
\begin{align}
	\lim_{n \to \infty}
	(\Pr[\{(\Omega_n \times \{0\} ) 
	\cup ( 
	&	(\partial \Omega_n)  
	\times [0,a_n r_n]) \subset   
	Z^o_n(\cU'_{n,\bH})  \}
	\nonumber \\
	& \setminus
	\{	( \Omega_n \times [0, a_n r_n] ) \subset Z_n( \cU'_{n,\bH})
	\}] ) = 0.
	\label{0517b2a}
\end{align}
\end{lemm}
\begin{proof}
For the duration of this proof only,
let $F_n$ be the event displayed in \eqref{0517b2a} but with
$\cU'_{n,\bH}$ changed to 
$\cU_{n,\bH}$ both times, i.e.
\begin{align*}
	F_n:=
	\{(\Omega_n \times \{0\} ) 
	\cup ( 
	(\partial \Omega_n)  
	\times [0,a_n r_n]) \subset   
	Z^o_n(\cU_{n,\bH})  \}
	\setminus
	\{	( \Omega_n \times [0, a_n r_n] ) \subset Z_n( \cU_{n,\bH})
	\}.
\end{align*}
By Lemma \ref{l:2PPs}, to prove \eqref{0517b2a} it suffices
to prove that $\Pr[F_n ] \to 0$ as $n \to \infty$.

Let $E_n$ be the (exceptional)
event that there exist $d$ distinct marked points $(x_1,s_1),\ldots,$ $
(x_d,s_d)$
of $\cU_{n,\bH}$
such that $ \cap_{i=1}^d \partial B(x_i, r_n s_i) $
has non-empty intersection with the hyperplane 
$\partial \bH$.
Then $\Pr[E_n]=0$.

Suppose that $F_n \setminus E_n$ occurs.
Let $w$ be a location  of minimal height (i.e., $d$-coordinate) 
in the closure of
$( \Omega_n \times [0,a_n r_n]) \setminus Z_n(\cU_{n,\bH})$.
Since we assume
$(\partial \Omega_n) 
\times [0,a_n r_n] \subset Z^o_n( \cU_{n,\bH})$ 
occurs, 
$w$ must lie in $\Omega_n^o \times
(0,a_nr_n]$. Also 
we claim that
$w$ must
be  a `corner' given by the
meeting point of the boundaries of exactly $d$ balls $B(x_1,r_ns_1),
\ldots, B(x_d,r_ns_d)$ with $\{(x_i,s_i)\}_{i=1}^d \subset  
\cU_{n,\bH} $,
and with $x_1$  the lowest
of the $d$ points $x_1,\ldots,x_d$,
and with
$\#(\cap_{i=1}^d \partial B(x_i,r_ns_i) 
) =2$, 
and $w \notin Z_n( \cU_{n,\bH} \setminus \{(x_1,s_1),\ldots,(x_d,s_d)\})$.

Indeed, if $w$ is not at the boundary of any such ball,
then for some $\delta >0$ we have $B(w, \delta) \cap 
Z_n(\cU_{n,\bH}) = \emptyset $,
and then we could find a location in $ (\Omega_n \times
[0,ar_n]) \setminus Z_n(\cU_{n,\bH})$ lower than $w$, a contradiction.
Next, suppose instead that $w$ lies at the boundary
of fewer than $d$ such balls.
Then denoting by $L$ the intersection of the supporting hyperplanes
at $w$ of each of these balls, we have that $L$ is an
affine subspace of $\R^d$, of
dimension at least 1. Take $\delta  >0$ small enough
so that $B(w,\delta)$ does not intersect any of the boundaries
of balls  centred at points of $\cU_{n,\bH}$ (viewed here
as a marked Poisson point process in $\R^d$) and with radius given by their 
marks times $r_n$,
other than those which meet at $w$.
Taking $w' \in L \cap B(w, \delta) \setminus \{w\}$ 
such that $w'$ is at least as low as $w$, we have
that $w'$ lies
in the interior of $Z_n(\cU_{n,\bH})^c$. Hence for some $\delta' >0$,
$B(w',\delta') \cap Z_n(\cU_{n,\bH}) = \emptyset$ and we can
find a location in $B(w',\delta')$ that is lower than $w$, 
yielding a contradiction for this case too. 
Finally, with probability 1 there is no set of
more than $d$ points of $\cU_{n,\bH}$ such that the boundaries
of the associated balls  have non-empty
intersection, so $w$ is not at the boundary of more than
$d$ such balls.
Thus we have justified the claim.

Moreover $w $ must be the point 
$ q_n(\bx_1,\ldots,\bx_d)$ rather
than $p_n(\bx_1,\ldots,\bx_d)$, where for
$1 \leq i \leq d$ we write $\bx_i$ for $(x_i,r_i)$,
because otherwise by extending the line segment from 
$ q_n(\bx_1,\ldots,\bx_d)$ to
$p_n(\bx_1,\ldots,\bx_d)$ slightly 
beyond  $p_n(\bx_1,\ldots,\bx_d)$ we could find  a  point in
$
(\Omega_n \times [0,a r_n])
\setminus Z_n(\cU_{n,\bH})$
lower than $w$, contradicting the statement that $w$ is
a location of minimal height in the closure  of
$(\Omega_n \times [0,a r_n]) \setminus Z_n(\cU_{n,\bH})$.
Moreover, $w$ must be strictly higher than $x_1$, since if
$\pi_d(w) \leq  \min(\pi_d(x_1), \ldots,\pi_d(x_d))$,
then locations  just below $w$ would lie
in $ (\Omega_n \times [0,a r_n]) \setminus Z_n(\cU_{n,\bH})$, contradicting
the statement that $w$ is
a  point of minimal height in the closure of $(\Omega_n \times 
[0,ar_n]) \setminus Z_n(\cU_{n,\bH})$.
Hence, $h_n(\bx_1,\ldots,\bx_d) = 1$, where $h_n(\cdot)$ was
defined at \eqref{e:hn}.

\textcolor{\blue}{Thus if  $F_n \setminus E_n$ occurs, then
$M_n(\Omega,a_n) \geq 1$, where we take $\Omega := \cup_{n >1} \Omega_n$,
and $M_n(\Omega,a)$ was defined just before Lemma \ref{l:corners}.
Hence by Markov's inequality, it suffices to show
that $\E[M_n(\Omega,a_n)] \to 0$ as $n\to \infty$.}


Choose $a \in [1,\infty)$ such that 
$$
\int_{(a,\infty)} x^d \mu_Y(dx) < \eps \E[Y^d]/4. 
$$
\textcolor{\blue}{Since $\E[M_n(\Omega,a')]$ is monotone nondecreasing in $a'$, we may
assume} without loss of generality that  $a_n > a$.
For $y \in \bH$ with $a r_n \leq \pi_d(y) \leq a_n r_n$, 
instead of \eqref{e:0212b} we use the estimate
\begin{align}
	\E[N_n(y)] & = n \int_{\R_+}  |B(y,r_ns) \cap \bH| \mu_Y(ds) 
	\nonumber \\
	& \geq n \int_{(0,a)} (\omega_d r_n^d s^d)\mu_Y(ds)
	\nonumber	\\
	& \geq n \omega_d r_n^d \E[Y^d] (1- \eps/4 ).
	\nonumber	
\end{align}
By this, and the Mecke formula,
there is a new constant $c$ such that
\begin{align}
	\E[M_n(\Omega,a_n) -M_n(\Omega,a)]
	\leq
	c n^d  (nr_n^d)^{k-1} \int_{\bH }  
	\cdots
	\int_{\bH }  \E [h_n((x_1,Y_1),\ldots,(x_d,Y_d))
	\nonumber \\
	\times
	{\bf 1}\{  q_n((x_1,Y_1),\ldots,(x_d,Y_d)) \in \Omega \times 
	[ar_n,a_nr_n] \}
	e^{- n \omega_d r_n^d \E[Y^d] (1- \eps/4)  
	}]
	dx_d \cdots dx_1
	.  ~~~
	\label{e:0212a}
\end{align}
Changing variables as before to $y_i= r_n^{-1}(x_i-x_1)$ for
$2 \leq i \leq d$, 
and using \eqref{rtd+},
we find  there is a constant $c'''$ such that
the expression  in the right hand side of \eqref{e:0212a} is at most
\begin{align}
	c''' n^{d+k-1} r_n^{d(d+k-2)}
	n^{- (2-2/d) + \eps/2} 
	\E \Big[
	\int_{\bH } \cdots \int_{\bH }
	h_n((o,Y_1),(r_ny_2,Y_2),\ldots,(r_ny_d,Y_d)) 
	\nonumber \\
	\times	{\bf 1}_{\Omega \times[ar_n,a_nr_n]} 
	( x_1 +  q_n((o,Y_1),(r_ny_2,Y_2),
	\ldots,(r_ny_d,Y_d)))   
	dx_1 dy_d \ldots dy_2 \Big]
	\nonumber \\
	\leq c''' (nr_n^d)^{d+k-2} n^{-1+(2/d)+ \eps/2}
	|\Omega|(a_n-a) r_n 
	\nonumber \\
	\times	\E \Big[ \int_{\bH^{d-1}}
	h((o,Y_1),(y_2,Y_2), \ldots,(y_d,Y_d))
	d ((y_2,\ldots,y_d)) \Big].
	\nonumber 
\end{align}
Since  $\E[Y^{d-1}]<\infty$ the
expectation in the last line is finite by Lemma \ref{l:inth},
and since  also $a_n =O(n^{(1/d)-\eps})$
we obtain that
$$
\E[M_n(\Omega,a_n) - M_n(\Omega,a)] = O( (nr_n^d)^{d+k-2+ 1/d} 
n^{-1+(2/d) + (\eps/2) -\eps}), 
$$
which tends to zero. Combined with \eqref{e:0212c} this
shows that $\E[M_n(\Omega,a_n)] \to 0$, 
\textcolor{\blue}{as required.}
\end{proof}

\subsection{{\bf Proof of Theorem \ref{th:general2d}}}

In this subsection, we
set $d=2$ and take
$A$ to be polygonal.
Denote the vertices of $A$ by $q_1,\ldots,q_\kappa$, and the angles subtended
at these vertices by $\alpha_1,\ldots,\alpha_\kappa$ respectively.

\textcolor{\blue}{To prove Theorem \ref{th:general2d}, we shall
split $A$ into an interior region, a region near the edges (but not
the corners) and a region near the corners of $A$. We shall use
Lemma \ref{lemHall} to determine the limiting probability of covering
the interior region, and the results from Subsection  \ref{ss:half-space}
to determine the limiting probability of covering the region near
the edges.
We shall provide a separate argument to show the probability that
the region near the corners is covered tends to 1.}

For the duration of this subsection (and the next) we
fix $\beta \in \R$.  Assume we are given real numbers
$(r_n)_{n >0}$ satisfying the case $d=2$ of \eqref{e:rn},
i.e.
\bea
\lim_{n \to \infty}
\left( \pi n r_n^2 \E[Y^2] -  \log n - (2k-1) \log \log n \right) = \beta.
\label{0220a}
\eea

Fix $\xxi  >0$.
For $n >0$, in this subsection we define the `corner regions' $Q_n$ and
$Q_n^-$ 
by  
$$
Q_n:= \cup_{j=1}^\kappa B(q_j,n^{3 \xxi}r_n) \cap A;
~~~~~
Q^-_n:= \cup_{j=1}^\kappa B(q_j, n^{2 \xxi} r_n) \cap A.
$$
Also we define $\cU_{n,A}$, $Z_n(\cU_{n,A})$ and 
$Z^o_n(\cU_{n,A})$ by \eqref{e:Zdef}, \eqref{e:occ} and \eqref{e:occint}.

\begin{lemm}[Coverage of regions near the corners of $A$]
\label{lemQ}
Provided $\xxi$ is taken to be small enough,
$\Pr[Q_n  \subset 
Z_n^o(\cU_{n,A})] \to 1$ as $n \to \infty$.
\end{lemm}
\begin{proof}
Let $\delta >0$ with $\Pr[Y> 3 \delta] > \delta$.
Then the restricted SPBM $Z^o_n(\cU_{n,A})$ stochastically dominates
a Poisson Boolean model $Z'_n$  with closed balls of
deterministic
radius $2 \delta r_n$, centred on the points of a homogeneous
Poisson point process in $A$ of intensity $\delta n$.
Let $Z''_n $ be
a Poisson Boolean model with closed balls of deterministic
radius $\delta r_n$, centred on the points of a homogeneous
Poisson point process in $A$ of intensity $\delta n$.

There is a constant $K$, independent of $n$,
such that for all $n \geq 1$ there exist points $y_{n,1},\ldots,y_{n,\lfloor K
	n^{6 \xxi}\rfloor}
\in A$
such that
$Q_n  \subset \cup_{j=1}^{\lfloor Kn^{6 \xxi}\rfloor}
B(y_{n,j},\delta r_n)$.  Also there is a constant $a >0$ such that for all
$n \geq 1$ and all $y \in A$
we have 
$|B(y,\delta r_n) \cap A |  \geq a r_n^2$.
Then
\begin{align}
	\Pr[ Q_n  \setminus Z_n^o(\cU_{n,A}) \neq \emptyset ] 
	&
	\leq 
	\Pr[ Q_n \setminus Z'_n \neq \emptyset ] 
	\nonumber \\
	& \leq  \Pr[ \cup_{j=1}^{\lfloor K n^{6 \xxi} \rfloor} \{y_{n,j} \notin Z''_n\}]
	\nonumber \\
	& \leq   K n^{6 \xxi} k(n\pi r_n^2)^{k-1}
	\exp(- a \delta n r_n^2), 
	\label{0213a}
\end{align}
and since $(n r_n^2)/\log n \to 1/(\pi \E[Y^2])$ by (\ref{0220a}),
provided we take $\xxi < a \delta /(6  \pi \E[Y^2])$
the expression in \eqref{0213a} tends to zero, and the result follows.
\end{proof}

Recall from \eqref{e:Zdef} the definition of the point processes
$\cU'_{n,D}$ for $n > 0,  D \subset \R^d$.
\begin{lemm}[Coverage of regions near the edges of $A$]
\label{LemRpp}
Assume $\xxi < 1/d$. Then
\bea
\lim_{n \to \infty} (
\Pr[(\partial A \setminus Q^-_n) \subset
Z_n(\cU'_{n,A}) ])  =
\exp \Big(-  \Big( \frac{c_{2,k} \E[Y]}{(  \E[Y^2])^{1/2}}
\Big) |\partial A| 
e^{-\beta/2} \Big).
\label{0429a}
\eea
Also,
\begin{align}
	\lim_{n \to \infty} (
	\Pr[\{A^{(3 n^\xxi r_n)} \subset Z_n(\cU'_{n,A})\}
	\cap \{ (\partial A \cup Q_n ) \subset Z^o_n(\cU'_{n,A})\}
	\setminus \{A \subset Z_n(\cU'_{n,A})\}] )   = 0.
	\label{eqlem0220}
\end{align}
\end{lemm}
\begin{proof}
Denote the line segments making up $\partial A$ by $I_1,\ldots,I_{\kappa}$,
and for $ n >0$ and $1 \leq i \leq \kappa$ set $I_{n,i} := I_i \setminus
Q^-_n$.

Let $i,j,k \in \{1,\ldots,\kappa\}$ be such that
$i\neq j$ and the edges $I_i$ and $I_j$ are both incident
to $q_k$. 
Provided $n$ is large enough,
if $x \in I_{n,i}$ and
$y \in I_{n,j}$, then $\|x-y \| \ge 
(n^{2\zeta} r_n) \sin \alpha_k \geq 3 n^\zeta r_n$. 
Hence for all large enough $n$
the events $\{I_{n,1} \subset Z_n( \cU'_{n,A})\},
\ldots,\{I_{n,\kappa} \subset Z_n( \cU'_{n,A}) \}$
are mutually independent. Therefore
\bean
\Pr[ (\partial A \setminus Q^-_n)
\subset Z_n(\cU'_{n,A}) ] =
\prod_{i=1}^\kappa \Pr[ I_{n,i} \subset Z_n( \cU'_{n,A}) ],
\eean
and by Lemmas 
\ref{lemhalfd} and \ref{l:2PPs},
this converges to the right hand
side of (\ref{0429a}).

Now we prove (\ref{eqlem0220}).
For $n >0$, and
$i \in \{1,2,\ldots,  \kappa \} $, let $S_{n,i}$ denote the rectangular block
of dimensions $|I_{n,i} | \times 3 n^\xxi r_n$, consisting of all
points in $A$ at perpendicular distance at most $3 n^\xxi r_n$ from 
$I_{n,i}$.
Let $\partial_{{\rm side}}S_{n,i}$ denote the union of the two
`short' edges of $S_{n,i}$, i.e. the two edges 
bounding $S_{n,i}$ which are perpendicular to $I_{n,i}$.

Then for $n$ large, $A \setminus (A^{(3 n^\xxi r_n)} \cup Q_n) 
\subset \cup_{i=1}^\kappa S_{n,i}$,
and also
$\partial_{\rm side} S_{n,i}  \subset Q_n$
for $1 \leq i \leq \kappa$,
so that 
\begin{align*}
	\{ (A^{(3 n^\zeta r_n)} \cup \partial A ) \subset Z_n(\cU'_{n,A})\}
	\cap \{ Q_n
	\subset Z^o_n(\cU'_{n,A})\}
	\setminus \{ A \subset Z_n(\cU'_{n,A})  \} 
	~~~~~~~~~~~~~~~~~~~~~~~
	\\
	\subset \cup_{i=1}^\kappa 
	[\{I_{n,i} \subset Z_n(\cU'_{n,A} ) \}  
	\cup \{ 
	\partial_{{\rm side}} S_{n,i}  
	\subset Z_n^o(\cU'_{n,A} ) \}  
	\setminus \{S_{n,i} \subset Z_n(\cU'_{n,A})\}].
\end{align*}
For  $i \in \{1,\ldots,\kappa\}$, 
let $I'_{n,i}$ denote an interval of length $|I_{n,i}|$ in $\R$.
Using the Mapping theorem \cite[Theorem 5.1]{LP} and the translation and
rotation
invariance of Lebesgue measure one sees that
\begin{align*}
	\Pr[\{I_{n,i}  \subset Z_n( \cU'_{n,A})  \}  
	\cap \{\partial_{{\rm side}} S_{n,i} 
	\subset Z^o_n( \cU'_{n,A})  \}  
	\setminus \{S_{n,i} \subset Z_n( \cU'_{n,A})\}] 
	\nonumber \\
	=
	\Pr[ 
	\{(I'_{n,i} \times \{0\}) 
	\subset Z_n(\cU'_{n,\bH}) \}
	\cap \{ (\partial I'_{n,i})
	\times [0,3n^\zeta r_n] 
	\subset Z_n^o(\cU'_{n,\bH}) \}
	\nonumber \\
	\setminus 
	\{(I'_{n,i} \times [0,3n^\xxi r_n]) \subset Z_n(\cU'_{n,\bH})\} ] .
\end{align*}
By  Lemma \ref{lemhalfdint},
this probability tends to zero,
and hence
\eqref{eqlem0220}.
\end{proof}

\begin{proof}[Proof of Theorem \ref{th:general2d}]
Suppose $(r_n)_{n >0}$ satisfies 
(\ref{0220a}). 
We assert that provided $\xxi < 1/2$,
\bea
\lim_{n \to \infty} 
\Pr[ A^{(3 n^\xxi r_n)} \subset Z_n(\cU'_{n,A})
] =
\exp \Big( -  
{\bf 1}_{\{k=1\}}   \Big( \frac{(\E[Y])^2}{\E[Y^2]} \Big)|A| e^{-\beta} \Big). 
\label{0322a}
\eea
Indeed, setting $\lambda = n$ and $\delta (\lambda) =r_n$,
we have that if $k=1$ then \eqref{0315c} holds, while
if $k \geq 2$ then the left hand side of \eqref{0315c}
tends to $+\infty$. Also
$\Pr[ A^{(3 n^\xxi r_n)} \subset Z_n(\cU'_{n,A})]$ is  bounded
from below by
$\Pr[ A \subset Z_n(\cU'_{n,\R^2})]$, which
converges to the right hand side of \eqref{0322a}
by Lemmas \ref{lemHall} and \ref{l:2PPs}.
Also, if $k=1$, then
given $\eps >0$, for $n $ large 
the event
$\{ A^{(3 n^\xxi r_n)} \subset Z_n(\cU'_{n,A}) \}$ is contained in
$\{ A^{[\eps]} \subset Z_n(\cU'_{n,A})\}$,
so by Lemmas \ref{lemHall} and \ref{l:2PPs}  again
\begin{align*}
	\limsup_{n \to \infty}
	\Pr[ A^{(n^\xxi r_n)} \subset Z_n(\cU'_{n,A})] & \leq
	\limsup_{n \to \infty}
	\Pr[ A^{[\eps]} \subset Z_n(\cU'_{n,A})]
	\\
	& = \exp \Big( -  
	\Big( \frac{(\E[Y])^2}{\E[Y^2]}\Big)
	|A^{[\eps]}| e^{-\beta}\Big),
\end{align*}
and since $|A^{[\eps]} | \to |A|$
as $\eps \downarrow 0$, this gives us the assertion \eqref{0322a}.

Also, by (\ref{0429a}) and Lemma \ref{lemQ}, 
\bea
\lim_{n \to \infty} (
\Pr[\partial A \subset Z_n(\cU'_{n,A}) 
])  =
\exp \Big(- \Big( \frac{ c_{2,k} \E[Y] }{\sqrt{ \E[Y^2]} } \Big)
|\partial A|  e^{-\beta/2} \Big).
\label{0510a}
\eea

Note $\Pr[\{( Q_n \cup \partial A ) \subset Z_n(\cU'_{n,A})\}
\setminus
\{( Q_n \cup \partial A ) \subset Z^o_n(\cU'_{n,A})\}]=0$.
Therefore using (\ref{eqlem0220})  
followed by  Lemma  \ref{lemQ}, and then Lemma \ref{l:2PPs},
we obtain that
\begin{align}
	\lim_{n \to \infty} \Pr[ A \subset Z_n(\cU'_{n,A}) ]  
	& = \lim_{n \to \infty} \Pr[ (A^{(3 n^\xxi r_n)}  \cup \partial A \cup
	Q_n ) \subset Z_n(\cU'_{n,A}) ]  
	\nonumber \\
	& = \lim_{n \to \infty} \Pr[ (A^{(3 n^\xxi r_n)}  \cup \partial A)
	\subset Z_n(\cU'_{n,A}) ],
	\label{0510b}
\end{align}
provided these limits exist.

The events $\{\partial A \subset Z_n(\cU'_{n,A}) \} $  and
$\{A^{(3 n^\xxi r_n)} \subset Z_n(\cU'_{n,A})\} $ are
independent since the first of
these events is determined by the configuration of
Poisson points with projection onto $\R^2$
distant at most $n^\zeta r_n$ from $\partial A$,
while the second event is determined by the Poisson points 
with projection distant at least $2 n^\zeta r_n$ from $\partial A$. 
Therefore the limit in (\ref{0510b}) does
indeed exist, and is the product of the limits arising  in \eqref{0322a}
and \eqref{0510a}. By Lemma \ref{l:2PPs} we get the same limit
for $\Pr[A \subset Z_n(\cU_{n,A})]$ as for
$\Pr[A \subset Z_n(\cU'_{n,A})]$.
This  
gives us
\eqref{0128a}.
\end{proof}


\subsection{{\bf First steps toward proving Theorem \ref{th:generaldhi}}}

\label{secfirststeps}

In this subsection we assume that $d \geq 2$
and  $ \partial A \in C^2$.
Let $k \in \N, \beta \in \R $, and
assume   that $(r_n)_{n >0}$ satisfies  
\eqref{e:rn} and hence also \eqref{e:trho} and \eqref{rtd+}.

\textcolor{\blue}{We shall now carry out the strategy the polytopal
approximation strategy that was outlined in Subsection \ref{ss:strategy}.
We shall approximate $A$ by a polytopal set $\Gamma_n$
with faces of width $O(n^{9 \zeta} r_n)$, where 
$\zeta >0$ is a fixed small positive constant.
Thus the faces of $\Gamma_n$ will be small on a macroscopic scale,
but large compared to $r_n$. We shall describe how to reassemble
our Poisson process in regions near the faces of $\Gamma_n$ to get a
Poisson process in the half-space $\bH$, so that we can
apply the results from subsection \ref{ss:half-space}.}

Given any 
$n >0$, Borel
$D \subset \R^d$,
and any point process $\X$ in $\R^d \times \R_+$, 
define $F_{n}(D,\X)$ to be
the event that
$D$ is  `fully $k$-covered' by  the SPBM
based on $\X$ 
with radii scaled by
$r_n$, that is,
\begin{align}
F_{n}(D,\X) := \{D \subset Z_n(\X ) \},  
\label{Ftdef01}
\end{align}
where $Z_n(\cdot)$ was defined at \eqref{e:occ}.
Also, given  $r>0$,
define the `covering number'
\bea
\kappa(D,r): = \min \{m  \in \N: \exists x_1,\ldots,x_m \in D
~{\rm with} ~ D \subset \cup_{i=1}^m B(x_i,r)
\}.
\label{covnumdef}
\eea

Given $x \in \partial A$ we can
express $\partial A$
locally in a neighbourhood of $x$,
after a rotation, as the graph of a $C^2$ function from
$\R^{d-1}$ to $\R$ with all of its partial derivatives taking the value
zero at $x$.
We shall approximate to that function
by the graph of a piecewise affine function 
(in $d=2$, a piecewise linear function).

For each $x \in \partial A$, we can find an open neighbourhood
$\NN_x$ of $x$, a number $r(x) >0$ such that
$B(x, 3r(x)) \subset \NN_x$ and  a rotation $\theta_x$ about $x$
such that
$\theta_x(\partial A \cap \NN_x)$ is the graph of
a real-valued $C^2$ function $f$ defined on an open ball 
$D \subset \R^{d-1}$, with
$\langle f'(x),e\rangle =0 $ and
$\langle f'(u),e\rangle \leq 1/9$ for
all $u \in D$ and
all unit vectors $e$ in $\R^{d-1}$, where $\langle \cdot,\cdot \rangle$
denotes the Euclidean inner product 
in $\R^{d-1}$ and $f'(u):= (\partial_1 f(u),\partial_2f(u), \ldots, \partial_{d-1} f(u) )$ is the derivative of $f$ at $u$.
Moreover, by taking a smaller neighbourhood if necessary, 
we can also assume that there exist $\eps>0$ and
$a  \in \R$ such that 
$f(u) \in [a+ \eps,a+ 2 \eps]$ for all $u \in D$
and also  $\theta_x(A) \cap (D \times [a,a + 3 \eps])
= \{(u,z): u \in D, a \leq z \leq f(u)\}$.

By a compactness argument, we can and do take a finite collection
of points $x_1,\ldots, x_J \in \partial A$ such that
\bea
\partial A \subset  \cup_{j=1}^J B(x_j,r(x_j)). 
\label{bycompactness}
\eea
Then there are  constants $\eps_j >0$, and
rigid motions $\theta_j, 1 \leq j \leq J$, 
such that for each $j$ 
the set $\theta_j(\partial A \cap \NN_{x_j}) $ is
the graph of a $C^2 $ function $f_j$ 
defined on a ball $I_j$ in $\R^{d-1}$, with
$\langle f'_j(u), e \rangle \leq 1/9$ for all $u \in I_j$ and
all unit vectors
$e \in \R^{d-1}$, and also with $\eps_j \leq f_j(u) \leq 2 \eps_j $
for all $u \in I_j$ and $\theta_j(A) \cap (I_j \times [0,3\eps_j]) =
\{(u,z):u \in I_j, 0 \leq z \leq f(u)\}$.

Later we shall refer to each tuple
$(x_j,r(x_j),\theta_j, f_j)$, $1 \leq j \leq J$, as 
a {\em chart}. More loosely we shall also refer to each  set
$B(x_j, r(x_j))$ as a chart.

Let $\Gamma \subset \partial A$ be
a closed set such that $\Gamma \subset  B(x_j,r(x_j))$
for some $j \in \{1,\ldots,J\}$, and such that
$\kappa(\partial \Gamma,r) = O(r^{2-d})$ as $r \downarrow 0$,
where
we set
$\partial \Gamma: = \Gamma \cap \overline{\partial A \setminus \Gamma} $,
the boundary of $\Gamma$ relative to
$\partial A$.
To simplify notation we shall assume that 
$\Gamma \subset B(x_1,r(x_1))$, and moreover that $\theta_1$ is the identity map. 
Then $\Gamma= \{(u,f_1(u)): u \in U\}$ for some bounded set $U \subset \R^{d-1}$.
Also,
writing $\phi(\cdot)$ for $f_1(\cdot)$ from now on, we assume
\bea
\phi(U) \subset [\eps_1,2 \eps_1]
\label{0901b}
\eea
and
\bea
A \cap (U \times [0,3 \eps_1]) = \{(u,z): u \in U, 0 \leq z \leq \phi(u) \}.
\label{0901a}
\eea

Note that for any $u,v \in U$, by the Mean
Value theorem we have for some $w \in [ u,v]$ that
\bea
|\phi(v) - \phi(u) | = | \langle v-u, \phi'(w)\rangle | \leq
(1/9) \|v-u\|.
\label{philip}
\eea

Choose (and keep fixed for the rest of this paper) a constant 
$\xxi$ 
with
\bea
0 < \xxi <
1/(198d(18+d)).
\label{eqgamma}
\eea
Henceforth
we shall use this choice of $\xxi$ in the definition of
$\cU'_{n,D}$ at \eqref{e:U'}. We shall use $n^{c\xxi} r_n$,
for various choices of constant $c$, to provide various different
length scales.

We shall approximate to $\Gamma$ by a polytopal 
surface $\Gamma_n$ with
face diameters
that are $\Theta(n^{9 \xxi} r_n)$,  taking all the faces of
$\Gamma_n$ to be $(d-1)$-dimensional simplices.
Later we shall fit together a finite number of surfaces like $\Gamma$ to 
make up the whole of $\partial A$.

For the polytopal approximation,
divide $\R^{d-1}$ into 
hypercubes of dimension $d-1$ and side $n^{9 \xxi} r_n$,
and divide each of these hypercubes
into $(d-1)!$ simplices (we take these simplices to be closed).
Let $U_n$ be the union  of all those
simplices in the resulting tessellation of 
$\R^{d-1}$ into simplices,  that are
contained within $U$,
and let $U_n^-$ be the
union of those simplices in the tessellation which are contained
within $U^{(n^{10 \xxi} r_n)}$,
where for $r>0$ we
set $U^{(r)}$ 
to be the set of $x \in U$ 
at a Euclidean distance 
more than $r$ from $\R^{d-1} \setminus U$.
If $d=2$, the simplices
are just intervals. See Figure \ref{f:Un}.

\begin{figure}[!h]
	\label{fig0}
	\center
	\includegraphics[width=8cm]{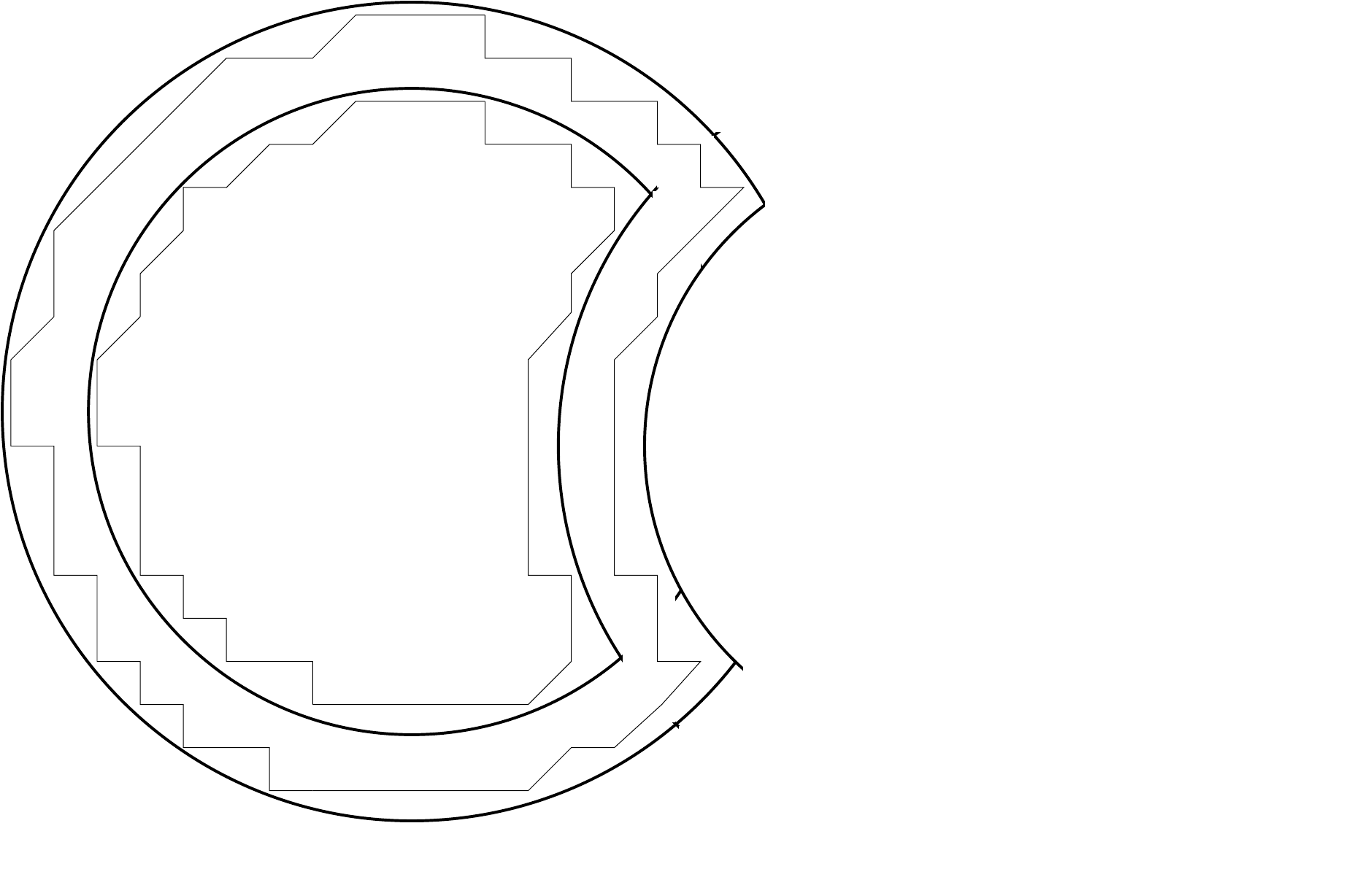} 
	\caption{\label{f:Un} Example in $d=3$.
		The outer crescent-shaped region is $U$,
		while the inner crescent is $U^{(10 \xxi r_n)}$
		The outer polygon is $U_n$, while
		the inner polygon is $U_n^-$
	}
\end{figure}

Let $\sigma^-(n)$
denote the number of simplices making up
$U_n^-$.
Choose $n_0 >0$ such that $\sigma^-(n) >0$ for all $n \geq n_0$.

Let  $\psi_n: U_n \to \R$ be the function that is
affine on each of the simplices
making up $U_n$, and agrees with the function $\phi$ on each
of the vertices of these simplices.
\textcolor{\blue}{
Such a function exists
because, if one such simplex has vertices
labelled $v_1,v_2,\ldots,v_d$ say, then $v_2-v_1, \ldots, v_d- v_1$
are linearly independent in $\R^{d-1}$ (in general it would not exist
if we used cubes instead of simplices because a cube would have too many vertices)}.
Our approximating  surface (or polygonal line if $d=2$)
will be defined by
$\Gamma_n := \{(u, \psi_n(u)-K n^{18 \xxi}r_n^2): u \in U^-_n\}$,
as depicted in Figure~\ref{f:Tri} for the case $d=3$, with 
the constant $K$ given by the following lemma.
This lemma uses Taylor expansion to show that $\psi_n$
a good approximation to $\phi$. 

\begin{lemm}[Polytopal approximation]
	\label{lemtaylor3}
	Set 
	$K :=
	\sup_{n \geq n_0, u \in U_n} ( (n^{9 \xxi} r_n)^{-2} 
	|\phi(u) - \psi_n(u) |)$.
	Then $K < \infty$. 
\end{lemm}
\begin{proof}
	See \cite[Lemma 7.18]{P23}. The notation $t$ there is equivalent
	to $n$ here. the length scale of $n^{9\xi} r_n$ here plays the
	role of $t^{-\gamma}$ there.
\end{proof}
We now subtract a constant from $\psi_n$ to obtain a
piecewise affine function
$\phi_n$ that approximates $\phi$ from below.
For $n \geq n_0$ and $u \in U_n$, define $\phi_n(u) := \psi_n(u) - 
K n^{18 \xxi}r_n^2$, with $K$ given by  
Lemma \ref{lemtaylor3}. Then
for all $n \geq n_0, u \in U_n$ we have 
$|\psi_n(u)-\phi(u)|\leq Kn^{18 \xxi} r_n^2$ so that
\bea 
\phi_n(u) \leq  \phi(u) \leq \phi_n(u) + 2K n^{18 \xxi} r_n^2.
\label{0126a}
\eea
Define the set $\Gamma_n: =
\{(u,\phi_n(u)): u \in U_n^-\}$.
We refer to each $(d-1)$-dimensional face of $\Gamma_n$
(given by the graph of $\phi_n$ restricted to one
of the simplices in our triangulation of $\R^{d-1}$)
as simply a {\em face} of $\Gamma_n$.
Denote these faces
of $\Gamma_n$ by $H_{n,1}, \ldots, H_{n,\sigma^-(n)}$.
The number of faces, $\sigma^-(n)$,  
is $\Theta((n^{9 \xxi} r_n)^{1-d})$ as $n \to \infty$.
The perimeter (i.e., the $(d-2)$-dimensional Hausdorff measure of the boundary)
of each individual face is $ \Theta((n^{9 \xxi} r_n)^{d-2})$.
For $1 \leq i \leq \sigma^-(n)$, let $\partial_{d-2}H_{n,i}$
denote the boundary  of $H_{n,i}$, which is the image under
the mapping $u \mapsto (u,\phi_n(u))$,
of the boundary  of the  simplex in $\R^{d-1}$
that is the pre-image of $H_{n,i}$ under that mapping.

\begin{figure}[h]
	\center
	\begin{tikzpicture}
		\node at (0,0) {\includegraphics[width=6cm,trim=80 40 70 35, clip]{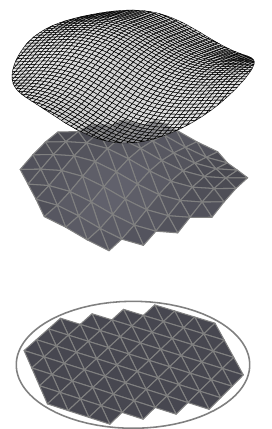}};
		\draw[<->] (-2.9,1) -- (-2.9,3) node [midway, left] {$K n^{18 \xxi}r_n^2$};
		\draw[->] (0,-1.8) -- (0,-0.8);
	\end{tikzpicture}
	\caption{\label{f:Tri}
		\textcolor{\blue}{We show here
		the set $\Gamma_n$ below a portion $\Gamma$ of the boundary of $A$,
		when $d=3$.
		The bottom part of the diagram shows the simplices (triangles) making up $U_n^-$,
		which are used to construct the triangulated surface above it.
		The faces of the triangulated surface $\Gamma_n$
		are the $H_{n,i}$.}
	}
\end{figure}

For $n \geq n_0$,
define subsets $A_n,A_n^-,
\tA_n,
A_n^{*},
\tA_n^*$
of
$ \R^d$ (illustrated in Figure \ref{f:Ans}) by
\bea
A_n := \{(u,z): u \in U_n, 0
\leq z \leq \phi(u)\}, 
~~~
\tA_n := \{(u,z): u \in U_n, 
0 \leq z \leq \phi_n(u)\},
\nonumber \\
A_n^- := \{(u,z): u \in U_n^-, 
0 \leq z \leq \phi(u)\}, 
~~~~~ ~~~~~ ~~~~~
~~~~~
~~~~~
~~~~~
~~~~~
~
\nonumber \\
A_n^* := \{(u,z): u \in U_n^-, \phi_n(u) - (3/2) 
n^\xxi r_n \leq z \leq \phi(u)\},
~~~~~ 
~~~~~ 
~~~~~ 
~~
~~
\nonumber \\
\tA_n^* := \{(u,z): u \in U_n^-, \phi_n(u) - (3/2) n^\xxi r_n \leq z
\leq \phi_n(u)\}.
~~~~~~~~ ~~~~~~~~
~
\label{Pdtdef}
\eea
Thus $A_n$ is a `thick slice' of $A$  near the boundary region $\Gamma$,
$\tA_n$ is an approximating region
having $\Gamma_n$ as  its upper boundary,
and $A_n^{*}$,
$\tA_n^*$
are `thin slices' of $A$ also 
having $\Gamma$, respectively $\Gamma_n$, as upper boundary.
By (\ref{0126a}), (\ref{0901b}) and  (\ref{0901a}),
$\tA_n^* \subset A_n^* \subset A_n^- \subset A_n \subset A$,
and $\tA_n^* \subset \tA_n \subset A_n$.

\begin{figure}[!h]
	\center
	\begin{tikzpicture}
		\node at (0,0) {\includegraphics[width=8cm]{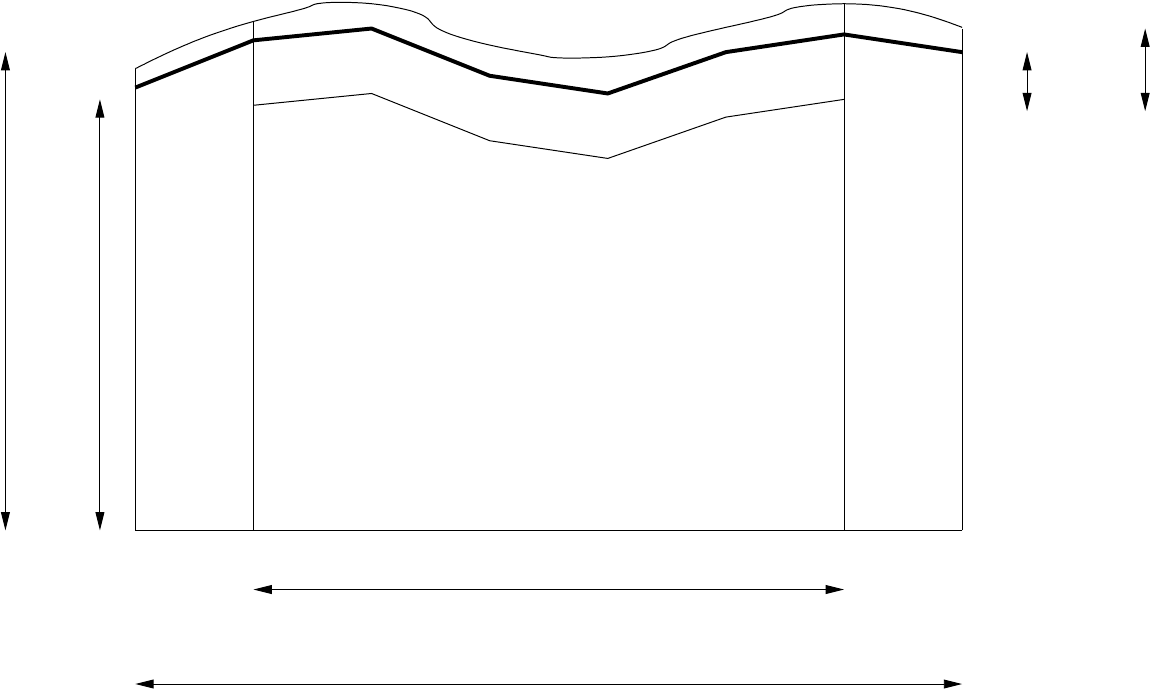}}; 
		\node at (-4.45,0.5) {\scriptsize $A_n, A_n^-$};
		\node at (-3.55,0.5) {\scriptsize $W_n$};
		\node at (3.45,1.8) {\scriptsize $W_n^*$};
		\node at (4.2,1.85) {\scriptsize $A_n^*$};
		\node at (0.6,-1.5) {\scriptsize $A_n^-, A_n^*, W_n^*$};
		\node at (-0.5,-2.2) {\scriptsize $A_n, W_n$};
	\end{tikzpicture}
	\caption{\label{f:Ans} 
		When $d=2$ the sets
		$A_n, \tA_n, A_n^-, A_n^*, \tA_n^*$
		are approximately rectangular but
		with a curved upper boundary for $A_n, A_n^-$ and $ A_n^*$,
		a polygonal upper boundary for $\tA_n^*$ and $\tA_n$,
		and a polygonal lower boundary for $A_n^*$ and $\tA_n^*$.
		\textcolor{\blue}{
		The `faces' $H_{n,i}$ are just line segments since $d=2$,
		and are the segments of the bold polygonal line.}
	}
\end{figure}

The rest of this subsection, and the next subsection,
are devoted to proving the following
intermediate step towards a proof of Theorem \ref{th:generaldhi}.
Recall the definition of $c_{d,k}$ at \eqref{cd1def}, and define
\begin{align}
	c_{d,k,Y} := \frac{c_{d,k} (\E[Y^{d-1}])^{d-1}}{(\E[Y^d])^{d-2+1/d}}.
	\label{e:cdY}
\end{align}

\begin{prop}[Limiting coverage probability for approximating
	polytopal shell]
	\label{lemsurf}
	It is the case that
	$
	\lim_{n \to \infty} \Pr[F_{n}(\tA_n^*, \cU'_{n,\tA_{n}}) ] =
	\exp (-  c_{d,k,Y} |\Gamma|  e^{-  \beta/2} ), 
	$
	where $|\Gamma|$ denotes the $(d-1)$-dimensional
	Hausdorff measure of $\Gamma$.
\end{prop}

The following corollary of Lemma \ref{lemtaylor3}
is a first step towards proving this.

\begin{lemm}
	\label{corotaylor}
	(a) It is the case that
	$|A_n \setminus \tA_n| = O(n^{18 \xxi}r_n^2)$ as  
	$n \to \infty$.
	
	(b)
	Let $K$ be as given in Lemma \ref{lemtaylor3}, and
	fix $L>0$. Then
	for all $n \geq n_0$ and 
	$x \in U_n^{(Lr_n)} \times \R$,
	$|B(x,L r_n) \cap A_n \setminus \tA_n| \leq 2K \omega_{d-1}
	L^{d-1}
	n^{18 \xxi} r_n^{d+1}$. 
\end{lemm}
\begin{proof}
	Since $|A_n \setminus \tA_n | = \int_{U_n} (\phi(u) - \phi_n(u)) du$, 
	where this is a $(d-1)$-dimensional Lebesgue integral,
	part (a)
	comes from 
	(\ref{0126a}).
	
	For (b), let $x \in U_n^{(L r_n)} \times \R$, and
	let $u \in U_n^{(L r_n)}$ be the projection of $x $ onto
	the first $d-1$ coordinates. Then if $y \in B(x,Lr_n) \cap
	A_n \setminus \tA_n$, we have $y = (v,s)$ with
	$\|v-u\| \leq L r_n$ and  $\phi_n(v) < s \leq \phi(v)$.
	Therefore using (\ref{0126a})  yields
	$$
	|B(x,Lr_n) \cap A_n \setminus \tA_n| \leq \int_{B_{(d-1)}(u,Lr_n)}
	(\phi(v) - \phi_n(v)) dv \leq 2K \omega_{d-1} 
	n^{18 \xxi} r_n^2 L^{d-1} r_n^{d-1},
	$$
	where the integral is a $(d-1)$-dimensional Lebesgue integral.
	This gives part (b). 
\end{proof}

The next lemma says that small balls centred in $A_n^*$
have almost half of their volume in $\tA_n$.

\begin{lemm}
	\label{lemhs}
	Let $\eps \in (0,1)$.
	Then
	for all large enough $n$, all $x \in A_n^*$,
	and all $s \in [\eps r_n,r_n/\eps]$,
	we have
	$|B(x,s)  \cap  \tA_n|> (1- \eps) (\omega_d/2) s^d$.
\end{lemm}
\begin{proof}
	For all large enough $n$, all
	$x \in A_n^*$ and $s \in [\eps r_n,r_n/\eps]$,
	we have $B(x,s )\cap A \subset A_n$,
	so $B(x,s) \cap A_n = B(x,s) \cap A$, 
	and hence by Lemma \ref{lemgeom1a} 
	and Lemma \ref{corotaylor}(b),
	\bea
	|B(x,s) \cap  \tA_n| 
	& = & |B(x,s) \cap A_n | - 
	|B(x,s ) \cap A_n \setminus \tA_n  | 
	\nonumber 	\\
	& \geq & (1 - \eps/2) (\omega_d /2) s^d   
	- O( 
	n^{18 \xxi}
	r_n^{d+1}
	). 
	\nonumber 
	\eea
	Since $n^{18 \xxi} r_n \to  0$ by \eqref{e:rn}
	and \eqref{eqgamma}, this gives us the result.
\end{proof}

Recall that $\bH$ and $\cU'_{n,D}$ were defined in Section \ref{subsecprelims}
and at \eqref{e:U'},
respectively.
The next lemma provides a bound on the probability
that a region of diameter $O(r_n)$ within $A$ or $A_n^*$
is not fully covered. This
will be used for dealing 
with `exceptional' regions such as those near the boundaries
of faces in the polytopal approximation.

\begin{lemm}
	\label{lemcov3}
	Let $\eps \in (0,1)$, $K_1 >0$. Then as $n \to \infty$, 
	\bea
	\sup_{z \in \R^d}
	\Pr[ F_{n}(B(z,K_1r_n) \cap A_n^*, \cU'_{n,\tA_n}
	)^c ]
	= O(n^{\eps - (d-1)/d}),
	\label{0816b}
	\\
	\sup_{z \in \R^d} \Pr[ F_{n}(B(z,K_1r_n) \cap \bH 
	,\cU'_{n ,\bH}
	)^c ]
	= O(n^{\eps - (d-1)/d}),
	\label{0816a}
	\eea 
	and
	\bea
	\sup_{z \in \R^d}
	\Pr[ F_{n}(B(z,K_1r_n) \cap A
	, \cU'_{n,A}
	)^c ]
	= O(n^{\eps - (d-1)/d}).
	\label{0816c}
	\eea
\end{lemm}
\begin{proof}
	Let $\delta \in (0,1/2)$ with
	$(1-\delta)^{d+1} \E[Y^d{\bf 1}_{\{\delta \leq Y \leq 1/\delta\}}] >
	(1-\eps/2) \E[Y^d]$.
	For any point set $\X \subset \R^d \times [\delta,\infty)$, define
	$\tilde{Z}_n(\X)$ similarly to $Z_n(\X)$ but with slightly smaller
	balls, namely
	$$
	\tilde{Z}_n(\X) : =  
	\{y \in \R^d: \# (\{(x,s) \in \X: y \in B(x,(1-\delta)r_ns)\}) \geq k\},
	$$
	and note that for any 
	$y \in \tilde{Z}_n(\X)$
	we have $B(y,\delta^2 r_n) \subset Z_n(\X)$. 
	
	There is a constant $\ell$ independent
	of $n$ such that given $z \in \R^d$ and given $n$,
	we can (and do) choose $x_{n,1},\ldots,x_{n,\ell} \in A_n^*$  with
	$B(z,K_1r_n) \cap A^*_n \subset 
	\cup_{i=1}^\ell B(x_{n,i}, \delta^2 r_n)$.  
	Then for all $n $ large enough, and for $1 \leq i \leq \ell$,
	using Lemma \ref{lemhs} in the third line below we have
	\begin{align*}
		1- \Pr[&F_{n}(B(x_{n,i},\delta^2 r_n) \cap A^*_n, \cU'_{n,\tA_n})] 
		\leq \Pr[ x_{n,i} \notin
		\tilde{Z}_n 
			(\cU'_{n,\tA_n} \cap (\tA_n \times [\delta,1/\delta] ))
		] 
		\\
		&\leq
		k (n \omega_d (r_n/\delta)^d )^{k-1}
		\exp \Big( - n
		\int_{[\delta ,1/\delta]}
		|\tA_n \cap B(x_i,(1-\delta)r_n t)|\mu_Y(dt) \Big)
		\\
		&\leq
		k (n \omega_d (r_n/\delta)^d )^{k-1}
		\exp \Big( - n
		\int_{[\delta ,1/\delta]}
		(1-\delta)^{d+1} (\omega_d/2) 
		r_n^d t^d\mu_Y(dt) \Big)
		\\
		&\leq
		k (n \omega_d (r_n/\delta)^d )^{k-1}
		\exp \Big( - ( \omega_d/2) 
		n r_n^d
		(1- \eps/2) \E[Y^d]
		\Big).
	\end{align*}
	Therefore by
	\eqref{rtd+},
	for $n $ large
	\begin{align*}
		\Pr[\{F_{n}(B(x_{n,i},\delta^2 r_n)
		\cap A_n^*,\cU'_{n,\tA_n})\}^c] 
		\leq n^{\eps - (d-1)/d}. 
	\end{align*}
	Taking the union of the above events for $1 \leq i \leq \ell$,
	and applying the union  bound, gives us \eqref{0816b}.
	The proofs of (\ref{0816a}) and (\ref{0816c})  are
	similar.
\end{proof}

Let $\partial_{d-2} \Gamma_n  := \cup_{i=1}^{\sigma^-(t)} \partial_{d-2}
H_{n,i}$,
the union of all $(d-2)$-dimensional faces in the boundaries 
of the faces making up
$\Gamma_n$
(the $H_{n,i}$ were defined just after (\ref{0126a})).
Recall from (\ref{eqgamma}) that
$\xxi \in (0,1/(198d(18+d)))$.
Given $n >0$, 
define the set $Q^+_n \subset \R^d$  by 
\bea
Q^+_n  :=   
(\partial_{d-2} \Gamma_n \oplus B(o,8dn^{4\xxi } r_n))
\cap A_n^{*}.
\label{Qplusdef}
\eea
Thus $Q_n^+$ is
a region near the corners of our polygon approximating
$\partial A$ (if $d=2$) or near the boundaries of the
faces of our polytopal surface approximating
$\partial A$ (if $d  \geq 3$). In the next lemma 
we show that $Q_n^+$ is  fully $k$-covered with high probability. 
\begin{lemm}
	\label{lemcorn3}
	It is the case that
	$\Pr[F_{n}(Q^+_n,\cU'_{n,\tA_n})] \to 1$ 
	as $n \to \infty$.
\end{lemm}
\begin{proof}
	Let $\eps := \xxi/2$.
	For each face $H_{n,i}$ of $\Gamma_n$,  $1 \leq i \leq 
	\sigma^-(n)$, we claim that we can
	take $x_{i,1},\ldots, x_{i,k_{n,i}} \in \R^d$
	with $\max_{1 \leq i \leq \sigma^-(n)} k_{n,i}
	= O(n^{9 d \xxi -10 \xxi} )$,
	such that
	\bea
	(\partial_{d-2} H_{n,i})
	\oplus B(o, 8 dn^{4 \xxi}r_n) 
	\subset \cup_{j=1}^{k_{n,i}} 
	B(x_{i,j},r_n).
	\label{0206a}
	\eea
	Indeed, we can cover $\partial_{d-2} H_{n,i}$
	by $O(n^{5 \xxi (d-2)})$ balls of radius
	$n^{4\xxi} r_n$, denoted $B_{i,\ell}^{(0)}$ say.
	Replace each  ball $B_{i,\ell}^{(0)}$
	with a 
	ball $B'_{i,\ell}$ with the same centre as
	$B_{i,\ell}^{(0)}$ and with 
	radius  $ 9d n^{4 \xxi }r_n$. Then cover $B'_{i,\ell} $
	by $O((n^{4 \xxi} )^{d})$ balls of radius $r_n$.
	Every point in the set on the left side of (\ref{0206a})
	lies in one of these balls of radius $r_n$, and the claim
	follows.
	
	By the definitions of $Q_n^+$ and $\partial_{d-1}\Gamma_n$,
	and then \eqref{0206a}, we
	have 
	$$
	Q_n^+ \subset \cup_{i=1}^{\sigma^-(n)}((
	\partial_{d-2} H_{n,i}) \oplus B(o,8dn^{4\xxi}r_n))
	%
	\subset \cup_{i=1}^{\sigma^-(n)} \cup_{j=1}^{k_{n,i}}
	B(x_{i,j}, r_n),$$ so by the union bound,
	\bean
	\Pr [ F_{n} ( Q^+_n ,\cU'_{n,\tA_n}) ^c]
	\leq 
	\sum_{i=1}^{\sigma^-(n)} \sum_{j=1}^{k_{n,i}}
	\Pr[
	F_{n}( B(x_{i,j},  r_n) \cap A_n^*,
	\cU'_{n,\tA_n}
	)^c]. 
	\eean
	Thus using Lemma \ref{lemcov3} 
	and the fact that $\sigma^-(n) = O((n^{9 \xxi }r_n)^{1-d})$, 
	we obtain that
	\begin{align*}
		\Pr [ F_{n} ( Q^+_n ,\cU'_{n,\tA_n}) ^c]
		&
		= O((n^{9 \xxi } r_n)^{1-d}n^{(9d-10)\xxi} 
		n^{\eps -1 + 1/d} ) \\
		&
		=O(n^{ - \xxi} r_n^{1-d} 
		n^{\eps - 1 +1/d}) = O(n^{-\xxi/2}).
	\end{align*}
	Thus
	$
	\Pr [ F_{n} ( Q^+_n ,\cU'_{n,\tA_n}) ^c] $ tends  to zero.
\end{proof}

\subsection{
	Induced coverage process and
	proof of Proposition \ref{lemsurf}}
\label{secinduced}

\textcolor{\blue}{In this subsection
we shall conclude the proof of Proposition \ref{lemsurf},
concerning the limiting probability of covering
an approximating polytopal shell. We shall do so} 
by means of
a device we refer to as the {\em induced coverage process.}
This is obtained by taking the parts of $\tA_n$
near the flat parts of $\Gamma_n$, along with any  Poisson
points therein, and rearranging them into a flat
region of macroscopic size.

Partition each face $H_{n,i}$, $1 \leq i \leq \sigma^-(n)$ into
a collection of
$(d-1)$-dimensional hypercubes of side 
$n^{3\xxi}r_n$
contained in
$H_{n,i}$ and distant more than $n^{4 \xxi}r_n$ from
$\partial_{d-2}H_{n,i}$, 
together with a `border region' contained within $\partial_{d-2} H_{n,i}
\oplus B(o,2 n^{4 \xxi }r_n)$.
Let $\tar_n$ be
the union (over all faces)  of the boundaries of the $(d-1)$-dimensional
hypercubes in this partition (see Figure \ref{f:plaid}: the $P$ stands
for `plaid').
Set
\begin{align}
	\tar^+_n := [ \tar_n \oplus B(o, 9 n^{ \xxi} r_n) ] \cap \tA_n^{*}.
	\label{e:Tn+}
\end{align}

\begin{figure}[!h]
	\centering
	\includegraphics[width=8cm]{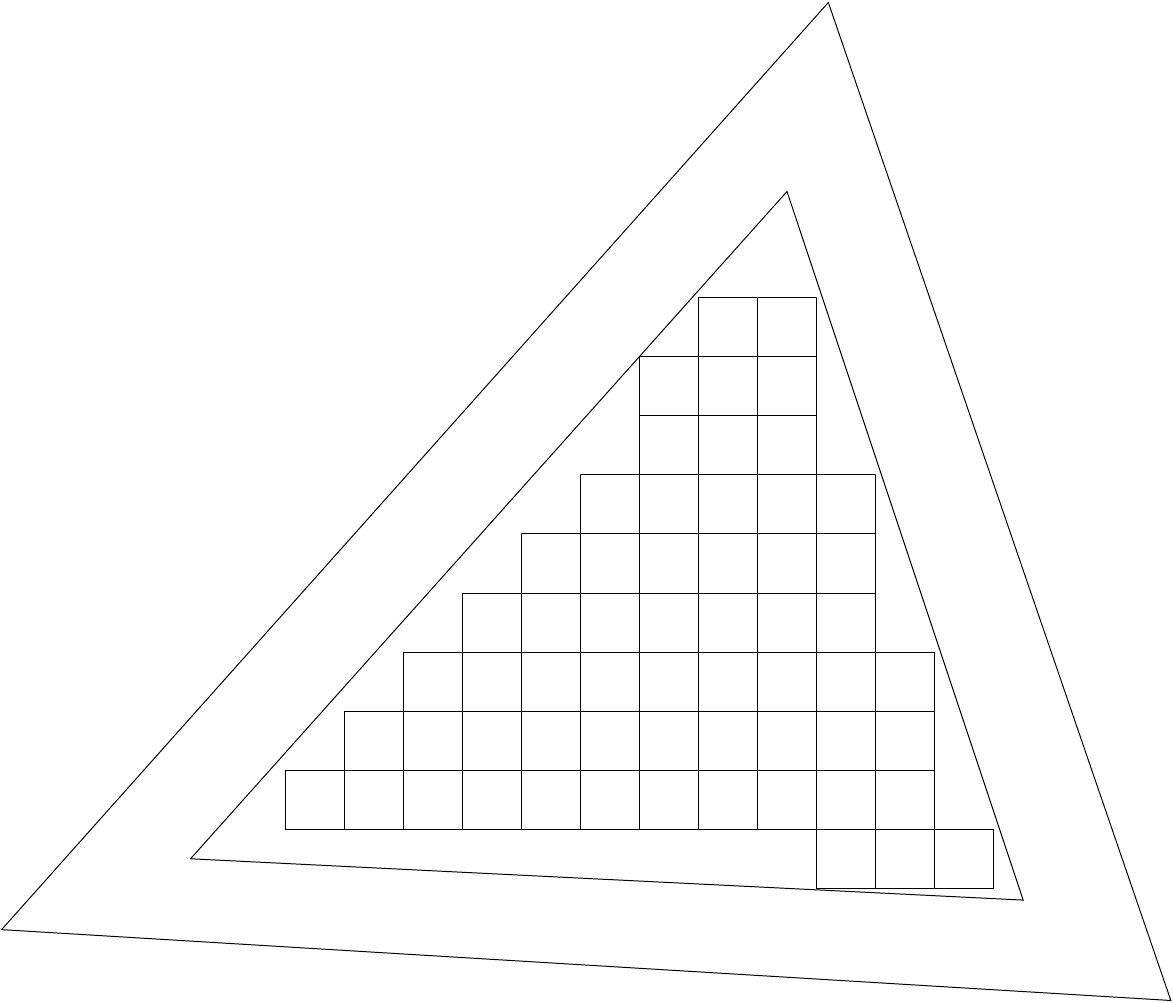} 
	\caption{\label{f:plaid} Part of the `plaid' region $\tar_n$ when $d=3$.
		The outer triangle represents one face $H_{n,i}$,
		and the part of $\tar_n$ within $H_{n,i}$ is
		given by the union of the boundaries of the
		squares. \textcolor{\blue}{The squares themselves are some of the
		$I_{n,i}^+$.} The outer triangle has sides of length
		$\Theta(n^{9 \xxi} r_n)$, while the squares
		have sides of length $n^{3 \xxi} r_n$. The region
		between the two triangles
		has thickness $n^{4 \xxi} r_n$ and 
		is contained in $Q_n^+$.}
\end{figure}

\begin{figure}[!h]
	\centering
	\includegraphics[width=8cm]{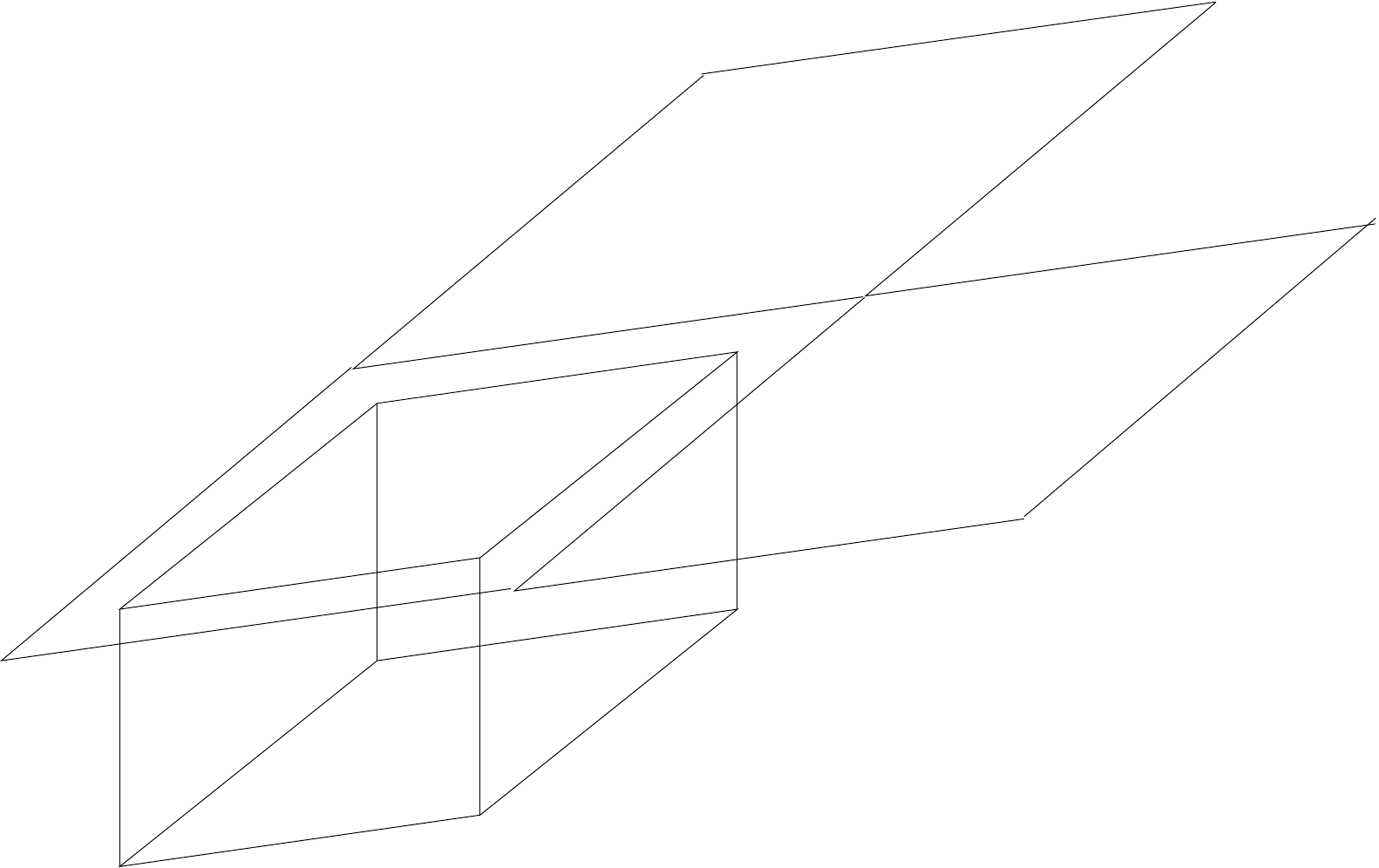} 
	\caption{\label{f:SnIn}Diagram for $d=3$ showing
	three of the squares $I_{n,i}^+$ (as in Figure \ref{f:plaid} but
	now shown with 3-dimensional perspective), along with one of
	the squares $I_{n,i}$ (a slightly smaller square) and
	the corresponding cube $S_{n,i}$.}
\end{figure}

Enumerate the $(d-1)$-dimensional hypercubes in the above subdivision
of the faces $H_{n,i}, 1 \leq i \leq \sigma^-(n)$, as
$I^+_{n,1},\ldots, I^+_{n,\lambda(n)}$. For $1 \leq i \leq \lambda(n)$
let 
\linebreak
$I_{n,i}:= I^+_{n,i} \setminus (\tar_n \oplus B(o, n^\xxi r_n))$, which is a 
$(d-1)$-dimensional hypercube 
of side length $(n^{3\xxi } - 2n^\xxi) r_n$
with the same centre and orientation as $I_{n,i}^+$.
Writing $|\cdot|$ below for $(d-1)$-dimensional 
Hausdorff measure,
we claim that the total $(d-1)$-dimensional Hausdorff measure of these
$(d-1)$-dimensional hypercubes
satisfies
\bea
\lim_{n \to \infty}( | 
\cup_{i=1}^{\lambda(n)} I_{n,i} | )
= 
|\Gamma|.
\label{0912a}
\eea
Indeed,  for $1 \leq i \leq \lambda(n)$ we have
$|I_{n,i}|/|I_{n,i}^+| = ((n^{3 \xxi} - n^\xxi) /n^{3\xxi} )^{d-1}$,
which tends to one,
so the proportionate amount removed
near the boundaries of the $(d-1)$-dimensional hypercubes $I^+_{n,i}$ to
give $I_{n,i}$ vanishes.
Also  the `border region' of a face
$H_{n,i}$ that is not contained in any of the 
of the $I^+_{n,j}$s has $(d-1)$-dimensional Hausdorff measure 
that is $O((n^{9 \xxi} r_n)^{d-2} n^{4\xxi}r_n)$, so that
the total $(d-1)$-dimensional measure of 
the removed  regions near the boundaries of the faces is
$O((n^{9 \xxi} r_n)^{1-d} \times (n^{9\xxi} r_n)^{d-2} n^{4\xxi}r_n) =
O(n^{ -5 \xxi} )$, which tends to zero.
Thus the claim (\ref{0912a}) is justified.

For $1 \leq i \leq \lambda(n)$, let
$S_{n,i}^+$, respectively $S_{n,i}$,
be a $d$-dimensional cube of side $n^{3\xxi} r_n$,
respectively of side $(n^{3\xxi} -2 n^\xxi)r_n$, 
having $I^+_{n,i}$, respectively $I_{n,i}$, as its `upper face',
i.e. having $I_{n,i}^+$ (resp. $I_{n,i}$) as a face and lying
within $\tA_n$.
%
We shall verify in Lemma \ref{lemblox}
that $S_{n,1}^+,\ldots, S_{n,\lambda(n)}^+$ are disjoint.

Define a  region $D_n \subset \R^{d-1}$ that is approximately a rectilinear
hypercube with lower
left corner at the origin, and obtained as the union of $\lambda(n)$
disjoint $(d-1)$-dimensional 
hypercubes of  side $n^{3 \xxi} r_n$. We can and do arrange
that 
$D_n \subset [0,|\Gamma_n|^{1/(d-1)}+ n^{3\xxi}r_n]^{d-1}$ for each $n$,
and $|D_n| \to |\Gamma|$ as $t \to \infty$.
Define the flat slabs (or if $d=2$, flat strips)
\begin{align}
	\Sn:= D_n  \times [0, (n^{3\xxi} -2 n^\xxi) r_n]; ~~~~~~
	\Sn^+ := D_n  \times [0,  n^{3 \xxi} r_n],
	\label{e:Srhodef3d}
\end{align}
and denote the lower boundary of $\Sn$ (that is, the set
$D_n \times \{0\}$) by $L_n$.

Now choose   rigid motions $\theta_{n,i}$
of $\R^d$, $1 \leq i \leq \lambda(n)$, such that
under applications of these rigid motions the blocks $S^+_{n,i}$ 
are reassembled to form the slab $\Sn^+$, with the square face
$I^+_{n,i}$
of the $i$-th block transported to part of the lower boundary $L_n$  of $\Sn$.
In other words, choose the rigid motions so that the sets 
$ \theta_{n,i}(S^+_{n,i})$, 
$ 1 \leq i \leq \lambda(n)$,
have pairwise disjoint interiors and their union is 
$\Sn^+$, and also  $\theta_{n,i}(I^+_{n,i}) \subset L_n$ for 
$1 \leq i \leq \lambda(n)$.

Recall that $n_0$ was chosen shortly before \eqref{lemtaylor3}.
Given $n \geq n_0$ and $i \leq \lambda(n)$,
define $\tilde{\theta}_{n,i}: \R^{d+1}  \to \R^{d+1} $ 
by $\tilde{\theta}_{n,i}(x,r) = (\theta_{n,i}(x),r)$
for all $x \in \R^d, r \in \R$.
By the Restriction, Mapping and Superposition theorems (see e.g.\ \cite{LP}),
the point process $\tilde{\cU}_{n,\Sn^+} := \cup_{i=1}^{\lambda(n)} 
\tilde{\theta}_{n,i} 
(\cU_{n,\tA_n} \cap (S^+_{n,i} \times \R_+)) $ is 
a Poisson process in
$\Sn^+ \times \R_+$ 
with intensity measure $n{\rm Leb}_d \otimes \mu_Y$
(strictly speaking, with intensity measure $n{\rm Leb}_d|_{\Sn^+} 
\otimes \mu_Y$).

We extend $\tcU_{n,\Sn^+}$ to a Poisson process $\tcU_{n,\bH}$
in $\bH \times \R_+$, where
$ \bH := \R^{d-1} \times [0,\infty)$, as follows.
Let $\tcU_{n,\bH \setminus \Sn^+}$
be a Poisson process with intensity measure
$n {\rm Leb}_d  \otimes \mu_Y$ in
$(\bH \setminus \Sn^+) \times \R_+$,
independent of $\cU_{n,\tA_n}$, and  set 
\bea
\tcU_{n,\bH}:= \tcU_{n,\Sn^+} \cup \tcU_{n,\bH \setminus \Sn^+};
~~~~~~~~~~~~~~~\tcU'_{n,\bH}:= \tcU_{n,\bH}
\cap (\R^d \times [0,n^\xxi r_n]).
\label{Uprdef}
\eea
By the Superposition theorem (see e.g.\ \cite[Theorem 3.3]{LP}),
$\tcU_{n,\bH}$ is a homogeneous Poisson process 
in
$\bH \times [0,\infty)$
with intensity measure $n{\rm Leb}_d \times \mu_Y$.  
We call the 
collection of balls of radius $r_nt$
centred on $x$, where $\{(x,t)\}$ are  the points
of this point process, the {\em induced coverage process}.

The next lemma says that 
the region
$\tar_n^+$ is covered with high probability. It is needed because locations
in $\tar_n^+$ lie near the boundary of blocks $S_{n,i}$, so that
coverage of these locations by $Z_n(\cU_{n,A})$ does not necessarily
correspond to coverage of their images in the induced coverage process.

\begin{lemm}
	\label{lemtartan}
	It is the case that
	$ \lim_{t \to \infty} \Pr[ F_n(P^+_n, \cU'_{n,\tA_n}) ] =1.$  
\end{lemm}
\begin{proof}
	We have $\lambda(n) = O((n^{3\xxi} r_n)^{1-d})$,
	and for $1 \leq i \leq \lambda(n)$,
	the number of balls of
	radius $n^\xxi r_n$ required
	to cover the boundary of the $(d-1)$-dimensional
	hypercube $I_{n,i}^+$
	is $O((n^{3\xxi}/n^\xxi)^{d-2})$.
	Thus we can take $x_{n,1}, \ldots, x_{n,k_n} \in \R^d$, with
	$k_n = O(r_n^{1-d}n^{-\xxi -d\xxi} )$,  such that
	$\tar_n \subset \cup_{i=1}^{k_n} B(x_{n,i}, n^\xxi r_n)$.
	
	Then $\tar_n^+ \subset \cup_{i=1}^{k_n}
	B(x_{n,i},10 n^\xxi r_n) \cap A_n^*$,  
	and $B(x_{n,i},10 n^\xxi r_n)$ can
	be covered by $O(n^{d\xxi})$ balls of radius $r_n$.
	Hence by taking $\eps = \xxi/2$ in
	Lemma \ref{lemcov3},
	and using \eqref{e:rn},
	we have
	\begin{align*}
		\Pr [ F_n(P^+_n,  \cU'_{n,\tA_n}) ^c] & \leq
		\sum_{i=1}^{k_n} \Pr[F_n(B(x_{n,i},10 n^\xxi r_n) \cap A_n^*, 
		\cU'_{n,\tA_n} )^c]
		\\
		& =
		O(r_n^{1-d} n^{-\xxi -d\xxi}   n^{d\xxi}  
		n^{\eps - (d-1)/d}) = O(n^{- \xxi/2 }),
	\end{align*}
	which tends to zero.
\end{proof}


\begin{lemm}
	\label{lemblox}
	Suppose 
	$n \geq n_0$ and $i < j \leq \lambda(n)$, $i,j \in \N$. 
	Then
	$(S_{n,i}^+)^o \cap (S_{n,j}^+)^o = \emptyset$.
\end{lemm}
\begin{proof} 
	Suppose $(S_{n,i}^+)^o \cap (S_{n,j}^+)^o \neq \emptyset$; we
	shall obtain a contradiction.
	Let $x \in (S_{n,i}^+)^o \cap (S_{n,j}^+)^o $.
	Let $y$ be the closest point in $I^+_{n,i}$ to $x$, and 
	$y'$ the closest point in $I^+_{n,j}$ to $x$.
	Choose $\ell,m$ such that $I^+_{n,i} \subset H_{n,\ell}$
	and $I^+_{n,j} \subset H_{n,m}$. Then $\ell \neq m$
	since  if $\ell = m$  we would clearly 
	have $(S_{n,i}^+)^o \cap (S_{n,j}^+)^o = \emptyset$.
	
	Let $J_{n,\ell} \subset \R^{d-1}$ be the image of 
	$H_{n,\ell}$ under projection  onto the first $d-1$ coordinates,
	and write $y= (u,\phi_n(u))$ with
	$u \in J_{n,\ell}$.
	Let $v \in \partial J_{n,\ell}$,
	so that $(v,\phi_n(v)) \in \partial_{d-2}H_{n,\ell}$,
	By (\ref{philip}) and (\ref{0126a}), 
	\bean
	|\phi_n(v)-  \phi_n(u) | \leq |\phi(v) -\phi(u)|
	+ 4 K n^{18 \xxi } r_n^2 \leq
	(1/9)\|v-u\| + 4 K n^{18 \xxi} r_n^2.
	\eean
	Since $y \in I^+_{n,i}$ we have $\|y - (v,\phi_n(v))\| 
	\geq \dist(y,\partial H_{n,\ell}) \geq n^{4 \xxi} r_n$, so that 
	\bean
	n^{4 \xxi} r_n 
	\leq \|u-v\| + |\phi_n(u) -
	\phi_n(v)| \leq (10/9)\|u-v\| + 4Kn^{18 \xxi} r_n^2,
	\eean
	and hence 
	provided $n$ is large enough,
	$\|u-v\| \geq  n^{4 \xxi}
	r_n/2 $,
	so that
	writing $y' := (u',\phi_n(u'))$ we have
	$$
	\| y- y' \| \geq \|u- u'\| \geq \dist(u,\partial J_{n,\ell}) \geq
	n^{4\xxi}  r_n/2.
	$$
	But also $\|y-y'\| \leq \|y-x\| + \|y'-x\|  \leq 2 n^{3 \xxi}
	r_n$,
	and we have our contradiction.
\end{proof}

Denote the union of the boundaries (relative to $\R^{d-1} \times \{0\}$)
of the lower faces of
the blocks making up the strip/slab $\Sn$, by $\tQ_n^0$,
and the $(9 n^\xxi r_n)$-neighbourhood in $\bH$ of this region by $\tQ_n$
(the $C$ can be viewed as standing for `corner region'),
i.e.
\bea
\tQ_n^0 
:= \cup_{i=1}^{\lambda(n)} \theta_{n,i} (\partial I^+_{n,i}),
~~~~~~~
\tQ_n := ( \tQ_n^0 \oplus B(o,9 n^\xxi r_n) ) \cap \bH.
\label{tQdef}
\eea
Here $\partial I^+_{n,i}$ denotes the relative boundary of $I^+_{n,i}$
(relative to the face $H_{n,j}$ containing $I^+_{n,i}$).

The next lemma says that the corner region $\tQ_n$
is covered with high probability. It is needed because locations
in $C_n$ lie near the boundaries of the blocks assembled to make
the induced coverage process, so that coverage of these locations
in the induced coverage process
does not necessarily
correspond to coverage of their pre-images in the original
coverage process.

\begin{lemm}
	\label{lemtQ3}
	It is the case that
	$\lim_{n \to \infty} 
	\Pr[ F_n(\tQ_n ,\cU'_{n,\bH}) ] =1$.
\end{lemm}
\begin{proof}
	For each of the
	$(d-1)$ dimensional hypercubes $\theta_{n,i}(S_{n,i}^+), 1 \leq i \leq \lambda(n)$,
	the number of balls of
	radius $n^\xxi r_n$ required
	to cover the boundary is $O((n^{3\xxi}/n^\xxi)^{d-2})$.
	Also $\lambda(n) = O((n^{3\xxi}r_n)^{1-d})$, so
	we can take points $x_{n,1}, \ldots, x_{n, m_n} \in L_n$,
	with $m_n = O(n^{-\xxi -d \xxi} r_n^{1-d})$,
	such that 
	$\tQ_n^0
	\subset \cup_{i=1}^{m_n} B(x_{n,i}, n^\xxi r_n)$.
	Then $\tQ_n \subset \cup_{i=1}^{m_n} B(x_{n,i},10 n^\xxi r_n)$,
	and $B(x_i,10 n^\xxi r_n)$ can be covered by $O(n^{d \xxi})$
	balls of radius $r_n$.
	Hence by (\ref{0816a}) from
	Lemma \ref{lemcov3}, taking $\eps = \xxi/2$, we
	obtain the estimate
	$$
	\Pr [ F_n(\tQ_n , \cU'_{n,\bH}) ^c] =
	O(n^{d\xxi} n^{-\xxi -d \xxi}  r_n^{1-d}  
	n^{\eps   - (d-1)/d}) 
	= O(n^{- \xxi /2}),
	$$
	which tends to zero.
\end{proof}

\begin{lemm}[Limiting coverage probabilities for the induced coverage process]
	\label{leminduced}
	With $c_{d,k,Y}$ given at \eqref{e:cdY},
	\bea
	\lim_{n \to \infty} \Pr[F_n(\Sn, \tcU'_{n,\bH})] 
	= 
	\lim_{n \to \infty} \Pr[F_n(L_n, \tcU'_{n,\bH})] 
	= \exp(- c_{d,k,Y} |\Gamma| e^{- \beta/2} ) .
	\label{e:slabU'}
	\eea
\end{lemm}

\begin{proof}
	The second  equality of (\ref{e:slabU'}) is easily obtained using
	(\ref{0517c2a}) from Lemma \ref{lemhalfd}, together with Lemma
	\ref{l:2PPs}.

	Recall that $L_n = D_n \times \{0\}$.
	Also $\partial L_n \subset \tQ_n^0$, so that
	$(\partial D_n \oplus B_{(d-1)}(o,n^\xxi r_n)) \times [0,2 n^\xxi r_n]
	\subset \tQ_n$,
	and therefore  by (\ref{0517b2a}) from Lemma \ref{lemhalfd}, 
	\bean
	\Pr[( F_n(L_n, \tcU'_{n,\bH}) \setminus F_n(\Sn, \tcU'_{n,\bH}) )
	\cap F_n(\tQ_n, \tcU'_{n,\bH}) ] \to 0.
	\eean
	Therefore using also
	Lemma \ref{lemtQ3} shows that $\Pr[ F_n(L_n, \tcU'_{n,\bH}) \setminus
	F_n(\Sn, \tcU'_{n,\bH}) ]
	\to 0$, and this gives us the rest of \eqref{e:slabU'}.
\end{proof}



\begin{proof}[Proof of Proposition \ref{lemsurf}]
	We shall approximate the event $F_n (\tA^{*}_n  , \cU'_{n,\tA_n})$ by 
	events \linebreak
	$ F_n(L_n, \tcU'_{n,\bH})$ and $F_n(\Sn,\tcU'_{n,\bH})$,  
	and apply Lemma \ref{leminduced}. The sets $Q_n^+$,
	$\tar_n^+$ and
	$\Sn$ and were defined at \eqref{Qplusdef},
	\eqref{e:Tn+} and
	\eqref{e:Srhodef3d} respectively.
	
	Suppose 
	$F_n(Q^+_n \cup \tar_n^+, \cU'_{n,\tA_n}) 
	\setminus
	F_n(\tA_n^*, \cU'_{n,\tA_n}) $
	occurs, and choose  
	$x \in
	\tA_n^{*} \setminus  (Q_n^+ \cup \tar_n^+)
	\setminus Z_n(\cU'_{n,\tA_n})  
	$. 
	Let $y \in \Gamma_n $ with $\|y-x\| = \dist (x,\Gamma_n)$.
	Then $\|y-x\| \leq  2 n^\xxi r_n$,
	and since
	$ x \notin Q_n^+ $, 
	we have $\dist(x, \partial_{d-2} \Gamma_n) \geq 8d n^{4\xxi} r_n$,
	and hence 
	$\dist(y,\partial_{d-2}\Gamma_n) \geq 8dn^{4\xxi}r_n - 2n^\xxi r_n
	\geq 3 n^{4\xxi}r_n$, 
	provided $n$ is large enough.
	Therefore $y$ lies in the interior of the face $H_{n,i}$ for some
	$i$ and $x-y$ is perpendicular to $H_{n,i}$ (if $y \neq x$).
	Also,
	since
	$x \notin \tar_n^+$,
	$\dist(x,\tar_n) \geq 9 n^{\xxi} r_n$, so
	$\dist(y,\tar_n) \geq 7 n^{\xxi}  r_n$.
	Therefore $y \in I_{n,j}$ for some $j$, and
	$x$ lies
	in the block $S_{n,j}$.
	Hence $B(\theta_{n,j}(x),n^\xxi r_n) \cap \bH  \subset
	\theta_{n,j}(S^+_{n,j})$,
	and hence by (\ref{Uprdef}),
	$$
	\#(\{
	(z,s) \in \tcU'_{n,\bH}: \theta_{n,j}(x) \in B(z,r_ns)\}) = 
	\#(\{
	(w,t) \in \cU'_{n,\tA_n}: x \in B(w,r_nt)\}) < k,
	$$
	so the event $F_n(\Sn, \tcU'_{n,\bH}) $ does not occur. Hence
	$$
	F_n(\Sn, \tcU'_{n,\bH}) \setminus F_n(\tA_n^*, \cU'_{n,\tA_n}) \subset
	F_n(Q^+_n   \cup P^+_n, \cU'_{n,\tA_n})^c,
	$$
	so  by Lemmas 
	\ref{lemcorn3}  and  \ref{lemtartan},  
	$\Pr[ F_n(\Sn,\tcU'_{n,\bH}) \setminus 
	F_n (\tA^{*}_n  , \cU'_{n,\tA_n})  ] \to 0$, and hence using
	\eqref{e:slabU'} we have
	\bea
	\liminf_{n \to \infty} \Pr[F_n(\tA_n^*,\cU'_{n,\tA_n})] \geq 
	\exp(-c_{d,k,Y} |\Gamma|e^{-\beta/2}).
	\label{0124d}
	\eea

	Suppose $F_n (\tA^{*}_n , \cU'_{n,\tA_n})  \setminus
	F_n(L_n,\tcU'_{n,\bH})$ occurs,
	and
	choose $ y \in L_n \setminus Z_n(\tcU'_{n,\bH})$.
	Take $i \in \{1,\ldots,\lambda(n)\}$
	such that $y \in \theta_{n,i}(I_{n,i}^+)$.
	Then 
	$ \dist ( y, \theta_{n,i} (\partial I_{n,i})) \leq
	n^\xxi r_n $,
	since otherwise $\theta_{n,i}^{-1}(y)$ would be
	a location in $A^{**}_n \setminus  Z_n( \cU'_{n,\tA_n})$.
	Thus $y \in \tQ_n$  by (\ref{tQdef}), and
	therefore using Lemma \ref{lemtQ3} yields that
	$$
	\Pr [ F_n(\tA^{*}_n, \cU'_{n,\tA_n}) \setminus
	F_n(L_n,\tcU'_{n,\bH}) ] \leq \Pr [ 
	F_n(\tQ_n ,\cU'_{n,\tA_n})^c ] \to 0 .
	$$ 
	Combining this with 
	\eqref{e:slabU'}
	and (\ref{0124d})
	completes the proof.
\end{proof}

\subsection{{\bf Proof of Theorem \ref{th:generaldhi}}}

\label{seclast}

\textcolor{\blue}{Having obtained  the limiting probability
of covering a polytopal region approximating a part of $A$
near the boundary and contained in a single chart, we shall now complete
the proof of Theorem \ref{th:generaldhi} by the following steps. First
we shall show that error from the polytopal approximation vanishes, and then
we put together finitely many regions of $A$, each of which
is contained in a single chart, to get the limiting probability
of covering the whole of $A$.}

Proposition \ref{lemsurf} gives
the limiting probability of coverage of a polytopal 
approximation to a region near part of $\partial A$.
The next two lemmas show that
$\Pr[F_n(\tA^{*}_n, \cU'_{n,\tA_n})]$ approximates
$\Pr[F_n(A^*_n,\cU'_{n,A_n})]$ (recall the definitions at (\ref{Pdtdef})).
From this we can deduce that
we get the same limiting probability 
even after dispensing with the polytopal approximation.

\begin{lemm}
\label{lemdiff4}
Let $E^{(1)}_n:= F_n( \tA_n^*  , \cU'_{n,\tA_n}) \setminus 
F_n(A^*_n,\cU'_{n,A_n})$. 
Then
$\Pr[E^{(1)}_n] \to 0$
as $ n \to \infty$.
\end{lemm}
\begin{proof}
Let $\eps \in (0, (1/(2d)) - 9 \xxi )$.
Suppose $E^{(1)}_n \cap F_n(Q_n^+ , \cU'_{n,A_n})$ occurs.
Then
since $\cU'_{n,\tA_n} \subset \cU'_{n,A_n}$, 
$\R^d \setminus Z_n(\cU'_{n,A_n})$ intersects
with $A^*_n \setminus \tA^{*}_n$, 
and therefore by
(\ref{0126a}), 
includes locations
distant at most  $2K n^{18 \xxi}r_n^2$ from $\Gamma_n$.
Also $\Gamma_n \subset  Z_n (\cU'_{n,A_n}) $,
since $\Gamma_n \subset \tA^{*}_n$.

Pick a location $x \in \overline{A_n^* \setminus Z_n(\cU'_{n,A_n}) } $
of minimal distance from
$\Gamma_n$.  
Then $x \notin Q_n^+$, so
the nearest point in $\Gamma_n$ to $x$ lies in the
interior of $H_{n,i}$ for some $i$.
We claim that $x$ lies at the intersection of the boundaries of
$d$ of the balls making up $Z_n(\cU'_{n,A_n})$;
this is proved similarly to the similar claim concerning
$w$ in the proof of Lemma \ref{lemhalfd}.
Moreover, $x$ lies in at most $k-1$ 
of the other balls making up  $Z_n(\cU'_{n,A_n})$.
Also $x$ does not lie in the interior of $\tA_n^*$.

Thus if $E^{(1)}_n \cap F_n(Q_n^+ , \cU'_{n,A_n})$ occurs, there
must exist $d$ points $(x_1,s_1),(x_2,s_2),\ldots,$ $ (x_d,s_d)$
of $\cU_{n,A_n} $ such that $\cap_{i=1}^d
\partial B(x_i,r_ns_i)$ 
includes a point in $A_n^*$ but  outside the interior of 
$\tA_n^*$, within distance $2Kn^{18 \xxi} r_n^2$ of 
$\Gamma_n$
and in  $B(x,r_n s)$ for at most $k-1$ of the
other points $(x,s)$ of $\cU_{n,A_n}$.
Hence by Markov's inequality and the Mecke formula,
we obtain that
\begin{align}
\Pr[E^{(1)}_n \cap F_n(Q_n^+ , \cU'_{n,A_n})] 
\leq I_{n,1} + I_{n,2}
\label{e:pE1n}
\end{align}
where,
taking $Y_1,\ldots , Y_d$ to
be independent random variables with the distribution of $Y$,
and writing $f_{n,A}(x)$ for $\Pr[x  \in Z_n(\cU'_{n,A})]$ for
all $x \in \R^d$, and recalling the definition of
$q_n(\cdot)$ at \eqref{e:qn},
we set
\begin{align*}
I_{n,1}
: = & n^d \int_{\R^d} \cdots \int_{\R^d}
\E[h_n((x_1,Y_1),(x_1+y_2,Y_2),\ldots,(x_1+y_d,Y_d))
\\
& \times {\bf 1} \{q_n((x_1,Y_1), (x_1+y_2,Y_2),
\ldots,(x_1+y_d,Y_d))
\in A_n^* \cap (\Gamma_n \oplus B(o,2Kn^{18\xxi} r_n^2) )
\}
\\
& \times 
(1-f_{n,A}(q_n( (x_1,Y_1), (x_1+y_2,Y_2),\ldots,(x_1+y_d,Y_d))
))] dy_d \cdots dy_2 dx_1,
\end{align*}
and $I_{n,2}$ is defined similarly with $p_n(\cdot)$ 
replacing $q_n(\cdot)$.
Then changing variables $y_i \mapsto r_n^{-1} y_i$ we have 
\begin{align*}
I_{n,1}
= & n^d r_n^{d(d-1)}
\E \Big[
\int_{\R^d} \cdots \int_{\R^d}
h_n((o,Y_1),(r_n y_2,Y_2),\ldots,(r_n y_d,Y_d))
\\
& \times {\bf 1} \{x_1 + q_n((o,Y_1), (r_n y_2,Y_2),
\ldots,(r_ny_d Y_d))
\in A_n^* \cap (\Gamma_n \oplus B(o,2Kn^{18\xxi} r_n^2))
\}
\\
& \times 
(1-f_{n,A}(x_1 + q_n( (o,Y_1), (r_ny_2,Y_2),\ldots,
(r_ny_d,Y_d))
)) 
dy_d \cdots dy_2 dx_1
\Big].
\end{align*}
By  Lemma \ref{lemcov3}, the last factor of $1-f_{n,A}(\cdot)$
is $O(n^{\eps-1+1/d})$ whenever the indicator function in
the previous factor is 1. Also using Fubini's theorem
we can take the integral over $x_1$ inside all the other
integrals and integrate it out, getting a factor of
$O(n^{18\xxi} r_n^2)$.
Thus we obtain for a
suitable constant $c$ that
\begin{align*}
I_{n,1}
\leq  c n^d r_n^{d(d-1)}
n^{18 \xxi + \eps -1+ 1/d} r_n^2
\E \Big[ \int_{\R^d} \cdots \int_{\R^d}
h((o,r_nY_1),(r_ny_2,r_nY_2),\ldots,(r_ny_d,r_nY_d))
\\
dy_d \cdots dy_2 \Big]
\\
=  c (nr_n^d)^{d-1 + 2/d} n^{18\xxi +\eps-1/d} 
\E \Big[ \int_{\R^d} \cdots \int_{\R^d}
h((o,Y_1),(y_2,Y_2),\ldots,(y_d,Y_d))
dy_d \cdots dy_2 \Big].
\end{align*}
In the last line the expectation is finite by Lemma \ref{l:inth} and
our  moment condition on $Y$. Also $(nr_n^d)^d = O(n^\eps)$ by
\eqref{e:rn}. Thus $I_{n,1}= O(n^{2 \eps + 18 \xxi - 1/d})$ 
so $I_{n,1} \to 0$ as $n\to \infty$,
and by an identical argument the same holds for $I_{n,2}$.
Also
$\Pr[ F_n(Q_n^+ , \cU'_{n,A_n})]  \to 1$ by Lemma \ref{lemcorn3}, so 
using \eqref{e:pE1n} we obtain that
$\Pr[E^{(1)}_n ] \to 0$, as required. 
\end{proof}

\begin{lemm}
\label{lemdiff3}
Let $E^{(2)}_n :=  F_n(A_n^*, \cU'_{n,A_n})  \setminus  F_n(\tA_n^*, \cU'_{n,\tA_n})$.
Then
$\lim_{n \to \infty} \Pr[ E^{(2)}_n]=0$.
\end{lemm}
\begin{proof}
%
If the event $E^{(2)}_n$ occurs, then
since $\tA^{*}_n \subset A_n^*$, 
the set
$ \tA^{*}_n \cap 
Z_n(\cU'_{n,A_n}) \setminus Z_n(\cU'_{n,\tA_n})$ is nonempty.
Hence there exists $(x,s) \in \cU'_{n,A_n} \setminus
\cU'_{n,\tA_n}$ with $B(x,r_n s) \cap
\tA^{*}_n 
\setminus
Z_n(\cU'_{n,\tA_n}) 
\neq 
\emptyset$.
Therefore  
\bea
E^{(2)}_n \subset  F_n(
Z_n(\cU'_{n,A_n} \setminus \cU'_{n,\tA_n})
\cap \tA_n^*
,   \cU'_{n,\tA_n})^c .
\label{0210a}
\eea

Let $\eps \in (0,(1/(2d))- (9 +d/2)\xxi)$.
Let
$ \cQ_n: = \cU'_{n,A_n} \setminus \cU'_{n,\tA_n}$. Then
$\cU'_{n,\tA_n} $ and $\cQ_n$ are independent Poisson processes
with intensity measures $n {\rm Leb}_d \otimes \mu_Y$
in $\tA_n \times [0,n^\xxi]$, $(A_n \setminus \tA_n) \times [0,n^\xxi]$
respectively.
By Lemma \ref{lemcov3} 
and the union bound,
there is a constant $c$ such that for
any $m \in \N$ and any set of $m$ points $(x_1,t_1),\ldots,(x_m,t_m)$
in $\R^d \times [0,n^\xxi]$, we have
\bean
\Pr \left[ F_n( \cup_{i=1}^m
B(x_i,r_n t_i) \cap \tA_n^* ,\cU'_{n,\tA_n} )^c  
\right]  
\leq c m n^{d\xxi} n^{\eps - (d-1)/d} .
\eean
Let $N_n := \cQ_n(\R^d \times [0,n^\xxi])$.
By Lemma \ref{corotaylor}(a), $\E[ N_n] = O(n^{1+ 18 \xxi} r_n^2  )$, so that
by conditioning on $\cQ_n$ we have 
\begin{align*}
\Pr[ F_n( \cup_{(x,t) \in \cQ_n }
B(x,r_nt) \cap  \tA_n^* , \cU'_{n,\tA_n} )^c] 
& \leq c n^{\eps + d \xxi - 1 +1/d} \E[N_n]
\\
& = O( n^{ (18 +d) \xxi  +  2 \eps - 1/d }),
\end{align*}
which tends to zero by the choice of $\eps$.
Hence by (\ref{0210a}), $\Pr[E^{(2)}_n ] \to 0$.
\end{proof}

To complete the proof of Theorem \ref{thm3d},
we shall break $\partial A$ into finitely many pieces, with each piece
contained in a single chart. We would like to write the
probability that all of $\partial A$ is covered as
the product of probabilities for each piece, but to
achieve the independence needed for this, we need to remove a
region near the boundary of each piece. By separate
estimates we can show the removed regions are covered with
high probability, and this is the content of the next lemma.

With $\Gamma$ and $\partial \Gamma$ as in Section 
\ref{secfirststeps},
define the sets
$
\DG_n := \partial \Gamma \oplus B(o, n^{29\xxi}r_n)
$
and
$
\DG_n^{+} := \partial \Gamma \oplus B(o,  n^{49 \xxi}r_n).
$

\begin{lemm}
\label{dBlem}
It is the case that 
$\lim_{n \to \infty}
F_n( \DG_n^{+} \cap A, \cU'_ {n,A}) =1$. 
\end{lemm}
\begin{proof}
Let $\eps \in (0,(1/d) - 98\xxi  )$.
Since we assume  $\kappa(\partial \Gamma,r)= O(r^{2-d})$
as $r \downarrow 0$,
for each $n $
we can take $x_{n,1},\ldots,x_{n,k(n)} \in \R^d$ with
$\partial \Gamma \subset \cup_{i=1}^{k(n)}
B(x_{n,i},n^{49\xxi}r_n)$, and
with $k(n) = O((n^{49 \xxi} r_n)^{2-d})$.
Then $\DG_n^{+} \subset \cup_{i=1}^{k(n)} B(x_{n,i},2n^{49\xxi}r_n)$.
For each $i \in \{1,\ldots, k(n)\}$, we can cover
the ball 
$B(x_{n,i},2 n^{49 \xxi}r_n)$ 
with $O(n^{49d\xxi})$ smaller balls of radius $r_n$.
Then we end up with  balls of
radius $r_n$, denoted $B_{n,1},\ldots,B_{n,m(n)}$ say,
such that $\DG_n^{+} \subset \cup_{i=1}^{m(n)} B_{n,i}$ and
$m(n) = 
O(r_n^{2-d} n^{49\xxi(3-d)}) =
O(r_n^{2-d} n^{98\xxi})$.
By (\ref{0816c}) from  Lemma
\ref{lemcov3}, and the union bound, 
\bean
\Pr[ \cup_{i=1}^{m(n)} ( F_{n}(B_{n,i} \cap A, \cU'_{n,A})^c)] 
= O( r_n^{2-d} n^{98 \xxi + \eps - 1 + 1/d})
= O(  n^{   98 \xxi + \eps - 1/d} ),
\eean
which tends to zero.
\end{proof}

Given $n >0$, define the sets
$
\Gamma^{(n^{29 \xxi}r_n)}:= \Gamma \setminus \DG_n
$ and
$$
\Gamma^{(n^{29 \xxi}r_n)}_{n^\xxi r_n} := (\Gamma^{(n^{29 \xxi}r_n)} 
\oplus B(o, n^\xxi  r_n))
\cap A;
~~~~~
\Gamma_{n^\xxi r_n} := (\Gamma \oplus B(o,n^\xxi r_n)) \cap A,
$$
and define the event $F_n^\Gamma:=
F_n(\Gamma^{(n^{29 \xxi}r_n)}_{n^\xxi r_n} , \cU'_{n,A} )$.

Note that the definition of $F_n^\Gamma$ does not depend on the 
choice of chart. This 
fact
will be needed for the last stage
of the proof of Theorem \ref{thm3d}.
Lemma \ref{lemcap} below shows that
$\Pr[F_n^\Gamma]$ is well approximated by $\Pr[F_n(A_n^*, \cU'_{n,A_n})]$ and
we have already determined the limiting behaviour of the latter.
We prepare for the proof of Lemma \ref{lemcap} with two
geometrical lemmas.

\begin{lemm}
\label{lemcapA}
For all large enough $n$,
it is the case that $\Gamma_{n^\xxi r_n}^{(n^{29 \xxi}r_n)}  \subset 
A_n^*$.
\end{lemm}
\begin{proof}
Let $x \in \Gamma^{(n^{29\xxi}r_n)}_{n^\xxi r_n} $, and
take $y \in \Gamma^{(n^{29\xxi}r_n)}$ with
$\|x-y\| \leq n^\xxi r_n$.  
Writing $y = (u,\phi(u))$ with
$u \in U$,  we claim that
$\dist(u, \partial U) \geq (1/2)n^{29 \xxi}r_n$. Indeed, if 
we had	$\dist(u, \partial U) < (1/2)n^{29 \xxi}r_n$, then we could
take $w \in \partial U$ with
$\|u -w \|  < (1/2)n^{29 \xxi}r_n$.
Then $(w,\phi(w)) \in \partial \Gamma$ and 
by (\ref{philip}), $|\phi(w) - \phi(u)| \leq (1/4) n^{29 \xxi}r_n$, so 
$$
\|(u,\phi(u)) - (w,\phi(w)) \| \leq \|u-w\| + |\phi(u)-\phi(w)|
\leq (3/4)n^{29 \xxi}r_n,
$$
contradicting the assumption that 
$y \in \Gamma^{(n^{29 \xxi}r_n)}$, so the claim is justified.

Writing $x = (v,s)$ with $v \in \R^{d-1}$,
and $s \in \R$, we have $\|v-u\| \leq \|x-y\| \leq n^\xxi  r_n$, so 
$\dist(v, \partial U) \geq (1/2)n^{29\xxi}r_n - n^\xxi r_n$, and hence
$v \in U_n^-$, provided $n$ is big enough
($U_n^-$ was defined shortly after (\ref{eqgamma}).)   Also 
$|\phi(v) - \phi(u)| \leq  n^\xxi r_n/4$ by (\ref{philip}),
so  $|\phi_n(v) - \phi(u)| \leq n^\xxi r_n/2$, provided $n$ is
big enough, by (\ref{0126a}).
Also
$|s - \phi(u)| \leq \|x-y\| \leq  n^\xxi r_n$, so 
$
|s - \phi_n(v)| \leq (3/2) n^\xxi r_n.
$
Therefore $x \in A^*_n$ by (\ref{Pdtdef}). 
\end{proof}
\begin{lemm}
\label{lemBA}
For all large enough $n$, we have
(a) $[A_n^* \oplus B(o,4 n^\xxi r_n)] \cap A \subset A_n$, 
and (b) 
$[ A_n^* \oplus B(o,4 n^\xxi r_n)]  
\cap \partial A \subset  \Gamma$,
and (c)
$[ \Gamma_{n^\xxi r_n}^{(n^{29\xxi}r_n)} \oplus B(o,4 n^\xxi r_n)]  
\cap \partial A \subset  \Gamma$.
\end{lemm}
\begin{proof}
Let
$x \in A_n^*$.
Write  $x = (u,z)$ with $u \in U_n^-$ and $\phi_n(u) - 3 n^\xxi
r_n/2 \leq z \leq \phi(u)$.

Let $y \in B(x,4 n^\xxi r_n) \cap A$,
and write $y = (v,s)$ with $v \in \R^{d-1}$
and $s \in \R$. Then $\|v-u\| \leq 4 n^\xxi r_n$ so provided $n$ is
big enough, $v \in U_n$. Also $|s-z| \leq 4 n^\xxi r_n$, and
$|\phi(v) - \phi(u)| \leq  n^\xxi r_n$
by (\ref{philip}), so 
$$
|s - \phi(v)| \leq |s - z| + |z - \phi(u)| + |\phi(u)- \phi(v)|
\leq 4 n^\xxi r_n + 2 n^\xxi r_n +  n^\xxi r_n,
$$
and since $y \in A$, by (\ref{0901b}) and (\ref{0901a}) 
we must have $ 0 \leq s \leq \phi(v) $, provided $n$ is big enough.
Therefore $y = (v,s) \in A_n$, which gives us (a).

If also $y \in \partial A$, then $\phi(v)=s$, so 
$y \in \Gamma$. Hence 
we have part (b). 
Then by Lemma \ref{lemcapA} we also have part (c).
\end{proof}
\begin{lemm}
\label{lemcap}
It is the case that 
$\Pr[F_n^\Gamma \triangle F_n(A_n^*, \cU'_{n,A_n})] \to 0$
as $n \to \infty$.
\end{lemm}
\begin{proof}
Since
$\Gamma_{n^\xxi r_n}^{(n^{29\xxi}r_n)}  \subset A_n^*$
by Lemma \ref{lemcapA},
and moreover $\cU'_{n,A_n} \subset \cU'_{n,A}$, it  follows that
$F_n(A_n^*,\cU'_{n,A_n}) \subset
F_n(\Gamma^{(n^{29\xxi}r_n)}_{n^\xxi r_n} , \cU'_{n,A}) = F_n^\Gamma$. 
Therefore it suffices to prove that
\bea
\Pr[F_n(\Gamma^{(n^{29\xi}r_n)}_{n^\xxi r_n} , \cU'_{n,A})
\setminus
F_n(A_n^*,\cU'_{n,A_n}) ] \to 0.  
\label{0830a}
\eea

Let $\eps >0$. Suppose event
$F_n(\Gamma^{(n^{29\xi}r_n)}_{n^\xxi r_n} , \cU'_{n,A}) \cap 
F_n( \DG_n^{+} \cap A, \cU'_{n,A}) \setminus F_n(A_n^*,\cU'_{n,A_n})$
occurs. Choose $x \in A_n^* \setminus Z_n(\cU'_{n,A_n})$.
Then 
by Lemma \ref{lemBA}(a), $B(x,n^\xxi r_n) \cap A \subset A_n $.
Hence $\cU'_{n,A} \cap ( B(x, n^\xxi r_n)
\times \R_+) \subset \cU'_{n,A_n}$, and therefore
$x \notin Z_n(\cU'_{n,A})$. 

Since we are assuming $
F_n(\Gamma^{(n^{29\xxi}r_n)}_{n^\xxi r_n} , \cU'_{n,A}) $ occurs,
we therefore have $\dist(x, \Gamma^{(n^{29 \xxi}r_n)} ) > n^\xxi r_n $. 
Since we also assume 
$F_n(\DG_n^{+} \cap A,\cU'_{n,A})$, we also have $\dist(x, \partial \Gamma) \geq 
n^{49 \xxi}r_n$ and therefore 
$\dist (x,\DG_n) =
\dist (x, \partial \Gamma \oplus B(o,n^{29\xxi}r_n))
\geq n^{29\xxi } r_n$.
Hence 
$$
\dist (x,\Gamma) \geq \min(\dist(x, \Gamma^{(n^{29\xxi }r_n)}),
\dist (x, \partial \Gamma \oplus B(o,n^{29 \xxi }r_n)) 
> n^\xxi r_n.
$$
Moreover, by Lemma \ref{lemBA}(b), $\dist (x, (\partial A)
\setminus \Gamma)
>  n^\xxi r_n$. Thus $\dist(x, \partial A) > n^\xxi r_n$. 
Moreover, $\dist(x,\partial A) \leq \dist (x,\Gamma) \leq 2 n^\xxi 
r_n $
because $x \in A_n^*$,
and therefore $x \notin A^{[\eps]}$ 
(provided $n$ is large enough)
since $\overline{A^{[\eps]}}$ is compact and contained in $A^o$
(the set $A^{[\eps]}$ was defined in Section \ref{secdefs}.)
Therefore the event
$F_n(A^{( n^\xxi r_n)} \setminus A^{[\eps]},\cU'_{n,A})^c$ occurs.
Thus, for large enough $n$ we have the event inclusion
\begin{align}
F_n(\Gamma^{(n^{29\xxi}r_n)}_{r_n} , \cU'_{n,A}) \cap 
F_n( \DG_n^{+} \cap A, \cU'_{n,A}) \setminus F_n(A_n^*,\cU'_{n,A_n})
\subset
F_n( A^{(n^\xxi r_n)} \setminus A^{[\eps]} ,\cU'_{n,A})^c .
\label{0901d}
\end{align}

By  (\ref{e:rn}), 
\bea
\lim_{n \to \infty} ( \omega_d \E[Y^d] n r_n^{d} -  \log n 
- (d+k-2)
\lglg n )
=  \begin{cases}
\beta & {\rm if} ~ d=2
,k=1
\\
+\infty  & {\rm otherwise.}
\end{cases}
\label{0901c}
\eea
Hence by Lemma  \ref{lemHall} 
applied to the set $A \setminus A^{[\eps]}$,
and Lemma \ref{l:2PPs}
we have that
\begin{align}
\liminf_{n \to \infty} \Pr[
F_n( A^{(n^\xxi r_n)} \setminus A^{[\eps]} ,\cU'_{n,A}) ]
=
\liminf_{n \to \infty} \Pr[
F_n( A^{(n^\xxi r_n)} \setminus A^{[\eps]} ,\cU'_{n,\R^d}) ]
\nonumber
\\
\geq
\liminf_{n \to \infty} \Pr[
F_n( 
A
\setminus A^{[\eps]} ,\cU_{n,\R^d}) ]
\nonumber
\\
= \begin{cases} 
	\exp \Big(- 
	c_d \Big( \frac{(\E[Y^{d-1}])^d}{ (\E[Y^d])^{d-1}} \Big)
	|A \setminus A^{[\eps]} |e^{- \beta} \Big) 
	& {\rm if~} d=2
	, k=1
	\\
	1 & {\rm otherwise.}
\end{cases}
\label{1230a}
\end{align}
Therefore since $\eps$ can be arbitrarily small
and $|A \setminus A^{[\eps]} | \to 0$
as $\eps \downarrow 0$,
the event displayed on the left hand side of
(\ref{0901d}) has probability
tending to zero. Then using
Lemma \ref{dBlem},
we have   \eqref{0830a}, which completes the proof.
\end{proof}

\begin{coro}
\label{coroFB}
It is the case that 
$
\lim_{n \to \infty} 	\Pr[F_n^\Gamma]
=
\exp (-  c_{d,k,Y} |\Gamma| e^{-  \beta/2} ). 
$
\end{coro}
\begin{proof}
By Lemmas \ref{lemdiff4} and \ref{lemdiff3},
$\Pr[ 
F_n(\tA^{*}_n, \cU'_{n,\tA_n})
\triangle
F_n(A_n^*, \cU'_{n,A_n}) 
] \to 0$.
Then by Lemma \ref{lemcap}, 
$\Pr[ F_n^\Gamma \triangle F_n(W^{*}_n, \cU'_{n,\tA_n}) ] \to 0$,
and now the result follows by Proposition \ref{lemsurf}.
\end{proof}

\begin{proof}[Proof of Theorem \ref{th:generaldhi}]
Let $x_1,\ldots,x_J $ and $r(x_1),\ldots,r(x_J)$
be as described at (\ref{bycompactness}).
Set $\Gamma_1 := B(x_1,r(x_1) ) \cap \partial A$, and for
$j =2,\ldots, J$, let 
$$ 
\Gamma_j := \overline{
B(x_j,r(x_j) ) \cap \partial A \setminus \cup_{i=1}^{j-1}
B(x_i,r(x_i))}, 
$$
and $\partial \Gamma_i := \Gamma_i \cap \overline{\partial A \setminus \Gamma_i}$.
Then  $\Gamma_1,\ldots,\Gamma_J$
comprise a finite collection of closed
sets in $\partial A$
with disjoint interiors, each of which
satisfies $\kappa(\partial \Gamma_i,r) =
O(r^{2-d})$ as $r \downarrow 0$, and is
contained in a single chart $B(x_j,r(x_j))$, and
with union $\partial A$.
For $1 \leq i \leq J$,
define $F_n^{\Gamma_i}$ analogously to $F_n^\Gamma$, 
that is,
$F_n^{\Gamma_i} := F_n(\Gamma_{i,n^{\xxi} r_n}^{(n^{29\xxi}r_n)}
, \cU'_{n,A})$
with
$$
\Gamma_{i,n^\xxi r_n}^{(n^{29 \xxi}r_n)} := 
\left( \left[ \Gamma_i \setminus ((\partial \Gamma_i) 
\oplus B(o,n^{29 \xxi }r_n ))
\right] \oplus B(o,  n^\xxi r_n) \right) \cap A.
$$
First we claim that
the following event inclusion holds:
\bean
\cap_{i=1}^J F_n^{\Gamma_i}
\cap
F_n(A^{(n^\xxi r_n)}, \cU'_{n,A}) 
\setminus
F_n(A,\cU'_{n,A}) 
\subset
\left(\cap_{i=1}^J F_n([(\partial \Gamma_i) \oplus B(o,n^{49\xxi}r_n)]
\cap A, \cU'_{n,A}) \right)^c.
\eean
Indeed, suppose 
$
\cap_{i=1}^J F_n^{\Gamma_i} 
\cap
F_n(A^{(n^\xxi r_n)}, \cU'_{n,A})
\setminus
F_n(A,\cU'_{n,A})
$ occurs, and choose 
$x \in A \setminus Z_n(\cU'_{n,A})$.
Then $\dist(x,\partial A) \leq  n^\xxi r_n$ since
we assume $F_n(A^{(n^\xxi r_n)},\cU'_{n,A})$ occurs.
Then for some $i \in \{1,\ldots,J\}$
and some $y \in \Gamma_i$ we have $\|x-y \| \leq  n^\xxi r_n$.
Since we assume $F_n^{\Gamma_i}$ occurs,
we have $x \notin \Gamma_{i,n^\xxi r_n}^{(n^{29\xxi}r_n)}$, and hence  
$\dist(y, \partial \Gamma_i) \leq n^{29\xxi}r_n$,
so 
$\dist(x, \partial \Gamma_i) < n^{49 \xxi}r_n$.
Therefore $F_n([(\partial \Gamma_i) \oplus B(o, n^{49\xxi}r_n )]
\cap A, \cU'_{n,A})$ fails to occur, justifying the claim.

By  the preceding claim and and the union bound,
\begin{align}
\Pr[F_n(A,\cU'_{n,A})
] & \leq \Pr[ \cap_{i=1}^J F_n^{\Gamma_i}
\cap F_n(A^{(n^\xxi r_n)},\cU'_{n,A}) ] 
\nonumber	\\
& \leq \Pr[F_n(A,\cU'_{n,A}) ]  
+ \sum_{i=1}^J
\Pr[ F_n([(\partial \Gamma_i) \oplus B(o,n^{49\xxi}r_n)]
\cap A, \cU'_{n,A})^c].
\nonumber
\end{align}
By  Lemma \ref{dBlem},
$\Pr[ F_n([(\partial \Gamma_i) \oplus B(o,n^{49\xxi}r_n)]
\cap A, \cU'_{n,A})]
\to 1 $ for each $i$.
Therefore 
\bea
\lim_{n \to \infty} \Pr[F_{n} (A,\cU'_{n,A}) ] = 
\lim_{n \to \infty} \Pr[
\cap_{i=1}^J F_{n}^{\Gamma_i}
\cap F_n(A^{(n^\xxi r_n)}, \cU'_{n,A})
],
\label{0901e}
\eea
provided the last limit exists.
By Corollary \ref{coroFB},
we have for each $i$ that
\bea
\lim_{t \to \infty}
(\Pr[ F^{\Gamma_i}_n]) 
= \exp(- c_{d,k,Y}|\Gamma_i|  e^{-\beta/2} ).
\label{0901f}
\eea
Also, we claim that for large enough $n$ 
the events $F_n^{\Gamma_1}$, \ldots, $F_n^{\Gamma_J}$ are
mutually independent.
Indeed, given distinct $i,j \in \{1,\ldots,J\}$,
if $x \in \Gamma_{i,n^\xxi r_n}^{(n^{29\xxi}r_n)}$ and
$y \in \Gamma_{j,n^\xxi r_n}^{(n^{29\xxi}r_n)}$, then
we can take
$y' \in \Gamma_j \setminus ( \partial \Gamma_j  \oplus B(o,n^{29\xxi}r_n))$
with $\|y'-y \| \leq n^\xxi r_n$. 
If $\| x -y \| \leq 3 n^\xxi r_n$ then by the triangle
inequality $\|x - y'\| \leq 4 n^\xxi r_n$, but since $y' \notin \Gamma_i$,
this would contradict Lemma \ref{lemBA}(c). Therefore
$\| x - y\| > 3 n^\xxi r_n$, and hence the $n^\xxi r_n$-neighbourhoods
of $
\Gamma_{i,n^\xxi r_n}^{(n^{29\xxi}r_n)}$ and of
$ \Gamma_{j,n^\xxi r_n}^{(n^{29\xxi}r_n)}$ are disjoint. This gives
us the independence claimed.

Now observe that $F_n(A^{(n^\xxi r_n)},\cU'_{n,A}) \subset
F_n(A^{(4n^\xxi r_n)},\cU'_{n,A})$.
We claim that 
\bea
\Pr[ F_n(A^{(4n^\xxi r_n)},\cU'_{n,A}) \setminus F_n(A^{(n^\xxi r_n)},
\cU'_{n,A})] \to 0 ~~~ {\rm as} ~ n \to \infty.
\label{0920a}
\eea
Indeed, given $\eps >0$, for large $n $ 
the probability on the left side of (\ref{0920a})
is bounded by
$\Pr[ F_n(A^{(n^\xxi r_n)}\setminus A^{[\eps]},\cU'_{n,A})^c]  $, and by
(\ref{1230a}) the limsup of the latter
probability 
can be made arbitrarily
small by the choice of $\eps$.
Hence by Lemma \ref{lemHall} 
and (\ref{0901c}),
Lemma \ref{l:2PPs} and the fact that $c_2=1$,
\begin{align}
\lim_{n \to \infty} \Pr[ F_n(A^{(4n^\xxi r_n)},\cU'_{n,A}) ]
& =
\lim_{n \to \infty} \Pr[ F_n(A^{(n^\xxi r_n)},\cU'_{n,A})] 
\nonumber \\
& =
\lim_{n \to \infty} \Pr[ F_n(A^{(n^\xxi r_n)},\cU'_{n,\R^d})] 
\nonumber \\
& =  \begin{cases}
	\exp \Big( - \Big( \frac{(\E[Y])^2}{
		\E[Y^2]} \Big) |A| e^{- \beta}  \Big) & {\rm if} ~ d=2,
	k=1
	\\
	1 & {\rm otherwise} .
\end{cases}
\label{0920c}
\end{align}
Moreover, by (\ref{0901e}) and (\ref{0920a}),
\bea
\lim_{n \to \infty} 
\Pr[ F_n( A , \cU'_{n,A})
] =
\lim_{n \to \infty} \Pr[
\cap_{i=1}^J F_n^{\Gamma_i}
\cap F_n(A^{(4n^\xxi r_n)}, \cU'_{n,A})
],
\label{0920b}
\eea
provided the last limit exists.
However, the events in the right hand side of (\ref{0920b})
are mutually independent, so using (\ref{0901f}), 
\eqref{e:cdY}
and
(\ref{0920c}),
we obtain that
\begin{align*}
\lim_{n \to \infty} 
\Pr[ F_{n}(A,\cU'_{n,A})] 
=
\exp\Big( - c_{d,k} 
\Big(\frac{(\E[Y^{d-1}])^{d-1}}{(\E[Y^{d}])^{d-2+1/d}} \Big)
|\partial A| e^{-\beta/2}
\\
- \Big( \frac{\E[Y]^2}{\E[Y^2]} \Big) |A| e^{-\beta} {\bf 1}_{\{(d,k)=(2,1)\}}\Big).
\end{align*}
By Lemma \ref{l:2PPs}, if we replace $\cU'_{n,A}$ with
$\cU_{n,A}$ on the left, we get the same limit, i.e.
\eqref{0128b} if $d \geq 3$ and \eqref{0128a} if $d=2$.
\end{proof}

\begin{appendices}

\section{Confirming consistency with Chiu's result}\label{a:consist}

In this appendix we verify that our Proposition  \ref{Hallthm}
is consistent with \cite[Theorem 4]{Chiu95}.
The latter result is concerned with a quantity denoted
$T_L$ in \cite{Chiu95} which is the same as our $\tilde{\tau}_L$,
assuming we take $v= \gamma =1$ in \cite{Chiu95}. Chiu 
takes $A = [0,1]^d$ and defines
$c = c(L) = d(d+1) \log L  - \log \omega_d$. Then 
in our notation,  the result in \cite{Chiu95} says that 
for any  $u \in \R$,
as $L \to \infty$ we have
\begin{align}
	\Pr \left[ 
	c^{d/(d+1)} \omega_d^{1/(d+1)} \tilde{\tau}_L - c - 
	\log
	\left(c^{1/(d+1)}\left( \frac{c+ \log c}{d+1}  \right)^{d-1} 
	\right)
	\leq u
	\right] \to F(u),
	\label{e:Chiu14}
\end{align}
with $F(u) = \exp(-c_d (d+1)^{d-1}d^{-d} e^{-u})$
(recall that our $c_d$ is the same as Chiu's $\psi_d$
so $c$ and $c_d$ are two different things here).

To verify that we can recover \eqref{e:Chiu14} from
Proposition \ref{Hallthm}, note first that $c \to \infty$
as $L \to \infty$ and
$ \log c = \log \log L + \log (d(d+1)) + o(1)$.
Also
\begin{align*}
	\log
	\left(c^{1/(d+1)}\left( \frac{c+ \log c}{d+1}  \right)^{d-1} 
	\right)
	= \log \left( c^{d^2/(d+1)} (d+1)^{1-d} \right)
	+ o(1).
\end{align*}
Hence for some $z_L$ which tends to zero as $L \to \infty$,
the 
left hand side of \eqref{e:Chiu14} equals
\begin{align*}
	\Pr  \left[
	c^d \omega_d \tilde{\tau}_L^{d+1} 
	\leq
	\left(
	c+ \log (c^{d^2/(d+1)} (d+1)^{1-d}
	)
	+ u + z_L \right)^{d+1} \right].
\end{align*}
By binomial expansion,
for some $y_L$ which tends to zero as $L \to \infty$,
this probability equals
\begin{align*}
	&	 \Pr  \left[
	c^d \omega_d \tilde{\tau}_L^{d+1} 
	\leq
	c^{d+1}+ (d+1) c^d
	\left(
	\log (c^{d^2/(d+1)} (d+1)^{1-d}
	)
	+ u + z_L \right) 
	+ y_L c^d \right]
	\\
	= & \Pr  \left[
	\omega_d \tilde{\tau}_L^{d+1} 
	\leq
	c+ (d+1) 
	\left(
	\log (c^{d^2/(d+1)} (d+1)^{1-d})
	+ u + z_L \right) 
	+ y_L \right]
	\\
	=  & \Pr  \left[
	\omega_d \tilde{\tau}_L^{d+1} 
	\leq d(d+1) \log L - \log \omega_d
	+ d^2 \log c + \log ((d+1)^{1-d^2} )
	+ (d+1)u + z'_L \right],
\end{align*}
where we set $z'_L := (d+1)z_L + y_L = o(1)$ as $L \to \infty$.
For some $z''_L$ tending to zero, this probability equals
\begin{align*}
	\Pr  \left[
	\omega_d \tilde{\tau}_L^{d+1} 
	\leq d(d+1) \log L +
	d^2 \log \log L 
	+ \log \left( 
	\frac{d^{d^2}(d+1)^{d^2}}{
		(d+1)^{d^2-1} \omega_d}
	\right)
	+ (d+1) u + z''_L \right],
	\\
	= \Pr  \left[
	\omega_d \tilde{\tau}_L^{d+1} 
	- d(d+1) \log L -
	d^2 \log \log L 
	\leq 
	\log \left( 
	d^{d^2}(d+1)/
	\omega_d
	\right)
	+ (d+1) u + z''_L \right].
\end{align*}
By \eqref{e:ttaulim} from Proposition \ref{Hallthm},
and the continuity of the limiting cumulative distribution
function in that result,
this probability tends to
\begin{align*}
	& \exp \left( - c_d  \left(\frac{(d+1)^{d^2}}{d^d
		\omega_d} \right)^{1/(d+1)} e^{-u}
	(d^{d^2}(d+1))^{-1/(d+1)} \omega_d^{1/(d+1)} \right)
	\\
	& = \exp \left(- c_d d^{-d} (d+1)^{d-1} e^{-u}
	\right) = F(u).
\end{align*}
In other words, we have derived \eqref{e:Chiu14} from
Proposition \ref{Hallthm}.
\end{appendices}

{\bf Acknowledgement.}
We thank Edward Crane for a useful conversation on the link between
the SPBM and the Johnson-Mehl model.
We also thank two anonymous referees for some
helpful comments. 

{\bf Data availability statement.}
The code used to generate Figure~\ref{f:tessellations} and Figure~\ref{f:Tri} and the videos discussed in Section~\ref{secdefs}, as well as all the code and data used to create Figure~\ref{f:sim}, are all available at \url{https://github.com/frankiehiggs/johnson-mehl}.


\bibliographystyle{plainurl}
\bibliography{jm-references}

\begin{thebibliography}{10}

\bibitem{Alm98}
S.E. Alm.
\newblock Approximation and simulation of the distributions of scan statistics
  for {Poisson} processes in higher dimensions.
\newblock {\em Extremes}, 1(1):111--126, January 1998.
\newblock \href {http://dx.doi.org/10.1023/a:1009965918058}
  {\path{doi:10.1023/a:1009965918058}}.

\bibitem{CC14}
P.~Calka and N.~Chenavier.
\newblock Extreme values for characteristic radii of a {Poisson-Voronoi}
  tessellation.
\newblock {\em Extremes}, 17(3):359--385, May 2014.
\newblock \href {http://dx.doi.org/10.1007/s10687-014-0184-y}
  {\path{doi:10.1007/s10687-014-0184-y}}.

\bibitem{Chiu95}
S.~N. Chiu.
\newblock Limit theorems for the time of completion of {Johnson-Mehl}
  tessellations.
\newblock {\em Advances in Applied Probability}, 27(4):889--910, December 1995.
\newblock \href {http://dx.doi.org/10.2307/1427927}
  {\path{doi:10.2307/1427927}}.

\bibitem{CSKM}
S.N. Chiu, D.~Stoyan, W.S. Kendall, and J.~Mecke.
\newblock {\em Stochastic Geometry and its Applications}.
\newblock Wiley, Chichester, August 2013.
\newblock \href {http://dx.doi.org/10.1002/9781118658222}
  {\path{doi:10.1002/9781118658222}}.

\bibitem{Gusakova}
A.~Gusakova, Z.~Kabluchko, and C.~Th\"{a}le.
\newblock The $\beta$-{Delaunay} tessellation: Description of the model and
  geometry of typical cells.
\newblock {\em Advances in Applied Probability}, 54(4):1252--1290, August 2022.
\newblock \href {http://dx.doi.org/10.1017/apr.2022.6}
  {\path{doi:10.1017/apr.2022.6}}.

\bibitem{HallZW}
P.~Hall.
\newblock Distribution of size, structure and number of vacant regions in a
  high-intensity mosaic.
\newblock {\em Zeitschrift für Wahrscheinlichkeitstheorie und verwandte
  Gebiete}, 70(2):237--261, August 1985.
\newblock \href {http://dx.doi.org/10.1007/bf02451430}
  {\path{doi:10.1007/bf02451430}}.

\bibitem{CovXY}
F.~Higgs, M.D. Penrose, and X.~Yang.
\newblock Covering one point process with another.
\newblock {\em Methodology and Computing in Applied Probability}, 27(2), April
  2025.
\newblock \href {http://dx.doi.org/10.1007/s11009-025-10165-7}
  {\path{doi:10.1007/s11009-025-10165-7}}.

\bibitem{Janson}
S.~Janson.
\newblock Random coverings in several dimensions.
\newblock {\em Acta Mathematica}, 156(0):83--118, 1986.
\newblock \href {http://dx.doi.org/10.1007/bf02399201}
  {\path{doi:10.1007/bf02399201}}.

\bibitem{LP}
G.~Last and M.~Penrose.
\newblock {\em Lectures on the {Poisson} Process}.
\newblock Cambridge University Press, Cambridge, October 2017.
\newblock \href {http://dx.doi.org/10.1017/9781316104477}
  {\path{doi:10.1017/9781316104477}}.

\bibitem{Moller}
J.~M{\o}ller.
\newblock Random {Johnson-Mehl} tessellations.
\newblock {\em Advances in Applied Probability}, 24(4):814--844, December 1992.
\newblock \href {http://dx.doi.org/10.2307/1427714}
  {\path{doi:10.2307/1427714}}.

\bibitem{OBSC}
A.~Okabe, B.~Boots, K.~Sugihara, and S.N. Chiu.
\newblock {\em Spatial Tessellations}.
\newblock John Wiley \& Sons, Inc., Chichester, July 2000.
\newblock \href {http://dx.doi.org/10.1002/9780470317013}
  {\path{doi:10.1002/9780470317013}}.

\bibitem{OT23}
M.~Otto and C.~Th\"{a}le.
\newblock Large nearest neighbour balls in hyperbolic stochastic geometry.
\newblock {\em Extremes}, 26(3):413--431, April 2023.
\newblock \href {http://dx.doi.org/10.1007/s10687-023-00470-0}
  {\path{doi:10.1007/s10687-023-00470-0}}.

\bibitem{P03}
M.~Penrose.
\newblock {\em Random Geometric Graphs}.
\newblock Oxford University Press, Oxford, 2003.
\newblock URL: \url{https://doi.org/10.1093/acprof:oso/9780198506263.001.0001},
  \href {http://dx.doi.org/10.1093/acprof:oso/9780198506263.001.0001}
  {\path{doi:10.1093/acprof:oso/9780198506263.001.0001}}.

\bibitem{P02}
M.D. Penrose.
\newblock Focusing of the scan statistic and geometric clique number.
\newblock {\em Advances in Applied Probability}, 34(4):739--753, December 2002.
\newblock \href {http://dx.doi.org/10.1239/aap/1037990951}
  {\path{doi:10.1239/aap/1037990951}}.

\bibitem{P23}
M.D. Penrose.
\newblock Random {Euclidean} coverage from within.
\newblock {\em Probability Theory and Related Fields}, 185(3–4):747--814,
  January 2023.
\newblock \href {http://dx.doi.org/10.1007/s00440-022-01182-5}
  {\path{doi:10.1007/s00440-022-01182-5}}.

\bibitem{PY24}
M.D. Penrose and X.~Yang.
\newblock Fluctuations of the connectivity threshold and largest
  nearest-neighbour link.
\newblock {\em The Annals of Applied Probability, to appear.}
\newblock URL: \url{https://arxiv.org/abs/2406.00647}.

\bibitem{Pryce}
J.D. Pryce.
\newblock {\em Basic Methods of Linear Functional Analysis}.
\newblock Hutchinson, London, 1973.

\bibitem{Rossi}
F.~Rossi, M.~Fiorentino, and P.~Versace.
\newblock Two‐component extreme value distribution for flood frequency
  analysis.
\newblock {\em Water Resources Research}, 20(7):847--856, July 1984.
\newblock \href {http://dx.doi.org/10.1029/wr020i007p00847}
  {\path{doi:10.1029/wr020i007p00847}}.

\bibitem{Schreiber}
T.~Schreiber and J.~E. Yukich.
\newblock Variance asymptotics and central limit theorems for generalized
  growth processes with applications to convex hulls and maximal points.
\newblock {\em The Annals of Probability}, 36(1), January 2008.
\newblock \href {http://dx.doi.org/10.1214/009117907000000259}
  {\path{doi:10.1214/009117907000000259}}.

\end{thebibliography}

\end{document}